\documentclass[aop]{imsart}

\RequirePackage{amsthm,amsmath,amsfonts,amssymb}
\RequirePackage[numbers]{natbib}
\usepackage{subcaption, graphicx}
\usepackage{paralist}
\usepackage[normalem]{ulem}
\RequirePackage[colorlinks,citecolor=blue,urlcolor=blue]{hyperref}
\RequirePackage{graphicx}
\usepackage{tcolorbox}
\usepackage{bbm}

\startlocaldefs
\numberwithin{equation}{section}

\providecommand{\R}{}
\providecommand{\Z}{}
\providecommand{\N}{}

\renewcommand{\R}{\mathbb{R}}
\renewcommand{\Z}{\mathbb{Z}}
\renewcommand{\N}{{\mathbb N}}

\newcommand{\one}{\mathbbm{1}}

\newcommand{\E}[1]{{\mathbf E}\left[#1\right]}	
\providecommand{\deg}{}
\renewcommand{\deg}{d}
\newcommand{\e}{{\mathbf E}}

\newcommand{\p}{\mathbf P}

\newcommand{\I}[1]{{\mathbf 1}_{[#1]}}
\newcommand{\set}[1]{\left( #1 \right)}
 
\newcommand{\probC}[2]{\mathbf{P}\set{#1 \; \left|  \; #2 \right. }}

\newcommand\cA{\mathcal A}

\newcommand\cC{\mathcal C}
\newcommand\cD{\mathcal D}

\newcommand\cM{\mathcal M}

\newcommand\cS{{\mathcal S}}
\newcommand\cT{{\mathcal T}}

\newcommand{\rA}{\mathrm{A}} 
\newcommand{\rB}{\mathrm{B}} 
\newcommand{\rC}{\mathrm{C}}

\newcommand{\eqdist}{\ensuremath{\stackrel{\mathrm{d}}{=}}}
\newcommand{\convdist}{\ensuremath{\stackrel{\mathrm{d}}{\rightarrow}}}

\newcommand{\pran}[1]{\left(#1\right)}
\providecommand{\eps}{}
\renewcommand{\eps}{\epsilon}

\newcommand{\sequence}[1]{\mathrm{#1}}
\newcommand{\tree}{\mathrm{t}}
\newcommand{\weight}{\mathrm{w}}
\newcommand{\dist}{\ensuremath{\mathrm{dist}}}

\newcommand{\height}{\mathrm{ht}}
\newcommand{\randomtree}{\mathrm{T}}

\theoremstyle{plain}
\newtheorem{thm}{Theorem}
\newtheorem*{thm*}{Theorem}
\newtheorem{lem}[thm]{Lemma}
\newtheorem{prop}[thm]{Proposition}
\newtheorem{cor}[thm]{Corollary}

\newtheorem{claim}[thm]{Claim}

\graphicspath{ {./images/} }
\newcommand{\pnt}{\mathrm{par}}
\newcommand{\dseq}{\sequence{d}}

\endlocaldefs

\begin{document}

\begin{frontmatter}
\title{Random trees have height $O(\sqrt{n})$}
\runtitle{Random trees have height $O(\sqrt{n})$}

\begin{aug}
\author[A]{\fnms{Louigi}~\snm{Addario-Berry}\ead[label=e1]{louigi.addario@mcgill.ca}},
\author[B]{\fnms{Serte}~\snm{Donderwinkel}\ead[label=e2]{s.a.donderwinkel@rug.nl}}
\address[A]{Department of Mathematics and Statistics, McGill University, Montr\'eal, Canada\printead[presep={,\ }]{e1}}

\address[B]{Bernoulli Institute for Mathematics, Computer Science and AI, and CogniGron (Groningen Cognitive Systems and Materials Center), University of Groningen, Groningen, The Netherlands\printead[presep={,\ }]{e2}}
\end{aug}



\begin{abstract} 
We obtain new non-asymptotic tail bounds for the height of uniformly random trees with a given degree sequence, simply generated trees and conditioned Bienaym\'e trees (the family trees of branching processes), in the process settling three conjectures of Janson \cite{janson_simpygen} and answering several other questions from the literature.
 Moreover, we define a partial ordering on degree sequences and show that it induces a stochastic ordering on the heights of uniformly random trees with given degree sequences. 
The latter result can also be used to show that sub-binary random trees are stochastically the tallest trees with a given number of vertices and leaves (and thus that random binary trees are the stochastically tallest random homeomorphically irreducible trees \cite{MR101846} with a given number of vertices).

Our proofs are based in part on the Foata--Fuchs bijection between trees and sequences \cite{FoataFuchs}, which can be recast to provide a line-breaking construction of random trees with given vertex degrees \cite{us}.
\end{abstract}

\end{frontmatter}
\maketitle

\section{\bf Introduction}

This paper proves optimal asymptotic and non-asymptotic tail bounds on the heights of random trees from several natural and well-studied classes, in the process proving conjectures and answering questions from \cite{janson_simpygen,MR3077536,addarioberryfattrees,McDiarmidScott}. We elaborate on connections with previous literature in Section~\ref{sec:relatedwork}, after the presentation of our results. In Section \ref{sec:conclusion}, we discuss directions for future research opened by the current work.

We begin by stating our results for the {\em Bienaym\'e} trees, which constitute the most elementary and well-studied random graph model and the oldest stochastic model for studying population growth. For $\mu=(\mu_k,k \ge 0)$ a probability distribution on $\N=\Z\cap[0,\infty)$, a Bienaym\'e tree with offspring distribution $\mu$, denoted $\randomtree_\mu$, is the family tree of a branching process with offspring distribution $\mu$.\footnote{Bienaymé trees are often referred to as Galton--Watson trees, but we adopt the change in terminology suggested in \cite{addarioberry2021universal}.} (An example appears in Figure~\ref{fig:bienayme}.) We write $|\randomtree_\mu|$ for the size (number of vertices) of $\randomtree_\mu$. For $n \in \N$ such that $\p(|\randomtree_\mu|=n)>0$, we write $\randomtree_{\mu,n}$ to denote a Bienaym\'e tree with offspring distribution $\mu$ conditioned to have size $n$. Bienaym\'e trees are random {\em plane} trees: rooted, unlabeled trees in which the set of children of each node is endowed with a (left-to-right) total order. 
\begin{figure}[hbt]
	\begin{centering}
		\includegraphics[scale=1]{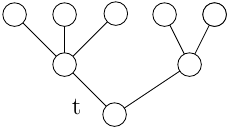}
		\caption{In the tree $\tree$, there are 5 vertices with no children, two vertices with two children and one vertex with 3 children, so for $\mu=(\mu_k,k\geq 0)$ a probability distribution on $\N$, we have that $\p(\randomtree_\mu=\tree)=\mu_0^5\mu_2^2\mu_3$.
		}
		\label{fig:bienayme}
	\end{centering}
\end{figure}

For $p>0$, write $|\mu|_p=\left(\sum_{k=0}^\infty k^p\mu_k\right)^{1/p}$.
\begin{thm}\label{thm.bieninfinitevariance}
	Fix a probability distribution $\mu$ supported on $\N$ with $|\mu|_1\leq 1$ and $|\mu|_2=\infty$. Then, $n^{-1/2}\height(\randomtree_{\mu,n})\to0$ as $n\to\infty$ along values $n$ such that $\p(|\randomtree_\mu|=n)>0$,  both in probability and in expectation.
\end{thm}
Theorem \ref{thm.bieninfinitevariance} answers an open question from \cite{MR3077536}. 

\begin{thm}\label{thm.biensubcriticalnoexpdecay}
	Fix a probability distribution $\mu=(\mu_k,k\geq 0)$ supported on $\N$ with $|\mu|_1<1$ and with $\sum_{k\geq 0} e^{tk}\mu_k=\infty$ for all $t>0$. Then, $n^{-1/2}\height(\randomtree_{\mu,n})\to0$ as $n\to\infty$ over all $n$ such that $\p(|\randomtree_\mu|=n)>0$,  both in probability and in expectation.
\end{thm}
Theorem~\ref{thm.biensubcriticalnoexpdecay} solves a slight variant of Problem~21.7 from~\cite{janson_simpygen}. That problem is stated for  simply generated trees, which are a slight generalization of Bienaym\'e trees. We in fact solve the problem from~\cite{janson_simpygen} in full; see Theorem~\ref{thm.simplygen}, below. 

A collection $(X_i)_{i \in I}$ of random variables is {\em sub-Gaussian} if there exist constants $c,C>0$ such that $\p(X_i \ge x) \le C\exp(-cx^2)$ for all $i \in I$ and all $x>0$. Our third theorem states that for any offspring distribution $\mu=(\mu_k,k\geq 0)$, the family of random variables ($n^{-1/2}\height(T_{\mu,n})$) has sub-Gaussian tails, with constants $c,C$ that only depend on $\mu_0$ and $\mu_1$.
\begin{thm}\label{thm.biengaussiantails}
	Fix $\epsilon \in [0,1)$. Then, there exist constants $c=c(\epsilon)$ and $C=C(\epsilon)$ such that the following holds. For any probability distribution $\mu=(\mu_k,k\geq 0)$ on $\N$ with $\mu_0+\mu_1<1-\epsilon$ and $\mu_0/(\mu_0+\mu_1)>\epsilon$, for all $n$ sufficiently large and for which $\p(|\randomtree_\mu|=n)>0$, for all $x > 0$,
	\[\p\left( \height(\randomtree_{\mu,n})\ge xn^{1/2}\right)< C\exp(-cx^2).\]
\end{thm}
This theorem has the following immediate corollary. 

\begin{cor}\label{cor.expectedheightbien}
For any probability distribution $\mu=(\mu_k,k\geq 0)$ on $\N$ with $\mu_0+\mu_1<1$, $\e[\height(\randomtree_{\mu,n})]=O(n^{1/2})$, and, more generally, for any fixed $r<\infty$, $\e[\height(\randomtree_{\mu,n})^r]=O(n^{r/2})$ as $n\to\infty$ over all $n$ such that $\p(|\randomtree_\mu|=n)>0$. \qed 
\end{cor}

Theorem \ref{thm.biengaussiantails} and Corollary \ref{cor.expectedheightbien} strengthen Theorem 1.2 and Corollary 1.3 in \cite{MR3077536}, as they are not restricted to critical offspring distributions with finite variance and the bounds only depend on $\mu$ via $\mu_0$ and $\mu_1$. This resolves conjectures stated in  \cite{addarioberryfattrees} and \cite{MR3077536}. Again, this theorem also holds for the larger class of simply-generated trees, stated as Theorem \ref{thm.simplygengaussian} below, solving Problem 21.9 from \cite{janson_simpygen}.  Theorem \ref{thm.simplygengaussian} also resolves a conjecture by McDiarmid and Scott on the \emph{block-diameter} of a family of random graphs called \emph{block-graphs} \cite{McDiarmidScott}; see Theorem \ref{thm:blockgraphs}. (In fact, our result is stronger than their conjecture; McDiarmid and Scott are interested in tightness of the diameter under rescaling, while our results also provide Gaussian tail bounds.) We provide more details in Section \ref{sec:blockgraphs}.

\noindent $\star$ The requirement in Theorem~\ref{thm.biengaussiantails} that $n$ is sufficiently large (where ``sufficiently large'' depends on $\mu$) is necessary. To see this, note that if $\mu$ has support $\{0,1,N\}$ then for any $n < N+1$, with probability one $\randomtree_{\mu,n}$ is a path (so has height $n-1$). 
The requirement in Theorem~\ref{thm.biengaussiantails} that $1-\mu_0-\mu_1$ and $\mu_0/(\mu_0+\mu_1)$ be bounded from below is also necessary. To see this, it suffices to consider probability distributions $\mu$ of the form 
\[
\mu_0 = q(1-p)\qquad
\mu_1 = (1-q)(1-p)\qquad
\mu_2 = p\, ,
\]
with $p,q\in (0,1)$. 
For any $x > 1$, it is possible to make $\liminf_{n \to \infty} \p(\height(\randomtree_{\mu,n})> xn^{1/2})$ arbitrarily close to one by 
either taking $p$ fixed and $q$ sufficiently small or $q$ fixed and $p$ sufficiently small; this fact is proved as Claim \ref{claim:smalldegreesyieldtalltrees}, below. 

For random variables $X,Y$, we write $X \preceq_{\mathrm{st}} Y$ to mean that $Y$  stochastically dominates $X$, which is to say that $\p(X \le t) \ge \p(Y \le t)$ for all $t \in \R$.  We also write $X \prec_{\mathrm{st}} Y$, and say that $Y$ strictly stochastically dominates $X$, if $X \preceq_{\mathrm{st}} Y$ and $X$ and $Y$ are not identically distributed.
Our final result for Bienaym\'e trees states that, among all conditioned Bienaym\'e trees which almost surely have no vertices with exactly $1$ child, the binary Bienaym\'e trees have the stochastically largest heights. \begin{thm}\label{thm:bienstoch}
Fix any probability distribution $\mu=(\mu_k,k\geq 0)$ on $\N$ with $\mu_0 \in (0,1)$ and $\mu_1=0$, and let $\nu$ be the probability distribution on $\N$ with $\nu(0)=\nu(2)=1/2$. Then for all $n \ge 1$ such that $\p(|\randomtree_{\mu}|=n) > 0$, 
\begin{align*}
\height(\randomtree_{\mu,n}) & \preceq_{\mathrm{st}} \height(\randomtree_{\nu,n})\mbox{ if  $n$ is odd}\, \\
\height(\randomtree_{\mu,n}) & \preceq_{\mathrm{st}} \height(\randomtree_{\nu,n+1})\mbox{ if  $n$ is even.}\, 
\end{align*}
\end{thm}

The probability measure $\nu$ in the preceding theorem could be replaced by any nondegenerate probability measure $\nu'$ with $\nu'(0)+\nu'(2)=1$, as for any such measure and any $m\in \N$, the tree $\randomtree_{\nu',2m+1}$ is uniformly distributed over rooted binary plane trees with $m$ internal nodes and $m+1$ leaves. The parity requirement is due to the fact that binary trees always have an odd number of vertices. Also, the requirement that $\mu_1=0$ is necessary, since if $\mu_1>0$, then  with positive probability $\randomtree_{\mu,2m}$ (resp. $\randomtree_{\mu,2m+1}$) is a path of length $2m-1$ (resp. $2m$), whereas $\randomtree_{\nu,2m+1}$ has height at most $m$.

The preceding theorem has the following consequence for tail bounds on the height of conditioned Bienaym\'e trees. Fix any constants $c$ and $C$ such that the inequality in Theorem \ref{thm.biengaussiantails} holds for uniform binary trees $\randomtree_{\nu,n}$. Then, the same inequality also holds for $\randomtree_{\mu,n}$ for any probability distributions $\mu$ with $\mu(1)=0$. In particular, by the remark above, we see that if we restrict the theorem to such probability distributions, the constants need not depend on $\epsilon$. 

\subsection{Trees with fixed degree sequences}
The above results for Bienaym\'e trees are consequences of more refined results, presented in this subsection, on the heights of random trees with fixed degree sequences. 

For a rooted tree $\tree$ and a vertex $v$ of $\tree$, the {\em degree} $\deg_\tree(v)$ of $v$ is the number of children of $v$ in $\tree$. A {\em leaf} of $\tree$ is a vertex of $\tree$ with degree zero. 

A {\em degree sequence} is a sequence of non-negative integers $\dseq=(d_1,\ldots,d_n)$ for which $\sum_{i \in [n]} d_i = n-1$. We write $\cT_{\dseq}$ for the set of finite rooted labeled trees $\tree$ with vertex set labeled by $[n]$ and such that for each $i \in [n]$, the vertex with label $i$ has degree $d_i$. 
Also write $\cT(n)$ for the set of all finite rooted labeled trees with vertex set labeled by $[n]$. For $\tree \in \cT(n)$, we write $d_\tree(i)$ for the degree of the vertex with label $i$ and  say that $(d_\tree(1),\ldots,d_\tree(n))$ is the degree sequence of~$\tree$. 

For $p>0$ and a degree sequence $\dseq=(d_1,\dots,d_n)$ write $|\dseq|_p=\left(\sum_{i=1}^n d_i^p\right)^{1/p}$ and let $(\sigma_\dseq)^2=\frac{1}{n}\sum_{i=1}^n d_i(d_i-1)$. Also, for $i \ge 0$ write $n_i(\dseq) = |\{j \in [n]: d_j=i\}|$ for the number of entries of $\dseq$ which equal $i$.

Given a finite set $S$, we write $X \in_u S$ to mean that $X$ is a uniformly random element of~$S$. 
\begin{thm}\label{thm.gaussiantails}
Fix any degree sequence $\dseq=(d_1,\ldots,d_n)$, let $\randomtree \in_u \cT_{\dseq}$, and write $\delta=(n-n_1(\dseq))/n$. Then for all $x > 0$, 
		\[\p\left(\height(\randomtree)\ge xn^{1/2}\right)<5\exp(- \delta x^2/2^{12}).\]
\end{thm}	

		\begin{thm}\label{thm.infinitevariance}
		Fix any degree sequence $\dseq=(d_1,\ldots,d_n)$, let $\randomtree\in_u \cT_{\dseq}$, and define 
$(\sigma')^2=\frac{n}{n-n_1(\dseq)}\sigma_{\dseq}^2$. Then for all $x\geq 2^{15}$, 
		\[\p\left(\height(\randomtree) \ge  xn^{1/2} \frac{\log(\sigma'+1)}{\sigma_{\dseq}} \right)< 4\exp\left(-x\log(\sigma'+1)/2^{14}\right).\]
	\end{thm}
\noindent {\bf Aside.} A different tail bound for the heights of random trees with given degrees was recently proved by Arthur Blanc-Renaudie~\cite{blanc}, in the course of proving scaling limit results for such trees (as well as for other models of trees, including inhomogeneous continuum random trees). In the language of our paper, Blanc-Renaudie's tail bound reads as follows. Given a degree sequence $\dseq=(d_1,\ldots,d_n)$, for 
$x \ge 1$ let $\mu_\dseq(x) = \sum_{i=1}^n (d_i-1)_+(1-(n-d_i)^{\lfloor x\rfloor}/(n-1)_{\lfloor x \rfloor})$, where $(m)_k := m(m-1)\cdot \ldots \cdot (m-k+1)$, and let $t_{\dseq}=\inf\{x\ge 1: 2\mu_{\dseq}(x) \ge n_0(\dseq)-1\}$. 
	\begin{thm*}[\cite{blanc}, Theorem 8(a)]\label{abr} There exist constants $c,C>0$ such that the following holds. Fix any degree sequence $\dseq=(d_1,\ldots,d_n)$ and let $\randomtree \in_u \cT_{\dseq}$. Then for all $x > 0$, 
	\begin{align*}
	& \p\!\left(
	\frac{c\sigma_{\dseq}}{\sqrt{n}}\height(T)
	\ge
	x
	+\frac{\sqrt{n}}{\sigma_{\dseq}}\int_{\frac{\sqrt{n}}{\sigma_{\dseq}}}^{t_{\dseq}} \frac{dy}{y\mu_\dseq(y)}
	\!+\sigma_{\dseq}\sqrt{n}\frac{\ln(n\!-n_0(\dseq)\!-n_1(\dseq))}{n_0(\dseq)-1}
	\!+\!\frac{\sigma_{\dseq}}{\sqrt{n}}\frac{\ln(n_0(\dseq)-1)}{\ln(n/n_1(\dseq))}
	\!\right)\\
	& \le
	C\exp\left(-\frac{cx}{\sigma_{\dseq}\sqrt{n}}\mu_d(x\sqrt{n}/\sigma_{\dseq})\right)\, .
	\end{align*}
	\end{thm*}
Although this bound may appear somewhat unwieldy, and the limit of integration $t_{\dseq}$ is ``hard to compute'' \cite{blanc}, its bound in fact appears to be quite powerful; in particular, the other results of \cite{blanc} demonstrate that in many circumstances it predicts the correct order of magnitude for the height.

\medskip
The following proposition is the tool which allows us to transfer results from the setting of random trees with given degree sequences to that of Bienaym\'e trees (and that of simply generated trees, which are introduced in Section~\ref{sec.bienandsimpygen}, below). 
\begin{prop}\label{prop:planetrees}
Fix non-negative weights $(w_i,i \ge 0)$. 
Let $\randomtree$ be a random plane tree of size $n$ with distribution given by 
\begin{equation}\label{eq:twn_ident}
\p(\randomtree=\tree) \propto \prod_{v \in v(\tree)} w_{\deg_\tree(v)},
\end{equation}
for plane trees $\tree$ of size $n$. 
Conditionally given $\randomtree$, let $\hat{\randomtree}$ be the random tree in $\cT(n)$ obtained as follows: 
label the vertices of $\randomtree$ by a uniformly random permutation of $[n]$, then forget about the plane structure. 
For $i \in [n]$ let $D_i$ be the degree of vertex $i$ in $\hat{\randomtree}$ of vertex $i$. Then for any degree sequence $\dseq=(d_1,\ldots,d_n)$, conditionally given that $(D_1,\ldots,D_n)=(d_1,\ldots,d_n)$, we have $\hat{\randomtree} \in_u \cT_\dseq$. 
\end{prop}
We provide the proof of Proposition~\ref{prop:planetrees} right away since it is short, but it is not necessary to read the proof in order to understand what follows.
\begin{proof}[Proof of Proposition~\ref{prop:planetrees}]
In this proof, by an {\em ordered labeled tree} we mean a 
finite tree $\mathbf{t}$ whose vertices are labeled by the integers $[n]$ (for some $n$) and such that the children of each node are endowed with a left-to-right order. From an ordered labeled tree, we may obtain a plane tree by ignoring the vertex labels, and we may obtain a labeled tree lying in $\cT(n)$ by ignoring the orderings. 

Let $\mathbf{T}$ be the random ordered labeled tree obtained from $\randomtree$ by labeling the vertices of $\randomtree$ by a uniformly random permutation of $[n]$. 

Now fix any labeled tree $\hat{\tree}$ with vertices labeled by $[n]$. For $j \ge 0$ write $n_j = \{i \in [n]: \deg_{\hat{\tree}}(i)=j\}$.
Then 
\[
\p(\hat{\randomtree}=\hat{\tree})
=\sum_{\mathbf{t}} \p(\mathbf{T}=\mathbf{t}), 
\]
where the sum is over ordered labeled trees $\mathbf{t}$ with vertices labeled by $[n]$ whose underlying (unordered) labeled tree is $\hat{\randomtree}$. This sum has 
\[
\prod_{v \in \hat{\tree}} \deg_{\hat{\tree}}(v)!
=\prod_{j \ge 0} (j!)^{n_j} 
\] 
summands, corresponding to the number of ways to choose an ordering of the children of each vertex of $\hat{\tree}$. 
Moreover, for any tree $\mathbf{t}$ included in the sum, writing $\tree$ for the (unlabeled) plane tree underlying $\mathbf{t}$, we have 
\[
\p(\mathbf{T}=\mathbf{t})=\frac{1}{n!} 
\p(\randomtree=\tree) \propto 
\prod_{v \in v(\tree)} w_{\deg_\tree(v)}
= \prod_{j \ge 0} w_j^{n_j}\, .
\]
This value is the same for all summands, so it follows that 
\[
\p(\hat{\randomtree}=\hat{\tree}) \propto 
\prod_{j \ge 0} (j!)^{n_j}\prod_{j \ge 0} w_j^{n_j}\, . 
\]
The right-hand side depends on $\hat{\tree}$ only through the number of vertices of each degree. In particular, for any degree sequence $\dseq$, it is constant over trees in $\cT_\dseq$. The result follows. 
\end{proof}
Using this proposition, we can transfer results from the setting of random trees with given degree sequences to that of random plane trees, provided that we can obtain sufficiently precise control of the degree statistics of the random plane trees. (By ``degree statistics'' we mean the number of nodes of given degrees.) Once such control is established, Theorems~\ref{thm.bieninfinitevariance},~\ref{thm.biensubcriticalnoexpdecay}, and~\ref{thm.biengaussiantails} follow straightforwardly from Theorems~\ref{thm.gaussiantails} and~\ref{thm.infinitevariance}. 

Similarly, the stochastic inequality of Theorem~\ref{thm:bienstoch} follows from a finer stochastic ordering on the heights of random trees with given degree sequences. We in fact obtain stochastic domination results for the heights of trees under two different partial orders on degree sequences.
We define both partial orders by  specifying a covering relation\footnote{
	For a partially ordered set $(\mathcal{P},\prec)$, $y \in\mathcal{P}$ covers $x \in \mathcal{P}$ if $x \prec y$ and for all $z \in \mathcal{P}$, if $x \preceq z \preceq y$ then $z \in \{x,y\}$. 
} on equivalence classes (under relabelling) of such degree sequences. Formally, we say degree sequences $\dseq=(d_1,\dots,d_n)$ and $\dseq'=(d'_1,\dots,d'_n)$ are equivalent if there exists a permutation $\eta:[n]\to [n]$ such that $(d_1,\dots,d_n)=(d'_{\eta(1)},\dots,d'_{\eta(n)})$. If this holds, then ``relabelling vertices according to $\eta$'' is a bijection from $\cT_{\dseq}$ to $\cT_{\dseq'}$. This bijection preserves tree heights, so for $\randomtree_{\dseq}\in_u\cT_{\dseq}$ and $\randomtree_{\dseq'}\in_u\cT_{\dseq'}$ it holds that $\height(\randomtree_{\dseq})\overset{d}{=}\height(\randomtree_{\dseq'})$.
\begin{thm}\label{thm.stochasticorder2}
	Let $\prec_{1}$ be the partial order on degree sequences of length $n$ defined by the following covering relation on equivalence classes: for $\cD$ and $\cD'$ equivalence classes of degree sequences of length $n$,  say $\cD$ covers $\cD'$ if there exist $\dseq=(d_1,\dots,d_n)\in \cD$ and  $\sequence{d'}=(d'_1,\dots,d'_n)\in \cD'$ such that 
	\begin{enumerate}
		\item $d_1\geq d_2$,
		\item $d'_1=d_1+1$,
		\item $d'_2=d_2-1$, and
		\item $d'_k=d_k$ for $3\le k \le n$. 
	\end{enumerate}
	 Then with $\randomtree_{\dseq} \in_u \cT_{\dseq}$ and $\randomtree_{\dseq'}\in_u \cT_{\dseq'}$, we have 
	\[
	\sequence{d'} \prec_{1} \dseq \implies \height(\randomtree_{\sequence{d'}})\preceq_{\mathrm{st}} \height(\randomtree_{\dseq}).
	\]
	Moreover, the stochastic domination of $\height(\randomtree_{\dseq})$ over $ \height(\randomtree_{\sequence{d'}})$ is strict if $\sequence{d'} \prec_{1} \dseq$ and $\dseq$ contains at least 3 non-leaf vertices.
\end{thm} 
In words, to obtain $\sequence{d'}$ from $\dseq$ in the definition of $\prec_{1}$, for $a\leq b$, we replace one vertex with $a$ children and one vertex with $b$ children by a vertex with $a-1$ children and one with $b+1$ children; informally, the degrees in $\sequence{d'} $ are more skewed than the degrees in $\dseq$. Then, $\sequence{d'} \prec_{1} \dseq$ and $\height(\randomtree_{\sequence{d'}})\preceq_{\mathrm{st}} \height(\randomtree_{\dseq})$. So, informally, Theorem \ref{thm.stochasticorder2} states that more skewed degrees yield shorter trees. Theorem \ref{thm.stochasticorder2} has the following corollary, which resolves a conjecture from \cite{us}. 
\begin{cor}\label{thm.stochasticorder1}
	Let $\prec_{2}$ be the partial order on degree sequences of length $n$ defined by the following covering
	relation: for $\cD$ and $\cD'$ equivalence classes of degree sequences of length $n$ ,  say $\cD$ covers $\cD'$ if there exists $\dseq=(d_1,\dots,d_n)\in \cD$ and $\sequence{d'}=(d'_1,\dots,d'_n)\in \cD'$ such that 
	\begin{enumerate}
		\item $d_1,d_2\geq 1$,
		\item $d'_{1}=d_i+d_j$,
		\item $d'_{2}=0$, and
		\item $d'_k=d_k$ for $3 \le k\le n$. 
	\end{enumerate}
	Then with $\randomtree_{\dseq} \in_u \cT_{\dseq}$ and $\randomtree_{\dseq'}\in_u \cT_{\dseq'}$, we have  
$	\sequence{d'} \prec_{2} \dseq \implies \height(\randomtree_{\sequence{d'}})\preceq_{\mathrm{st}} \height(\randomtree_{\dseq}).$
\end{cor} 

Corollary~\ref{thm.stochasticorder1} follows immediately from Theorem~\ref{thm.stochasticorder2} by observing that if $\sequence{d'} \prec_{2} \dseq$ then $\sequence{d'} \prec_{1} \dseq$. 
Arthur Blanc-Renaudie told us a proof of the stochastic domination presented in Corollary~\ref{thm.stochasticorder1} (which he does not intend to write up), using a bijection presented in \cite{blanc}. Our proof of Theorem~\ref{thm.stochasticorder2} proceeds differently, but was inspired by Blanc-Renaudie's approach.

Before concluding the introduction, we state and prove two further corollaries of Theorem~\ref{thm.stochasticorder2}, and use the second of them to prove Theorem~\ref{thm:bienstoch}. The first states that height of any random tree with a fixed degree sequence is stochastically dominated by that of a random sub-binary tree with the same number of leaves and degree-one vertices. (We say $\dseq$ is sub-binary if $n_i(\dseq)=0$ for all $i \ge 3$, and in this case also say that $\randomtree_{\dseq}$ is sub-binary.) 
The second essentially says that among all random trees with a given size and given number of vertices of degree $1$, the sub-binary trees have the stochastically largest heights (with a very minor {\em caveat} that essentially addresses a potential parity issue).  
\begin{cor}\label{cor:sub_binary}
Let $\mathrm{b}$ be a sub-binary degree sequence, 
let $\dseq$ be a degree sequence with $n_0(\dseq)=n_0(\sequence{b})$ and  
$n_1(\dseq)=n_1(\sequence{b})$, and let $\randomtree_{\dseq} \in_u\cT_{\dseq}$ and $\randomtree_{\sequence{b}}\in_u\cT_{\sequence{b}}$. Then $\height(\randomtree_{\dseq})\preceq_{\mathrm{st}} \height(\randomtree_{\mathrm{b}})$. 
\end{cor}
\begin{proof} 
Note that necessarily $n_2(\sequence{b})=n_0(\sequence{b})-1$, so $\sequence{b}$ has length $2n_0(\sequence{b})+n_1(\sequence{b})-1$
Write $\dseq=(d_1,\ldots,d_n)$. If $n = 2n_0(\sequence{b})+n_1(\sequence{b})-1$ then $\dseq$ is a permutation of $\sequence{b}$ and there is nothing to prove. So suppose that $n< 2n_0(\sequence{b})+n_1(\sequence{b})-1$; in this case 
necessarily $\dseq$ contains at least one entry which is three or greater. We construct a degree sequence $\dseq'=(d_1',\ldots,d_{n+1}')$ with 
$n_0(\dseq)=n_0(\dseq')$ and $n_1(\dseq)=n_1(\dseq')$ and $\height(T_{\dseq}) \preceq_{\mathrm{st}} \height(T_{\dseq'})$. This proves the corollary (by induction on $2n_0(\sequence{b})+n_1(\sequence{b})-1-n$), as we can then transform $\dseq$ into a degree sequence of length sub-binary degree sequence $2n_0(\sequence{b})+n_1(\sequence{b})-1$ without changing the number of zeros or of ones, and while stochastically increasing the height of the associated random tree.

To construct $\dseq'$, first let $\dseq^+=(d_1^+,\ldots,d_{n+1}^+)=(d_1,\ldots,d_n,1)$. 
The trees $\randomtree_{\dseq}$ and $\randomtree_{\dseq^+}\in_u\cT_{\dseq^+}$ may be coupled as follows. 
If $n+1$ is not the root of $\randomtree_{\dseq^+}$ then form $\randomtree_{\dseq}$ from $\randomtree_{\dseq^+}$ by replacing the two-edge path containing $n+1$ in $\randomtree_{\dseq^+}$ by a single edge connecting its endpoints. If $n+1$ {\em is} the root, then instead form $\randomtree_{\sequence{b}}$ by deleting $n+1$ and rerooting at its unique child. (The reverse operation is to add $n+1$ as the parent of a uniformly random vertex of $\randomtree_{\dseq}$.) 
Under this coupling, the height of $\randomtree_{\dseq^+}$ is at least that of $\randomtree_{\dseq}$, so it follows that 
\[
\height(\randomtree_{\dseq}) \preceq_{\mathrm{st}} \height(\randomtree_{\dseq^+})\, .
\]
Now choose $k \in [n]$ such that $d_k^+=d_k \ge 3$ and define $\dseq' = (d_1',\ldots,d_{n+1}')$ by 
\[
d_i' = \begin{cases}
		d_i^+ & \mbox{ if }i \not\in\{k,n+1\}\\
		d_k^+-1& \mbox{ if } i=k\\
		2	& \mbox{ if }i=n+1\, .
		\end{cases}
\]
Then $n_0(\dseq)=n_0(\dseq')$ and $n_1(\dseq)=n_1(\dseq')$. Moreover, $\height(\randomtree_{\dseq^+}) \preceq_{\mathrm{st}} \height(\randomtree_{\dseq'})$ by Theorem~\ref{thm.stochasticorder2}, and so $\height(\randomtree_{\dseq}) \preceq_{\mathrm{st}} \height(\randomtree_{\dseq'}$) as required. \end{proof}

\begin{cor}\label{cor:unary_stochast}
Let $\dseq=(d_1,\ldots,d_n)$ be any degree sequence and let $\randomtree_{\dseq}\in_u \cT_{\dseq}$. Let 
\[
n^+ = 	\begin{cases}
			n & \mbox{ if }n_1(\dseq) \ge 1 \mbox{ or if $n$ is odd} \\
			n+1	& \mbox{ if }n_1(\dseq)=0 \mbox{ and $n$ is even.}
		\end{cases}
\]
Then there is a sub-binary degree sequence
$\sequence{b}=(b_1,\ldots,b_{n^+})$ with $n_1(\sequence{b})\le n_1(\sequence{d})$, such that with $\randomtree_{\sequence{b}}\in_u \cT_{\sequence{b}}$, then $\height(\randomtree_{\dseq}) \preceq_{\mathrm{st}} \height(\randomtree_{\sequence{b}})$.
\end{cor}
\begin{proof}
Suppose that $\dseq$ has at least one entry $d_k \ge 4$. 
Choose $j \in [n]$ with $d_j=0$. Then the degree sequence 
$\dseq'=(d_1',\ldots,d_n')$ given by 
\[
d_i' = \begin{cases}
		d_i-2	&\mbox{ if }i=k\\
		2		&\mbox{ if }i=j\\
		d_i		&\mbox{ otherwise}\, 
		\end{cases}
\]
has the same length as $\dseq$ and satisfies  $n_1(\dseq') = n_1(\dseq)$. Moreover, by Theorem~\ref{thm.stochasticorder2}, $\height(\randomtree_{\dseq})  \preceq_{\mathrm{st}}\height(\randomtree_{\dseq'})$. 
We have reduced the total degree of vertices with degree at least $4$ while stochastically increasing the height, and without changing $n$ or the number of vertices of degree $1$. It follows that to prove the corollary we may restrict our attention to degree sequences corresponding to trees with maximum degree three.

Next suppose that $\dseq$ has $n_3(\dseq) \ge 2$ and choose distinct $k,\ell \in [n]$ with $d_k=d_\ell=3$ and $j \in [n]$ with $d_j=0$. Then the sequence $\dseq'=(d_1,\ldots,d_n')$ given by 
\[
d_i' = \begin{cases}
		d_i-1	&\mbox{ if }i\in\{k,\ell\}\\
		2		&\mbox{ if }i=j\\
		d_i		&\mbox{ otherwise}\, 
		\end{cases}
\]
has the same length as $\dseq$ and satisfies $n_1(\dseq') = n_1(\dseq)$ and has two fewer  vertices of degree three, and two applications of Theorem~\ref{thm.stochasticorder2} give that $\height(\randomtree_{\dseq})  \preceq_{\mathrm{st}}\height(\randomtree_{\dseq'})$. This shows that we may restrict our attention to degree sequences corresponding to trees of maximum degree three and with at most one vertex of degree three.

Among such degree sequences, if $n_3(\dseq)=0$ there is nothing to prove -- the tree $\randomtree_{\dseq}$ is already sub-binary. If $n_3(\dseq)=1$ and $n_1(\dseq) \ge 1$ then let $k$ be such that $d_k=3$ and choose $j\in [n]$ with $d_j=1$. 
Then the sequence $\dseq'=(d_1',\ldots,d_n')$ given by 
\[
d_i' = \begin{cases}
		2	&\mbox{ if }i\in\{j,k\}\\
		d_i		&\mbox{ otherwise}\, 
		\end{cases}
\]
is sub-binary and has $n_1(\dseq') < n_1(\dseq)$, and $\height(\randomtree_{\dseq})  \preceq_{\mathrm{st}}\height(\randomtree_{\dseq'})$. 

Finally, if $n_3(\dseq)=1$ and $n_1(\dseq)=0$ then necessarily $n$ is even. 
In this case, choose $k \in [n]$ such that $d_k=3$, and define $\dseq' = (d_1',\ldots,d_{n+1}')$ by 
\[
d_i' = \begin{cases}
		d_i & \mbox{ if }i \not\in\{k,n+1\}\\
		d_k-1& \mbox{ if } i=k\\
		2	& \mbox{ if }i=n+1\, .
		\end{cases}
\]
Then $\dseq'$ is sub-binary. 
Moreover, the same argument as that in the final paragraph of the proof of Corollary~\ref{cor:sub_binary} shows that 
with $\randomtree_{\dseq'}\in_u \cT_{\dseq'}$, then $\height(\randomtree_{\dseq})  \preceq_{\mathrm{st}}\height(\randomtree_{\dseq'})$. 
This completes the proof. 
\end{proof}
\begin{proof}[Proof of Theorem~\ref{thm:bienstoch}]
For any $n\in \N$ 
with $\p(|\randomtree_{\mu}|=n)> 0$
and any plane tree $\tree$ of size $n$, 
\[
\p(\randomtree_{\mu,n}=\tree)=
\frac{1}{\p(|\randomtree_{\mu}|=n)} \prod_{v \in \tree} \mu_{\deg_\tree(v)}\, 
\]
and likewise if $\p(|\randomtree_{\nu}|=n)> 0$ then $\p(\randomtree_{\nu,n}=\tree)=
(\p(|\randomtree_{\nu}|=n))^{-1}\prod_{v \in \tree} \nu_{\deg_\tree(v)}$. 
Thus, the laws of $\randomtree_{\mu,n}$ and $\randomtree_{\nu,n}$ both have the product structure required by Proposition~\ref{prop:planetrees}. 

Now fix $n$ for which $\p(|\randomtree_{\mu}|=n)> 0$ and let $n^+ = n+\I{n~\mathrm{even}}$, so that $\p(|\randomtree_{\nu}|=n^+)>0$. Label the vertices of $\randomtree_{\mu,n}$ (resp.\ $\randomtree_{\nu,n^+}$) by a uniformly random permutation of $[n]$ (resp.  $[n^+]$) and write $\sequence{D}$ (resp.\ $\sequence{B}$) for the resulting degree sequence. Then $\sequence{D}$ is a degree sequence of length $n$ with $n_1(\sequence{D})=0$ and $\sequence{B}$ is a binary degree sequence of length $n^+$. Corollary~\ref{cor:unary_stochast} yields that with $\randomtree_{\sequence{B}} \in_u \cT_{\sequence{B}}$ and $\randomtree_{\sequence{D}} \in_u \cT_{\sequence{D}}$, then 
\[
\height(\randomtree_{\sequence{D}})\preceq_{\mathrm{st}} \height(\randomtree_{\sequence{B}})\, .
\]
But Proposition~\ref{prop:planetrees} implies that $\height(\randomtree_{\mu,n}) \eqdist \height(\randomtree_{\sequence{D}})$ and $\height(\randomtree_{\nu,n^+}) \eqdist \height(\randomtree_{\sequence{B}})$; the result follows. 
\end{proof}

\subsection{Related work}\label{sec:relatedwork}

The study of the heights of trees spans decades, beginning with work of Harary and Prins \cite{MR101846}, Riordan \cite{MR140434} and R\'enyi and Szekeres \cite{MR0219440}. 
The work \cite{MR101846} developed generating functions for the number of labeled rooted trees with given height (which they called {\em root diameter}); their work was extended to more general models of trees, including partially labeled trees, in \cite{MR140434}. R\'enyi and Szekeres \cite{MR0219440} analyzed the generating functions developed in \cite{MR101846,MR140434} to prove that for a uniformly random rooted labeled tree $\randomtree_n \in_u \cT(n)$, the height satisfies $n^{-1/2}\height(\randomtree_n) \convdist H$ for a random variable $H$ with $\e H=(2\pi )^{1/2}$. Most other early work also focussed on heights and diameters of specific random tree models, such as random labeled trees \cite{MR731595} and random plane trees \cite{MR0505710}, and on distances between typical pairs of points in such models \cite{MR506256,MR263685}.

A number of works then investigated the height (and width) of somewhat more general families of trees, such as the so-called ``simple'' trees 
\cite{MR680517,MR1798284}.
(Simple trees may be thought of as random coloured plane trees, where nodes with $c$ children may receive any colour from $\{1,\ldots,\kappa(c)\}$; the values $\{\kappa(c),c \ge 0\}$ must all be bounded by some fixed constant $M$.) In all these models, for a random tree $\randomtree_n$ of size $n$, the height $\height(\randomtree_n)$ is of order $\sqrt{n}$ in probability, for $n$ large. Of particular note in the context of the current work are the work of Flajolet {\em et al} \cite{MR1249127}, which stated the first explicit non-asymptotic tail bounds on the heights of uniform random binary trees; and of \L uczak \cite{MR1329866}, which established uniform sub-Gaussian tail bounds for the rescaled height of random labeled trees. \L uczak showed in particular that there exists an absolute constant $C>0$ such that for $\randomtree_n \in_u \cT(n)$, for all $k \ge \sqrt{n}$, 
\[
\p(\height(\randomtree_n) = k) \le C \frac{n!}{n^k(n-k)!} \frac{k^3}{n^2}\, .
\]
A sub-Gaussian tail bound follows on observing that $n!/(n^k(n-k)!)=\prod_{i=0}^{k-1}(1-i/n) \le \exp(-k(k-1)/(2n))$. 

Kolchin \cite[Theorem 2.4.3]{MR865130} proved a far-reaching generalization of the above distributional results, showing that $n^{-1/2}\height(\randomtree_{\mu,n})$ converges in distribution, to a random variable with sub-Gaussian tails, whenever $\sum_{i \ge 1} i\mu_i=1$ and $\sum_{i \ge 1} i^2\mu_i \in (0,\infty)$. All the models described in the previous paragraphs are handled by Kolchin's results.
This bound plays an essential role in Aldous's proof that critical Bienaym\'e trees conditioned to have $n$ vertices converge to the {\em Brownian continuum random tree (CRT)}, whenever the offspring distribution has finite variance. (The recent paper \cite{igorus} uses Theorem~\ref{thm.stochasticorder2} to prove lower bounds on the height of critical Bienaym\'e trees conditioned to have $n$ vertices, without any further assumptions on the offspring distribution. That work also establishes convergence of the rescaled height for critical offspring distributions in the domain of attraction of a Cauchy random variable.)

In contrast to our work, almost all recent research on non-asymptotic bounds for heights of random trees has been focussed on showing that, for sequences of trees which converge in distribution (after rescaling) to limiting continuum random trees, the tail bounds which hold for the height of the limit object can already be observed in the finite setting. In the ``finite variance'' case, where the limiting object is the Brownian CRT, this is accomplished in \cite{MR3077536,MR2956056}. Building on \cite{MR3077536}, such tail bounds have also been proved for random graph ensembles which are {\em not} trees, but which have the Brownian CRT as their scaling limit; see \cite{MR3551197,MR4132643}. 
In the case of {\em heavy-tailed} degree distributions, for which the associated random trees converge in distribution to the so-called {\em $\alpha$-stable trees}, non-asymptotic tail bounds for the height (which match those of the limiting objects) have been obtained by Kortchemski \cite{MR3651047}. 

One of the main points of this work is to show that the assumption of convergence under rescaling, a feature of all the above works, is not necessary in order to obtain strong, non-asymptotic bounds on the height. Indeed, in the setting of conditioned branching processes $\randomtree_{\mu,n}$, our main results precisely describe what information is required in order to obtain sub-Gaussian tail bounds for $n^{-1/2}\height(\randomtree_{\mu,n})$, and also under precisely what conditions $n^{-1/2}\height(\randomtree_{\mu,n})\to 0$ in probability; in both cases, convergence is not a necessary ingredient. 
In the setting of simply generated trees and random trees with fixed degree sequences, our results likewise do not depend on any assumptions about asymptotic behaviour.

\subsection{Outline}
The remainder of the paper is structured as follows. Section~\ref{sec:degseq_proofs} presents the proofs of our results on random trees with fixed degree sequences, Theorems~\ref{thm.gaussiantails} and~\ref{thm.infinitevariance}. Section~\ref{sec.bienandsimpygen} uses these results to prove our results on Bienaym\'e trees, Theorems~\ref{thm.bieninfinitevariance},~\ref{thm.biensubcriticalnoexpdecay} and~\ref{thm.biengaussiantails}. This section also introduces simply generated trees, and presents our results on their heights. Section~\ref{sec:stochastic}, which can be read independently of the rest of the paper, contains the proof of Theorem~\ref{thm.stochasticorder2}. Finally, in Section \ref{sec:conclusion}, we discuss applications and extensions of our results and techniques to convergence of random trees under rescaling; height bounds and convergence in other random tree models; convergence of the diameter of the configuration model at criticality; and non-asymptotic bounds on distances in the components of the configuration model. 

\section{\bf Proofs of Theorems~\ref{thm.gaussiantails} and~\ref{thm.infinitevariance}}\label{sec:degseq_proofs}\label{sec:bij}

In this section we present the proofs of Theorems \ref{thm.gaussiantails} and \ref{thm.infinitevariance}. To streamline the presentation, a few of the longer proofs are postponed to Section~\ref{sec:missingproofs}.
Our approach is based on a bijection between rooted trees on $[n]$ and sequences in $[n]^{n-1}$ introduced by Foata and Fuchs \cite{FoataFuchs}. (See also the recent note by Addario-Berry, Blanc-Renaudie, Donderwinkel, Maazoun and Martin on probabilistic applications of the bijection \cite{us}.) We use the version of the bijection introduced in \cite[Section 3]{us}, specialized to trees with a fixed degree sequence. 

To explain the bijection, it is convenient to focus on degree sequences with a particular form. Given a degree sequence $\dseq=(d_1,\ldots,d_n)$, define another degree sequence $\dseq'=(d_1',\ldots,d_n')$ as follows. 
Let $m$ be the number of non-zero entries of $\dseq$; necessarily $1 \le m \le n-1$. List the non-zero entries of $\dseq$ in order of appearance as $d_1',\ldots,d_m'$, and then set $d_{m+1}'=\ldots=d_n'=0$. So, for example, if $\dseq=(1,0,3,0,0,2,0)$ then $\dseq'=(1,3,2,0,0,0,0)$. Say that the degree sequence $\dseq'$ is {\em compressed}. (So a degree sequence is compressed if all of its non-zero entries appear before all of its zero entries.) 

There is a natural bijection between $\cT_{\dseq}$ and $\cT_{\dseq'}$: from a tree $\tree\in \cT_{\dseq}$, construct a tree $\tree' \in \cT_{\dseq'}$ by relabelling the non-leaf vertices of $\tree$ as $1,\ldots,m$ and the leaves of $\tree$ as $m+1,\ldots n$, in both cases in increasing order of their original labels. 
Using this bijection provides a coupling $(\randomtree,\randomtree')$ of uniformly random elements of $\cT_{\dseq}$ and $\cT_{\dseq'}$, respectively, such that $\randomtree$ and $\randomtree'$ have the same height. It follows that any tail bound for the height of a uniformly random tree in $\cT_{\dseq'}$ applies {\em verbatim} to the height of a uniformly random tree in $\cT_{\dseq}$. 

Now, let $\dseq=(d_1,\dots,d_n)$ be a compressed degree sequence, so there is $1 \le m \le n-1$ such that $d_i=0$ if and only if $i>m$.  Write $n_0=n_0(\dseq)=n-m$ for the number of leaves in a tree with degree sequence $\dseq$.  Then define
\begin{align*}
\cS_{\dseq}&:=\left\{(v_1,\dots, v_{n - 1}):|\{k:v_k=i\}|=d_i\text{ for all }i\in [n]\right\}\, .
\end{align*}
For example, if $\dseq=(1,3,2,0,0,0,0)$ then $\cS_{\dseq}$ is the set of all permutations of the vector $(1,2,2,2,3,3)$, so has size ${6 \choose 1,3,2} = 60$. 

The following bijection between $\cS_{\dseq}$ and $\cT_{\dseq}$ appears in \cite[Section 3]{us}. For a sequence $\sequence{v}=(v_1,\dots, v_{n-1}) \in \cS_{\dseq}$, we say that $j\in \{2,\dots, n-1\}$ is the location of a repeat of $\sequence{v}$ if 
$v_j=v_i$ for some $i<j$. 
\begin{tcolorbox}[title=Bijection $\tree$ between $\cS_{\dseq}$ and $\cT_{\dseq}$.]
	\begin{itemize}
		\item Let $j(0)=1$, let $j(1)<j(2)<\dots<j(n_0-1)$ be the locations of the repeats of the sequence $\sequence{v}$, and let $j(n_0)=m+n_0=n$.
		\item For $i=1,\dots, n_0$, let $P_i$ be the path
		$(v_{j(i-1)}, \dots, v_{j(i)-1}, m+i)$.
		\item Let $\tree(\sequence{v}) \in \cT_{\dseq}$ have root $v_1$ and edge set given by the union of the edge sets of the paths
		$P_1, P_2, \dots, P_{n_0}$.
	\end{itemize}
\end{tcolorbox}
The inverse of the bijection works as follows. Fix a tree $\tree \in \cT_{\dseq}$. Let $S_0 = \{r(\tree)\}$ consist of the root of $\tree$. Recursively, for $1 \le i \le n_0$ let $P_i$ be the path from $S_{i-1}$ to $m+i$ in $\tree$, and let $P_i^*$ be $P_i$ excluding its final point $m+i$. Then let $\sequence{v}=\sequence{v}(\tree)$ be the concatenation of $P_1^*,\ldots,P_{n_0}^*$. For later use, we  observe that this bijection implies that 
\begin{equation}\label{eq:td_size}
|\cT_{\dseq}| = {n-1 \choose d_1,\ldots,d_n} = \frac{(n-1)!}{\prod_{i\in[n]} d_i!}. 
\end{equation}
This formula (which appears as Theorem~5.3.4 in \cite{MR1676282}) holds for all degree sequences, not just compressed ones, by the observation from earlier in the section. 

We now discuss how to use the bijection to bound the height of $\tree(\sequence{v})$. 
For this, we think of the bijection as constructing $\tree$ from $\sequence{v}(\tree)=(v_1,\ldots,v_{n-1})$ by adding vertices one-at-a-time, in order of their first appearance in the concatenation of $P_1,\ldots,P_{n_0}$. Formally, define a permutation $(w(1),\ldots,w(n))=(w_\tree(1),\ldots,w_\tree(n))$ of $[n]$ as follows. For $1 \le k \le n$, let $w(k)=v_k$ if $k$ is not the location of a repeat, and let $w(k)=m+r$ if $j(r)=k$ (so either $r < n_0$ and $k$ is the location of the $r$'th repeat in $\sequence{v}$, or $r=n_0$ and $k=n$). Then let $\tree_k$ be the subtree of $\tree$ with vertices $w(1),\ldots,w(k)$. An example appears in Figure~\ref{fig:sequence_ex}. 

\begin{figure}[hbt]
\begin{centering}
\includegraphics[width=0.7\textwidth]{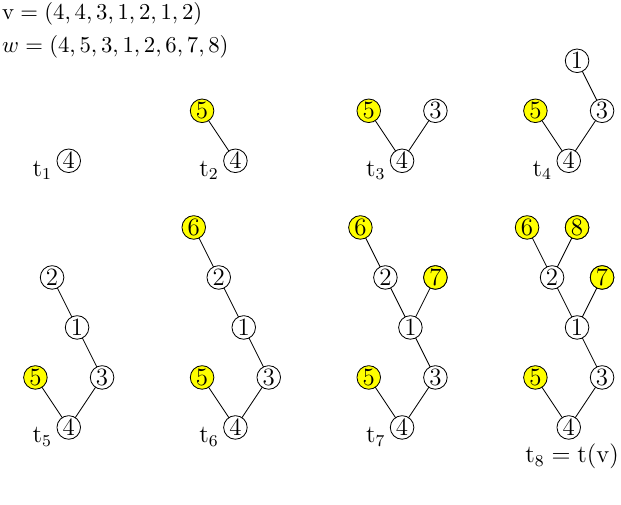}
\caption{This figure illustrates the bijection and the sequential construction. In this example, we have $\pi(2)=5$ since $\tree_5$ is the first tree in the sequence containing vertex $2$. We also have $\rho(2)=4$, since $3$ is the minimal $k$ such that $\sum_{1\leq  j\leq k}(d_{i(j)}-1)$ is at least $2$ and $\tree_4$ is the first tree in the sequence to contain vertices $\{i(1),i(2),i(3)\}=\{4,3,1\}$.}
\label{fig:sequence_ex}
\end{centering}
\end{figure}

Let $(\pi(i),i \in [n])=(\pi_\tree(i),i \in [n])$ be the inverse permutation of $w$, so $\pi(i) = w^{-1}(i)$ is the step at which the vertex with label $i$ is added when constructing $\tree$ from $\sequence{v}(\tree)$. List the non-leaf vertices $\{1,\ldots,m\}$ of $\tree$ in the order they appear when constructing $\tree$ as $(i(1),\ldots,i(m))=(i_\tree(1),\ldots,i_\tree(m))$. 
In other words, $(i(1),\ldots,i(m))$ is the permutation of $\{1,\ldots,m\}$ such that $(\pi(i(1)),\ldots,\pi(i(m)))$ is increasing. 
Then for $0 \le x \le n_0-1$, let $k=k(x)=k_\tree(x)$ be minimal so that
\[
\sum_{j=1}^k (d_{i(j)}-1) \ge x,
\]
and let $\rho(x)=\rho_\tree(x) = \pi(i(k(x)))$, so that $v_{\rho(x)}=i(k(x))$. Equivalently, $\rho(x)$ is the smallest integer $\rho$ such that the subtree of $\tree$ consisting of the nodes of $\tree_\rho$ and all their children has strictly more than $\lceil x \rceil$ leaves; this is also the smallest integer $\rho$ such that $\tree_\rho$ contains vertices $i(1),\ldots,i(k(x))$.
 Note that $\rho(x)$ is non-decreasing, and since 
\[
\sum_{j=1}^m (d_{i(j)}-1) = (\sum_{j=1}^m d_j)-m = n-1-m=n_0-1\, ,
\] 
the tree $\tree_\rho(n_0-1)$ has $n_0$ leaves. It follows that all vertices of $\tree$ that do not belong to $\tree_{\rho(n_0-1)}$ have degree either $0$ or $1$ in $\tree$. This in particular implies that if $\tree$ has no vertices of degree exactly $1$ then $\rho(n_0-1) = \pi(i(m))$ and $\tree_{\rho(n_0-1)}$ contains all vertices of $\tree$ except a subset of its leaves, so $\height(\tree) \le \height(\tree_{\rho(n_0-1)})+1$.

For any real sequence $0 \le y_0 < y_1 < \ldots < y_N = n_0-1$, 
writing $\tree=\tree(\sequence{v})$ and $\tree_k=\tree_k(\sequence{v})$, 
we may now bound $\height(\tree)$ via the telescoping sum
\begin{equation}\label{eq.telescopingsum}
\height(\tree)
\le \height(\tree_{\rho(y_0)})
+ \Big(\sum_{i=1}^N 
\big(\height(\tree_{\rho(y_i)})
-
\height(\tree_{\rho(y_{i-1})})
\big)
\Big) + (\height(\tree)-\height(\tree_{\rho(n_0-1)})),
\end{equation}
and the final term in the sum is at most $1$ if $\tree$ has no vertices of degree $1$. 
Here is the key lemma which makes such a decomposition useful. Given $\tree \in \cT_{\dseq}$, let $\sequence{v}=\sequence{v}(\tree)=(v_1,\ldots,v_{n-1})$.
\begin{lem}\label{lem:geometric_attachment}
Fix any degree sequence $\dseq=(d_1,\ldots,d_n)$. Then for all $0 \le x \le y$ and for any integer $b\geq 0$, if $\randomtree \in_u \cT_{\dseq}$ then 
\[
\p(
\dist(v_{\rho(y)},\randomtree_{\rho(x)})\, 
> b) \le \pran{1-\frac{x}{n-1}}^b.
\]
\end{lem}
The proof of Lemma~\ref{lem:geometric_attachment} appears in Section~\ref{sec:missingproofs}, below. 
This lemma is not quite enough to control an increment of the form 
$\height(\randomtree_{\rho(y)})
-
\height(\randomtree_{\rho(x)})
$, since for that we need 
to control the distance from $v_{\rho_\tree(z)}$ to $\tree_{\rho(x)}$ for {\em all} $x \le z \le y$ -- i.e.\ we need a ``maximal inequality'' version of Lemma~\ref{lem:geometric_attachment}. We achieve this in the following corollary. (Although in this section we defer the proofs of most of the supporting results, we prove the corollary immediately, as its proof is quite short.) 
\begin{cor}\label{cor.maximalinequality}
Fix any degree sequence $\dseq=(d_1,\ldots,d_n)$ with 
$n_1(\dseq)=0$, and let $\randomtree \in_u \cT_{\dseq}$. 
Then for any integers $0 \le x \le y\le n_0-1$ and any $p \in (0,1)$ and $b > 0$, 
\[
\p(
\height(\randomtree_{\rho(y)}) - \height(\randomtree_{\rho(x)})
> b+1)
\le 
\frac{y-x}{(1-p)b} \pran{1-\frac{x}{n-1}}^{\lceil pb\rceil}\, .
\]
\end{cor}
\begin{proof}
By relabelling the vertices, we may assume without loss of generality that $\dseq$ is compressed. Write $m=\max(i: d_i \ne 0)$. 

Since $\randomtree$ has no vertices of degree $1$, for all $k \in [m]$
we have $d_{i(k)} \ge 2$ and so $\sum_{j=1}^k (d_{i(j)}-1) > \sum_{j=1}^{k-1} (d_{i(j)}-1)$. It follows that the non-leaf vertices of $\randomtree_{\rho(y)}$ lying outside of $\randomtree_{\rho(x)}$ are 
a subset of 
$\{i(k(x+1)),\ldots,i(k(y))\}=\{v_{\rho(x+1)},\ldots,v_{\rho(y)}\}$, 
and so 
\begin{align*}
\height(\randomtree_{\rho(y)}) - \height(\randomtree_{\rho(x)})
& \le 1 + \max
\big(
\dist(v_{\rho(z)},\randomtree_{\rho(x)}): 
z \in\{ x+1,\ldots, y\}
\big). 
\end{align*}
Now, if for a given $z \in\{ x+1,\ldots, y\}$ we have 
\[
\dist(v_{\rho(z)},\randomtree_{\rho(x)})) > b
\]
then by considering the path in $\randomtree$ from $v_{\rho(z)}$ to $\randomtree_{\rho(x)}$, we see that there must be at least $(1-p)b$ distinct vertices $v \in \{v_{\rho(x+1)},\ldots,v_{\rho(y)}\}$ for which $\dist(v,\randomtree_{\rho(x)}) > \lceil pb\rceil $. It follows that 
\begin{align*}
&  \p(\height(\randomtree_{\rho(y)}) - \height(\randomtree_{\rho(x)}) > b+1)\\
& \le 
\p\Big(\max
\big(
\dist(v_{\rho(z)},\randomtree_{\rho(x)}): 
z \in\{ x+1,\ldots, y\}
\big)>b\Big) \\
& \le 
\p\Big(\big| 
\{z \in \{x+1,\ldots,y\}:
\dist(v_{\rho(z)},\randomtree_{\rho(x)})>\lceil pb\rceil \} 
\big| > (1-p)b\Big) \\
& \le 
\frac{1}{(1-p)b} \E{\big| 
\{z \in \{x+1,\ldots,y\}:
\dist(v_{\rho(z)},\randomtree_{\rho(x)})>\lceil pb \rceil \} 
\big|} \\
&= \frac{1}{(1-p)b}\sum_{z=x+1}^y 
\p\big(\dist(v_{\rho(z)},\randomtree_{\rho(x)})>\lceil pb\rceil \big)\, ,
\end{align*}
from which the corollary follows by the bound from Lemma~\ref{lem:geometric_attachment}. 
\end{proof}
Combining the telescoping sum bound \eqref{eq.telescopingsum} with  Corollary~\ref{cor.maximalinequality} will allow us to prove tail bounds for the heights of random trees $\randomtree\in_u \cT_{\dseq}$ for degree sequences $\dseq$ with $n_1(\dseq)=0$. The next lemma allows us to transfer tail bounds from this setting to that of general trees with fixed degree sequences.

\begin{lem}\label{lem:degree_one_stretch}
	Fix a degree sequence $\dseq=(d_1,\ldots,d_n)$ and write $n_0=n_0(\dseq)$ and $n_1=n_1(\dseq)$. 
Let $\dseq'$ be obtained from $\dseq$ by removing all entries which equal $1$. 
	Let $\randomtree \in_u \cT_\dseq$ and $\randomtree' \in_u \cT_{\dseq'}$. 
	Then for any $h> 0$ and $y \ge 4hn/(n-n_1)$, 
	\[
	\p(\height(\randomtree)\ge  y) \le \p(\height(\randomtree') >  h ) + n_0\exp(-y(n-n_1)/(8n))\, .
	\]
\end{lem}
Lemma~\ref{lem:degree_one_stretch}, whose proof appears in Section~\ref{sec:missingproofs},  is the last tool we need to  prove Theorem~\ref{thm.gaussiantails}. 
\begin{proof}[Proof of Theorem~\ref{thm.gaussiantails}]

Note that if $n \le 64$ then the result is trivially true, since $5e^{-\delta x^2/2^{12}} >1$ for $x < 8$, and $\p(\height(\tree)> x n^{1/2})=0$ for $x \ge 8$. We thus hereafter assume that $n > 64$. 

As described above, we shall bound $\p(\height(\randomtree) > xn^{1/2})$ 
by bounding the height via the telescoping sum (\ref{eq.telescopingsum}), which for the random tree $\randomtree$ becomes 
\begin{equation}\label{eq:telescope2}
\height(\randomtree)
\le \height(\randomtree_{\rho(y_0)})
+ \Big(\sum_{i=1}^N 
\big(\height(\randomtree_{\rho(y_i)})
-
\height(\randomtree_{\rho(y_{i-1})})
\big)
\Big) + (\height(\randomtree)-\height(\randomtree_{\rho(n_0-1)})),
\end{equation}
for a suitably chosen sequence $0 \le y_0 < y_1 < \ldots < y_N=n_0-1$, 
and then controlling the contribution of each summand. 

We first treat the case that $\dseq$ does not contain any entries equal to $1$. In this case, the final term in \eqref{eq:telescope2} is at most $1$, since when there are no vertices of degree one, all vertices of $\randomtree$ not lying in $\randomtree_{\rho(n_0-1)}$ are leaves. It follows that for any positive constants $(b_i,i \ge 0)$ such that $\sum_{i=0}^N b_i n^{1/2}+1 < x n^{1/2}$, 
\begin{equation}\label{eq:newtelescope}
\p(\height(\randomtree) \ge xn^{1/2})
\le \sum_{i=0}^N \p\big(\height(\randomtree_{\rho(y_i)})
-
\height(\randomtree_{\rho(y_{i-1})})
 > b_i n^{1/2}\big)\, ,
\end{equation}
where we use the convention that $\height(\randomtree_{\rho(y_{-1})})=0$. 
We choose $b_0=x/2$ and $b_i = x\cdot i/2^{i+2}$, so that 
$\sum_{i\ge 0} b_i = x$ and $\sum_{i=0}^{N+1} b_i < x$ for all $N\in \N$. We will later pick $N$ for which $b_{N+1}n^{1/2}\ge 1$, so that $\sum_{i=0}^N b_i n^{1/2}+1 \le \sum_{i=0}^{N+1} b_i n^{1/2} < xn^{1/2}$ as required.

Provided that $b_in^{1/2} \ge 4$, Corollary~\ref{cor.maximalinequality} applied with $b=b_in^{1/2}-1$ and $p=(b-1)/2b$ yields 
\begin{align*}
\p\big(\height(\randomtree_{\rho(y_i)})
-
\height(\randomtree_{\rho(y_{i-1})})
 > b_i n^{1/2}\big)
& \le \frac{y_i-y_{i-1}}{(1-p)b}\pran{1-\frac{y_{i-1}}{n-1}}^{\lceil pb\rceil}\\
 & \le \frac{2(y_i-y_{i-1})}{b+1}\pran{1-\frac{y_{i-1}}{n-1}}^{(b-1)/2} \\
 & \le \frac{2(y_i-y_{i-1})}{b_i n^{1/2}} \pran{1-\frac{y_{i-1}}{n-1}}^{b_i n^{1/2}/4}\, .
\end{align*}

Now set $y_i = \min(n_0-1,2^{i-2}xn^{1/2})$, and note that $y_0=\min(n_0-1,b_0n^{1/2}/2)$. For $i \ge 1$, if $y_{i-1}=n_0-1$ then $\randomtree_{\rho(y_i)}=\randomtree_{\rho(y_{i-1})}$ so $\p(\height(\randomtree_{\rho(y_i)})
-
\height(\randomtree_{\rho(y_{i-1})})
 > b_i n^{1/2})=0$.
On the other hand, if $y_{i-1} =2^{i-3}xn^{1/2} < n_0-1$, then we have $2(y_i-y_{i-1})/(b_in^{1/2}) \le 2^{2i}/i$ and $1-y_{i-1}/(n-1) \le e^{-2^{i-3}x/n^{1/2}}$, so if also $b_in^{1/2} \ge 4$ then this gives 
\begin{align}
\p\big(\height(\randomtree_{\rho(y_i)})
-
\height(\randomtree_{\rho(y_{i-1})})
 > b_i n^{1/2}\big)
 & \le \frac{2^{2i}}{i} \exp\pran{-\frac{2^{i-3}x}{n^{1/2}}\frac{b_in^{1/2}}{4}}
\nonumber \\
&  = \frac{2^{2i}}{i}\exp\pran{-\frac{ix^2}{2^{7}}}\, .
\label{eq:one_term}
\end{align}
But if $i\ge 2$ is small enough that $2^{i-2}xn^{1/2} \le n_0-1$, then since $n_0< n$, we also have $b_i n^{1/2}= n^{1/2} xi/2^{i+2} \ge x^2 i/16$ and so $b_in^{1/2} \ge 4$ provided $x \ge 6$. If $i=1$ then we also have $n^{1/2}b_i = n^{1/2}x/8 \ge 4$ when $x \ge 6$, since we assumed that $n \ge 64$. It follows that \eqref{eq:one_term} holds for all $i \ge 1$ when $x \ge 6$. 

To handle the $i=0$ term, note that by the definition of $\rho(y_0)$, the subtree of $\randomtree$ consisting of all nodes of $\randomtree_{\rho(y_0)-1}$ and their children contains at most $\lceil y_0\rceil$ leaves. Since there are no vertices of degree exactly $1$, this also implies that $\randomtree_{\rho(y_0)-1}$ has fewer than $\lceil y_0\rceil$ non-leaf vertices. Therefore, 
$\height(\randomtree_{\rho(y_0)-1}) \le \lceil y_0\rceil-1$, so 
$\height(\randomtree_{\rho(y_0)}) \le \lceil y_0 \rceil\le \lceil b_0 n^{1/2}/2\rceil \le  b_0 n^{1/2} $ for $x\geq 6$, and thus
\begin{equation}\label{eq:first_term}
\p\big(\height(\randomtree_{\rho(y_0)})
-
\height(\randomtree_{\rho(y_{-1})}) > b_0 n^{1/2}
\big) = \p(\height(\randomtree_{\rho(y_0)}) > b_0n^{1/2})=0\, .
\end{equation}
We claim that we may use the telescoping sum bound (\ref{eq:newtelescope}) 
with $N=\min(i:2^{i-2}xn^{1/2} \ge n_0-1)$ for $x\ge 8$. For that, we need to show that $b_{N+1}n^{1/2}=x(N+1)2^{-N-3}n^{1/2}\ge 1$. Indeed, note by definition of $N$ that $2^{N-3}x n^{1/2}\le n_0-1\le n$ so that $2^{N-3}\le n^{1/2}/x$. Therefore, $x(N+1)2^{-N-3}n^{1/2}\ge x^22^{-6}\ge 1$ as claimed. 
Then, using  (\ref{eq:one_term}) and (\ref{eq:first_term}) in the telescoping sum bound (\ref{eq:newtelescope}), 
with $N=\min(i:2^{i-2}xn^{1/2} \ge n_0-1)=\min(i:y_i=n_0-1)$, we obtain that for all $x \ge 8$, 
\[
\p(\height(\randomtree) \ge  xn^{1/2})
\le \sum_{i \ge 1}\frac{2^{2i}}{i}\exp\pran{-\frac{ix^2}{2^{7}}} 
= \sum_{i \ge 1} \exp\pran{2i\log 2 - \log i - \frac{ix^2}{2^{7}}}\, .
\]
Provided that $ x^2/2^8\ge 2\log 2$ (and in particular $x\ge 8$), the exponent in the sum is less than $-ix^2/2^8$, so for such $x$ the probability we aim to bound is at most 
\begin{equation}\label{eq:boundnodeg1largex}
\sum_{i \ge 1} \exp(-i x^2 /2^8) = e^{-x^2/2^8}\big(1-e^{-x^2/2^8}\big)^{-1} < 2e^{-x^2 /2^8}< 4e^{-x^2 /2^8}\, .
\end{equation}
On the other hand, if  $ x^2/2^8<2\log 2$ then $4e^{-x^2/2^8} > 4e^{-2\log 2}=1$, which is an upper bound on any probability.
It follows that for all $x>0$, 
\begin{equation}\label{eq.gaussianboundnodeg1}
\p(\height(\randomtree) \ge   xn^{1/2})
< 4e^{-x^2/2^8}\, .
\end{equation}
This proves (in fact, is stronger than) the necessary bound when $\randomtree$ has no vertices of degree~$1$. 

Now suppose that $n_1 > 0$ and 
write $\delta=(n-n_1(\dseq))/n$; so $\delta < 1$. 
By permuting the vertex labels we may assume that all the $1$s in $\dseq$ are at the end, i.e., $d_i \ne 1$ for $i \le n-n_1$ and $d_i=1$ for $n-n_1 < i \le n$. 
 Let $\dseq' = (d_1,\ldots,d_{n-n_1})$ be obtained by removing all entries equal to $1$ from $\dseq$, and let $\randomtree' \in_u \cT_{\dseq'}$. By Lemma~\ref{lem:degree_one_stretch}, for any $h > 0$ and any $y \ge 4hn/(n-n_1)$, 
\begin{align*}
\p(\height(\randomtree) \ge  y) & \le \p(\height(\randomtree')>  h) + n_0 e^{-y(n-n_1)/(8n)} \\
& \le \p(\height(\randomtree')>  h)+ (n-n_1)e^{-y(n-n_1)/(8n)}\, .
\end{align*}
Take $y=xn^{1/2}$ and $h= y(n-n_1)/(4n)=  x'(n-n_1)^{1/2}$, where we have set $x'=x\delta^{1/2}/4$. 
Then the above bound and \eqref{eq.gaussianboundnodeg1} together yield 
\begin{align*}
\p(\height(\randomtree) \ge  xn^{1/2})
& \le \p(\height(\randomtree') >  x'(n-n_1)^{1/2} ) + (n-n_1)e^{-x(n-n_1)/(8n^{1/2})} \\
& < 4e^{-(x')^2/2^8} + (n-n_1)e^{-x(n-n_1)/(8n^{1/2})}\\
& = 4 e^{-\delta x^2/2^{12}} + \delta n e^{-x\delta n^{1/2}/8}.
\end{align*}
We will show that the second term on the right is at most $e^{-\delta x^2/2^{12}}$ for all relevant values of $x$ and $\delta$. 
First, note that we may assume $\delta x^2\ge2^{12} \log 5>16^2$, as otherwise the claimed upper bound in the theorem is at least $1$ which is an upper bound for any probability. Second, we only need to consider $x\leq n^{1/2}$, because otherwise $\height(\randomtree) <xn^{1/2}$ deterministically. 

By taking logarithms and rearranging terms we see that we need to show that
\[8\log(\delta n)\leq \delta n^{1/2}x (1-n^{-1/2}2^{-9}x),\]
but, since $x\leq n^{1/2}$, we see that the last factor is at least $(1-2^{-9})>1/2$, and $\delta x^2>16^2$ gives that $x> 16\delta^{-1/2}$, so it is in fact enough to show that 
 \[\log(\delta n)\leq (\delta n)^{1/2}.\]
Finally, $\log(u)u^{-1/2}$ is at most $2/e<1$ for $u>0$, which establishes this last bound.\end{proof}

In the preceding proof, the first term in the telescoping sum could be essentially ``given away'' - we just used the deterministic bound $\height(\randomtree_{\rho(y_0)}) \le \lceil y_0\rceil$. In the setting of Theorem~\ref{thm.infinitevariance}, where we aim for stronger bounds when $\sigma_{\dseq}$ is large than when it is small, we can not be so casual about this first term. 
The additional tool we require, to bound the height of the random tree constructed during the early stages of the bijection in the case when $\sigma_{\dseq}$ is large, is given in the following proposition. 
\begin{prop}\label{cor.firststep}
	Fix a degree sequence $\dseq=(d_1,\ldots,d_n)$ and let $\randomtree \in_u \cT_{\dseq}$. 
	
	Then 
	taking $\alpha =(1-e^{-2})/24$, 
	with $\sigma=\sigma_\dseq$, we have that for all $b\geq 1$,
	
	\[
	\p\left(\height\left(\randomtree_{\rho(\alpha\sigma n^{1/2})} \right)>\tfrac{bn^{1/2}}{2\sigma} \right)
	\leq \exp\left(-\tfrac{3}{32}\tfrac{bn^{1/2}}{\sigma} \right)+\exp\left(-\tfrac{\alpha b}{8}\right).
	\]
\end{prop}
The proof of Proposition~\ref{cor.firststep} appears in Section~\ref{sec:missingproofs}. 
With this proposition in hand, the proof of Theorem~\ref{thm.infinitevariance} proceeds quite similarly to that of Theorem~\ref{thm.gaussiantails} (though it is necessarily somewhat more involved as the choices of the values $y_i$ and $b_i$ for the telescoping sum bound must depend on $\sigma_{\dseq}$).

\begin{proof}[Proof of Theorem~\ref{thm.infinitevariance}]
In the proof we write $n_i=n_i(\dseq)$ and $\sigma=\sigma_{\dseq}$ for succinctness. 
We first assume that $n_1=0$, so $\randomtree$ contains no degree-$1$ vertices and $\sigma=\sigma'_{\dseq}$.  For such degree sequences, we shall only assume that $x\ge 2^{13}$. 
 In the proof of Theorem~\ref{thm.gaussiantails} we showed that for degree sequences with $n_1=0$, for $y \ge 2^9\log 2$, it holds that $\p(\height(\randomtree)\ge yn^{1/2}) < 2e^{-y^2/2^{8}}$; see inequality~(\ref{eq:boundnodeg1largex}). If $x^2 > 2^{11}$ then $(x\log(\sigma+1))^2/(2^8\sigma^2)>8x\log^2(\sigma+1)/\sigma^2$, so for $\sigma  \in (0,e)$ we have

\begin{align*}
\p\left(\height(\randomtree) \ge  xn^{1/2}\tfrac{\log(\sigma+1)}{\sigma}\right)
& < 2\exp\pran{-\frac{ (x\log(\sigma+1))^2}{2^8\sigma^2}}\\
& < 2\exp\pran{-\frac{ 8\log(\sigma+1)}{\sigma^2}(x\log(\sigma+1))} < 2e^{-x\log(\sigma+1)}\, ,
\end{align*}
where for the last bound we use that $\log(1+u)/u^2> 1/8$ for $0<u< e$. This bound is stronger than what we claim in the theorem, so continuing to assume that $n_1=0$, we now turn our attention to the case that $\sigma \ge e$. Like in the proof of Theorem~\ref{thm.gaussiantails}, we begin by using the telescoping sum inequality  (\ref{eq.telescopingsum}), and obtain the bound
\begin{equation}\label{eq:jameswebb} 
\p(\height(\randomtree) \ge  xn^{1/2}\tfrac{\log (\sigma+1)}{\sigma})
\le \sum_{i=0}^N \p\big(\height(\randomtree_{\rho(y_i)})
-
\height(\randomtree_{\rho(y_{i-1})})
 > b_i n^{1/2}\tfrac{\log (\sigma+1)}{\sigma} \big)\, ,
\end{equation}
which holds for any sequence $0 \le y_0 < y_1 < \ldots < y_N=n_0-1$ and positive constants $b_0,\ldots,b_N$ with $\sum_{i=0}^N b_i n^{1/2}\tfrac{\log (\sigma+1)}{\sigma}+1< x n^{1/2}\tfrac{\log (\sigma+1)}{\sigma}$, 
and then controlling the contribution of each summand. 
(We again take $\height(\randomtree_{\rho(y_{-1})}):=0$ for convenience.) 
Also as in the proof of Theorem~\ref{thm.gaussiantails}, we take $b_0=x/2$ and $b_i = x\cdot i/2^{i+2}$, so that 
$\sum_{i\ge 0} b_i = x$ and $\sum_{i=0}^{N+1} b_i < x$ for all $N\in \N$.  We will later pick $N$ for which $b_{N+1}n^{1/2}\ge \sigma/\log(\sigma+1)$, so that $\sum_{i=0}^N b_i n^{1/2}+1 \le \sum_{i=0}^{N+1} b_i n^{1/2} \le xn^{1/2}$ as required.

Now, we take $y_0 = \sigma n^{1/2}(1-e^{-2})/24$, 
so that by Proposition~\ref{cor.firststep} applied with $b=x\log (\sigma+1)=2b_0\log(\sigma+1)>1$ we obtain 
\begin{align}\begin{split}\label{eq:t0_term}
\p\big(\height(\randomtree_{\rho(y_0)} > b_0n^{1/2}\tfrac{\log(\sigma+1)}{\sigma})
&<
\exp\pran{-\tfrac{3}{32}xn^{1/2}\tfrac{\log (\sigma+1)}{\sigma}}
+ \exp\pran{-\tfrac{1-e^{-2}}{192}x\log(\sigma+1)} \\
&\le 2e^{-x\log(\sigma+1)/256}\, ,
\end{split}\end{align}
the last inequality holding since $\sigma \le n^{1/2}$ and $\min(3/32,(1-e^{-2})/192) > 1/256$. 

For the remaining summands, we take $y_i=\min(n_0-1,2^iy_0)$. Provided we have that $b_i n^{1/2}\tfrac{\log(\sigma+1)}{\sigma} \ge 4$, then 
Corollary~\ref{cor.maximalinequality} applied with $b=b_in^{1/2}\tfrac{\log (\sigma+1)}{\sigma}-1$  and $p=(b-1)/2b$ yields 
\begin{align*}
&\p\big(\height(\randomtree_{\rho(y_i)})
-
\height(\randomtree_{\rho(y_{i-1})})
 > b_i n^{1/2}\tfrac{\log(\sigma+1)}{\sigma}\big)\\
&\quad \le \frac{y_i-y_{i-1}}{(1-p)b}\pran{1-\frac{y_{i-1}}{n-1}}^{\lceil pb\rceil}\\
 & \quad\le \frac{2(y_i-y_{i-1})}{b+1}\pran{1-\frac{y_{i-1}}{n-1}}^{(b-1)/2} \\
 &\quad \le \frac{2(y_i-y_{i-1})}{b_i n^{1/2}}\frac{\sigma}{\log(\sigma+1)} \pran{1-\frac{y_{i-1}}{n-1}}^{b_i n^{1/2}(\log (\sigma+1))/(4\sigma)}\, .
\end{align*}
To simplify this upper bound, note that  
\[
\frac{2(y_i-y_{i-1})}{b_in^{1/2}}\frac{\sigma}{\log(\sigma+1)} 
\leq \frac{2^{2i+2}(1-e^{-2})}{24ix}\frac{\sigma^2}{\log(\sigma+1)}
< 
\frac{2^{2i-2}\sigma^2}{ix\log(\sigma+1)}\, ,
\] 
and provided $y_{i-1} < n_0-1$ we also have $1-y_{i-1}/(n-1) \le e^{-2^{i-1}\sigma (1-e^{-2})/(24n^{1/2})}$, so 
\begin{align*}
\pran{1-\frac{y_{i-1}}{n-1}}^{b_i n^{1/2}(\log(\sigma+1))/(4\sigma)}
&\le 
\exp\pran{-\frac{2^{i-1}\sigma (1-e^{-2})}{24 n^{1/2}} \frac{b_i n^{1/2}\log(\sigma+1)}{4\sigma}} \\&< \exp(-ix\log(\sigma+1)/2^{10})\, .
\end{align*}
Combining the three preceding bounds, it follows that when $b_in^{1/2}\tfrac{\log(\sigma+1)}{\sigma} \ge 4$ we have 
\begin{align*}
&  \p\big(\height(\randomtree_{\rho(y_i)})
-
\height(\randomtree_{\rho(y_{i-1})})
 > b_i n^{1/2}\tfrac{\log(\sigma+1)}{\sigma}\big) \\
 & < 
\frac{2^{2i-2}\sigma^2}{ix\log(\sigma+1)}
\exp\pran{-\frac{ix\log(\sigma+1)}{2^9}}\\
& \le \frac{\sigma^2}{4x\log(\sigma+1)} 
\exp\pran{(2\log 2)i -  \frac{ix\log(\sigma+1)}{2^{10}}}\, .
\end{align*}
(If $y_{i-1} \ge n_0-1$ then $\p(\height(\randomtree_{\rho(y_i)})
-
\height(\randomtree_{\rho(y_{i-1})})
 > b_i n^{1/2})=0$ so the bound holds in this case as well.)
By the assumption that $x \ge 2^{13}$ and since $\sigma \ge e$, we have $(x\log(\sigma+1))/2^{10} > 2(2\log 2)$, so this yields 
\[
\p\big(\height(\randomtree_{\rho(y_i)})
-
\height(\randomtree_{\rho(y_{i-1})})
 > b_i \tfrac{\log(\sigma+1)}{\sigma} n^{1/2}\big) 
<
\frac{\sigma^2}{2^{{15}}}\exp\pran{-\frac{ix\log(\sigma+1)}{2^{11}}}\, .
\]

For the above bound to hold we needed that $b_in^{1/2}\tfrac{\log(\sigma+1)}{\sigma} \ge 4$, 
or, equivalently, that $ixn^{1/2}\log( \sigma+1) /(2^{i+4}\sigma) \ge 1$. 
But since $n_0-1 < n$, the condition 
$y_{i-1} < n_0-1$ implies that 
\[
\frac{2^{i-1}\sigma n^{1/2}(1-e^{-2})}{24}< n\, ,
\]
so $ix\log(\sigma+1) n^{1/2}/(2^{i+4}\sigma) \ge 1$ 
provided that $x\log(\sigma+1) > 3\cdot 2^8/(1-e^{-2})$, which holds since $x \ge 2^{13}$ and $\sigma \ge e$. 

We claim that we may use the  bound \eqref{eq:jameswebb} 
with $N=\min(i:2^i\sigma n^{1/2}\tfrac{1-e^{-2}}{24} \ge n_0-1)$. To justify this, we need to show that $b_{N+1}n^{1/2}\log(\sigma+1)/\sigma=x(N+1)2^{-N-3}\log(\sigma+1)n^{1/2}\sigma^{-1}\ge 1$. Note by definition of $N$ that $2^{N-1}\sigma n^{1/2}\tfrac{1-e^{-2}}{24} \le n_0-1\le n$ so that $2^{N-1}\le n^{1/2}\tfrac{24}{1-e^{-2}}\sigma^{-1}$. Therefore, $x(N+1)2^{-N-3}\log(\sigma+1)n^{1/2}\sigma^{-1}\ge x\log(\sigma+1)\tfrac{(1-e^{-2})}{3\cdot 2^7}\ge 1$, the last inequality holding since $x\log(\sigma+1) > 3\cdot 2^8/(1-e^{-2})$, as we previously observed.

It follows that 
\begin{align*}
\sum_{i =1}^N 
&\p\big(\height(\randomtree_{\rho(y_i)})
-
\height(\randomtree_{\rho(y_{i-1})})> b_i n^{1/2}\tfrac{\log(\sigma+1)}{\sigma}\big) \\
& \le 
\frac{\sigma^2}{2^{15}}\exp\pran{-\frac{x\log(\sigma+1)}{2^{11}}}\pran{1-e^{-(x\log(\sigma+1))/2^{11}}}^{-1}\\
&  
\le \frac{1}{2^{14}}\exp\pran{-\frac{x\log(\sigma+1)}{2^{12}}}\, ,
\end{align*}
the last bound holding since $x\log(\sigma+1) \ge x \ge  2^{13}$ and thus 
$\sigma^2 e^{-(x\log(\sigma+1))/2^{11}} < e^{-(x\log(\sigma+1))/2^{12}}$ and $(1-e^{-(x\log(\sigma+1))/2^{11}})^{-1} < 2$.

Using this bound and \eqref{eq:t0_term} in \eqref{eq:jameswebb}, we thus obtain 
\[
\p(\height(\randomtree) > xn^{1/2}\tfrac{\log(\sigma+1)}{\sigma})
\le
3\exp\pran{-\frac{x\log(\sigma+1)}{2^{12}}}\, .
\]
Combining this with the bound from the start of the proof, which handles the case $\sigma < e$, it follows that when there are no vertices of degree one, for all $x \ge 2^{13}$ we have 
\begin{equation}\label{eq:inf_var_nodegone}
\p\big(\height(\randomtree) >  xn^{1/2}\tfrac{\log (\sigma+1)}{\sigma}\big) \le 3\exp\pran{-\frac{x\log(\sigma+1)}{2^{12}}}\, .
\end{equation}
This proves the theorem (in fact, something slightly stronger) when $\randomtree$ has no vertices of degree one. 

\medskip
Now suppose that $\dseq$ contains vertices of degree $1$. Like in the proof of Theorem~\ref{thm.gaussiantails}, we let $\dseq'$ be obtained from $\dseq$ by removing all entries equal to $1$.  Let $\randomtree\in_u\cT_{\dseq}$ and let $\randomtree'\in_u\cT_{\dseq'}$. Write $\sigma=\sigma_{\dseq}$. Also write $n'=n-n_1$ for the number of vertices of $\randomtree'$ so that for $\delta=n'/n$ we have that $\sigma_{\dseq'}^2=\delta^{-1}\sigma^2=(\sigma')^2$.

 We restrict our attention to $x$ such that $x\log(\sigma'+1)\geq 2^{13}\log 4$, since for smaller $x$ our claimed upper bound exceeds $1$, which is an upper bound for any probability. 
	
	We claim that our upper bound for $\sigma'<e$ follows from Theorem \ref{thm.gaussiantails}. Indeed,  for $x\log(\sigma'+1)\geq 2^{13}\log 4$, Theorem \ref{thm.gaussiantails} gives that
	\begin{align*}\p\left(\height(\randomtree)\ge xn^{1/2}\frac{\log(\sigma'+1)}{\sigma}\right)&\le 5\exp(-2^{-12}\delta \sigma^{-2} x^2 \log^{2}(\sigma'+1) )\\
	&\le5\exp(-2\log(4) (\sigma')^{-2} x\log(\sigma'+1))\\
	&\le 5 \exp(-2e^{-2}\log(4)  x\log(\sigma'+1)),\end{align*}
	which is stronger than the bound asserted in Theorem~\ref{thm.infinitevariance}. 
	
	Now assume $\sigma'\ge e$. By Lemma \ref{lem:degree_one_stretch}, for any $h> 0$ and any $y\geq 4h n/n'$, 
	
	\[\p\left(\height(\randomtree)\ge y\right)\leq \p\left(\height(\randomtree')> h\right)+n_0 e^{-yn'/(8n)}\leq  \p\left(\height(\randomtree')> h\right)+n' e^{-yn'/(8n)}.\]
	
	We take 
	\[y=xn^{1/2} \frac{\log(\sigma'+1)}{\sigma} , \qquad h=\frac{y}{4}\frac{n'}{n}=\frac{x}{4} (n')^{1/2}\frac{\log(\sigma'+1)}{\sigma'}.\]
	Then, if $x/4\geq 2^{13}$, \eqref{eq:inf_var_nodegone} gives that 
	
	\[
	\p\left(\height(\randomtree')> h \right)\leq 3\exp\left(-x\log(\sigma'+1)/2^{14}\right).
	\]
	It remains to show that 
		\[
		n' e^{-yn'/(8n)}=n'\exp\left(-x (n')^{1/2}\log(\sigma'+1)/(8\sigma')\right)< \exp\left(-x\log(\sigma'+1)/2^{14}\right).
		\]
		By taking logarithms and rearranging we see that this is equivalent to 
	\[\frac{16 \log (n')^{1/2}}{(n')^{1/2}} < x\frac{\log(\sigma'+1)}{\sigma'}(1-2^{-{11}}\sigma'(n')^{-1/2}).\]
	However, $(n')^{1/2}>\sigma'$, so the last factor is at least $(1-2^{-{11}})>1/2$, and we assumed $x\ge 2^{15}$ so it is sufficient to show that
	\[\frac{ \log (n')^{1/2}}{(n')^{1/2}} < 2^{10} \frac{\log(\sigma'+1)}{\sigma'}.\]
	But $(n')^{1/2}>\sigma'\ge e$, so this inequality indeed holds.  
\end{proof}

\subsection{\bf The postponed proofs}\label{sec:missingproofs}
This section contains the proofs of the results that were stated without proof earlier in Section~\ref{sec:degseq_proofs}: namely, Lemmas~\ref{lem:geometric_attachment} and~\ref{lem:degree_one_stretch} and Proposition~\ref{cor.firststep}. 
\begin{proof}[Proof of Lemma~\ref{lem:geometric_attachment}]
It is useful to assume $\dseq$ is compressed (and in particular that $d_n = 0$). 
Write $\pnt^1_{\tree}(u)=\pnt_{\tree}(u)$ for the parent of vertex $u$ in tree $\tree$; we define the parent of the root to be the root itself. Inductively set $\pnt^{b+1}_\tree(u)=\pnt_{\tree}(\pnt^{b}_{\tree}(u))$ for $b \ge 1$. 
Throughout the proof, we take $\randomtree \in_u \cT_{\dseq}$ and let $\sequence{V}=\sequence{v}(\randomtree)$, so $\sequence{V} \in_u \cS_{\dseq}$.

Fix any vector $(i(1),\ldots,i(j))$ of distinct elements of $[m]$ with $\sum_{k=1}^j (d_{i(k)}-1) \ge y$, 
and let $\cS_{\dseq}(i(1),\ldots,i(j))$ be the set of vectors $\sequence{v} \in \cS_{\dseq}$ such that, writing $\tree=\tree(\sequence{v})$, 
\[
(i_\tree(1),\ldots,i_\tree(j))=(i(1),\ldots,i(j)). 
\]
Let $r \in [j]$ be such that 
\[
\sum_{k=1}^{r-1} (d_{i(k)}-1) < y \le \sum_{k=1}^r (d_{i(k)}-1),
\]
Note that if $\sequence{v} \in \cS_{\dseq}(i(1),\ldots,i(j))$ and $\tree=\tree(\sequence{v})$ then $v_{\rho_\tree(y)}=i(r)=i_\tree(r)$. 
This allows us to rewrite 
\begin{align*}
& \probC{
\dist(v_{\rho(y)},\randomtree_{\rho(x)})>k
}
{
(i_{\randomtree}(1),\ldots,i_{\randomtree}(j))=(i(1),\ldots,i(j))} \\
& 
=
\probC{
\pnt_{\randomtree}^k(i(r)) \not\in \randomtree_{\rho_{\randomtree}(x)}}
{
(i_{\randomtree}(1),\ldots,i_{\randomtree}(j))=(i(1),\ldots,i(j))
}\, .
\end{align*}
For 
$\sequence{V} \in_u \cS_{\dseq}$, the ordering of the integers $\{1,\ldots,n\}$ by their first appearance in $\sequence{V}$ is degree-biased, so 
\[
\p{(i_{\randomtree}(1),\ldots,i_{\randomtree}(j))=(i(1),\ldots,i(j))}
=
\prod_{k=1}^j \frac{d_{i(k)}}{n-1-d_{i(1)}-\ldots-d_{i(k-1)}}\, ,
\]
or equivalently 
\[
|\cS_{\dseq}(i(1),\ldots,i(j))| = {n-1 \choose d_1,\ldots,d_n} \cdot 
\prod_{k=1}^j \frac{d_{i(k)}}{n-1-d_{i(1)}-\ldots-d_{i(k-1)}}.
\]

Fix $\ell < r$ and let $\dseq^\ell=(d_1^\ell,\ldots,d_{n-1}^\ell)$ where $d_{i(\ell)}^\ell=d_{i(\ell)}-1$ and $d^\ell_i=d_i$ for all $i \in [n-1]\setminus\{i(\ell)\}$. Since $d_n=0$ we have $\sum_{i \in [n-1]}d_i^\ell=n-2$, so $\dseq^\ell$ is another degree sequence. 
Now consider the subset $\cS^\ell_{\dseq}(i(1),\ldots,i(j))$ of $\cS_{\dseq}(i(1),\ldots,i(j))$ consisting of those vectors $\sequence{v}$ where the first instance of $i(r)$ in $\sequence{v}$ is the immediate successor of some instance of $i(\ell)$ other than the first. 
The set $\cS^\ell_{\dseq}(i(1),\ldots,i(j))$ is in bijection with the set of vectors $\sequence{v}' \in \cS_{\dseq^\ell}(i(1),\ldots,i(j))$, i.e., the set of vectors $\sequence{v}' \in \cS_{\dseq^\ell}$ such that 
\[
(i(1,\sequence{v}'),\ldots,i(j,\sequence{v}'))=(i(1),\ldots,i(j)). 
\]
To see this, fix $\sequence{v} \in \cS^\ell_{\dseq}(i(1),\ldots,i(j))$, and let $\sequence{v}'$ be the vector obtained from $\sequence{v}$ by deleting the entry with value $i(\ell)$ immediately preceding the first instance of $i(r)$. Then $(i(1,\sequence{v}),\ldots,i(j,\sequence{v}))=(i(1,\sequence{v}'),\ldots,i(j,\sequence{v}'))$, since the deleted entry was not the first instance of $i(\ell)$ in $\sequence{v}$, so $\sequence{v}'$ is an element of $\cS_{\dseq^\ell}(i(1),\ldots,i(j))$. To recover $\sequence{v}$ from $\sequence{v}'$, simply insert an entry with value $i(\ell)$ immediately before the first instance of $i(r)$ in $\sequence{v}'$. 

The same computation as for the size of $\cS_{\dseq}(i(1),\ldots,i(j))$ now yields the formula 
\[
|\cS^\ell_{\dseq}(i(1),\ldots,i(j))| = |\cS_{\dseq^\ell}(i(1),\ldots,i(j))|
=\! {n-2 \choose d^\ell_1,\ldots,d^\ell_n}
\prod_{k=1}^j \frac{d^\ell_{i(k)}}{n-2-d^\ell_{i(1)}\!-\ldots-d^\ell_{i(k-1)}}.
\]

Writing $E(r,\ell)$ for the event that the first instance of $i(r,\sequence{V})$ in $\sequence{V}$ is an immediate successor of some instance of $i(\ell,\sequence{V})$ other than the first, it follows that 
\begin{align*}
& \probC{E(r,\ell)}{(i_{\randomtree}(1),\ldots,i_{\randomtree}(j))=(i(1),\ldots,i(j))}\\
& = \frac{|\cS^\ell_{\dseq}(i(1),\ldots,i(j))|}{|\cS_{\dseq}(i(1),\ldots,i(j))|}\\
& = \frac{d_{i(\ell)}}{n-1} 
\cdot\prod_{k=1}^j \frac{d^\ell_{i(k)}}{d_{i(k)}} 
\cdot\prod_{k=1}^j\frac{n-1-d_{i(1)}-\ldots-d_{i(k-1)}}{n-2-d^\ell_{i(1)}-\ldots-d^\ell_{i(k-1)}}\\
& = \frac{d_{i(\ell)}-1}{n-1}\cdot \prod_{k=1}^j\frac{n-1-d_{i(1)}-\ldots-d_{i(k-1)}}{n-2-d^\ell_{i(1)}-\ldots-d^\ell_{i(k-1)}} \\
& > \frac{d_{i(\ell)}-1}{n-1}\, .
\end{align*}

Now note that if $E(r,\ell)$ occurs then $i_{\randomtree}(\ell)$ is the parent of $i_{\randomtree}(r)$ in $\randomtree=\tree(\sequence{V})$. Moreover,  
letting $q \in [j]$ be such that 
\[
\sum_{k=1}^{q-1} (d_{i(k)}-1)< x \le \sum_{k=1}^{q} (d_{i(k)}-1)\, ,
\]
then on the event $\{(i_{\randomtree}(1),\ldots,i_{\randomtree}(j))=(i(1),\ldots,i(j))\}$, the tree $\randomtree_{\rho(x)}$ has vertices $i(1),\ldots,i(q)$. It follows that
\begin{align*}
& \probC{\pnt_{\randomtree}(v_{\rho_\randomtree(y)}) \in \tree_{\rho(x)}(\sequence{V})}{(i_{\randomtree}(1),\ldots,i_{\randomtree}(j))=(i(1),\ldots,i(j))}\\
&= \probC{\pnt_{\randomtree}(i_\randomtree(r)) \in \tree_{\rho(x)}(\sequence{V})}{(i_{\randomtree}(1),\ldots,i_{\randomtree}(j))=(i(1),\ldots,i(j))}\\
& \ge
\sum_{\ell=1}^q 
\probC{E(r,\ell)}{(i_{\randomtree}(1),\ldots,i_{\randomtree}(j))=(i(1),\ldots,i(j)} \\
& > \sum_{\ell=1}^q \frac{d_{i(\ell)}-1}{n-1} \ge \frac{x}{n-1}\, .
\end{align*}
This proves the case $b=1$ of the lemma by summing over $(i_{\randomtree}(1),\ldots,i_{\randomtree}(j))$. In fact, we have proved something slightly stronger: in the case $b=1$, the necessary bound holds even conditionally given the values of $(i_{\randomtree}(1),\ldots,i_{\randomtree}(j))$. From this we will prove the full lemma by arguing inductively, by conditioning on the label of the parent of $i_\randomtree(r)$. 

There are two cases to consider. In the first case, the first instance of $i_\randomtree(r)$ in $\sequence{V}$ is the immediate successor of a repeat and in the second case the first instance of $i_\randomtree(r)$ is the immediate successor of the first instance of $i_\randomtree(r-1)$.  In either case, we continue to take $\randomtree \in_u \cT_{\dseq}$ and $\sequence{V}=\sequence{v}(\randomtree)$. 
We now fix a sequence $(i(1),\ldots,i(j))$ of distinct elements of $[m]$ and consider the two cases in turn.

For the first case, fix $\ell \in \{q+1,\ldots,r-1\}$ and suppose that $(i_{\randomtree}(1),\ldots,i_{\randomtree}(j))=(i(1),\ldots,i(j))$ and that $E(r,\ell)$ occurs. Under this conditioning, $\sequence{V} \in_u \cS^\ell_{\dseq}(i(1),\ldots,i(j))$, and in particular the parent of $i(r)=i_\randomtree(r)$ in $\randomtree$ is $i(\ell)$. 
Now let $\sequence{V}'$ be obtained from $\sequence{V}$ by deleting the entry immediately preceding the first instance of $i(r)$ in $\sequence{V}$ and let $\randomtree'=\tree(\sequence{V}')$. 
Under this conditioning, $\sequence{V}' \in_u \cS_{\dseq^\ell}(i(1),\ldots,i(j))$, which in particular implies that $(i_{\randomtree'}(1),\ldots,i_{\randomtree'}(j))=(i_{\randomtree}(1),\ldots,i_{\randomtree}(j))$. Moreover, the sequences $\sequence{V}$ and $\sequence{V}'$ agree until after the first instance of $i(\ell)$, so the ancestral line of $i(\ell)$ is the same in both $\randomtree$ and $\randomtree'$. Since $q < \ell$, it likewise
holds that $\rho_{\randomtree}(x)=\rho_{\randomtree'}(x)$ and, writing $\rho(x)$ for their common value, that $\randomtree_{\rho(x)}=\randomtree'_{\rho(x)}$. Now, let $W \in_u \cS_{\dseq^{\ell}}$, where $\dseq^\ell$ is as defined above. It follows that 
\begin{align}
& \probC{\pnt_{\randomtree}^{b}(i(r)) \not\in \randomtree_{\rho(x)}}
{(i_\randomtree(1),\ldots,i_\randomtree(j))=(i(1),\ldots,i(j)),E(r,\ell)} \nonumber\\
& 
= 
\probC{\pnt_{\randomtree'}^{b-1}(i(\ell)) \not\in \randomtree'
_{\rho(x)}}
{(i_\randomtree(1),\ldots,i_\randomtree(j))=(i(1),\ldots,i(j)),E(r,\ell)} \nonumber\\
& 
= 
\probC{\pnt^{b-1}_{\tree(\sequence{W})}(i(\ell)) \not\in \tree(\sequence{W})_{\rho(x)}}
{(i_{\tree(\sequence{W})}(1),\ldots,i_{\tree(\sequence{W})}(j))=(i(1),\ldots,i(j))}\, \nonumber\\
& \le \pran{1-\frac{x}{n-2}}^{b-1}\, . \label{eq:jump_induction}
\end{align} 
The second equality holds since under the second conditioning 
$\sequence{V}' \in_u \cS_{\dseq^\ell}(i(1),\ldots,i(j))$ and under the third conditioning 
$\sequence{W} \in_u \cS_{\dseq^\ell}(i(1),\ldots,i(j))$ and $\rho_{\tree(\sequence{W})}(x)=\rho_{\randomtree}(x)$.
The last inequality holds by induction (on $b$ or $n$ or both) and since $\dseq^\ell$ has length $n-1$.

For the second case, 
consider the subset $\cS^*_{\dseq}$ of $\cS_{\dseq}(i(1),\ldots,i(j))$ consisting of those vectors $\sequence{v}$ where the first instance of $i(r)$ in $\sequence{v}$ is the immediate successor of the first instance of $i(r-1)$. 
Let $\dseq^*=(d^*_1,\ldots,d^*_{n-1})$, where 
\[
d^*_{i(k)} = 
\begin{cases}
	d_{i(k)}	& \mbox{ if }k < r-1 \\
	d_{i(r-1)}+d_{i(r)}-1 & \mbox{ if }k=r-1\\
	d_{i(k+1)} & \mbox{ if } r \le k < n\, ,
  \end{cases}
\]
and let $(i^*(1),\ldots,i^*(j-1))=(i(1),\ldots,i(r-1),i(r+1),\ldots,i(j-1))$. 
From a vector $\sequence{v} \in \cS^*_{\dseq}$, we 
can obtain a vector $\sequence{v}' \in \cS_{\dseq^*}(i^*(1),\ldots,i^*(j-1))$ by deleting the first instance of $i(r)$, then changing all other instances of $i(r)$ to $i(r-1)$. Conversely, given a vector $\sequence{v}' \in \cS_{\dseq^*}(i^*(1),\ldots,i^*(j-1))$, there are ${d_{i(r-1)}+d_{i(r)}-2 \choose d_{i(r)}-1}$ 
ways to reconstruct a vector $\sequence{v} \in \cS^*_{\dseq}$: first insert an entry with value $i(r)$ just after the first instance of $i(r-1)$, then replace $d_{i(r)}-1$ of the other instances of $i(r-1)$ by $i(r)$'s. It follows that for $\sequence{V} \in_u \cS_{\dseq}$ as above, conditionally given that $\sequence{V} \in \cS_{\dseq}^*$, then $\sequence{V}' \in_u \cS_{\dseq^*}(i^*(1),\ldots,i^*(j-1))$. Moreover, the ancestral lines of $i(r-1)$ are the same in both $\tree(\sequence{V})$ and 
$\tree(\sequence{V}')$, and 
if $\pnt_{\tree(\sequence{V})}(i(r)) \not\in \tree_{\rho(x)}(\sequence{V})$ then we must have $q < r-1$, so also $\rho_{\tree(\sequence{V})}(x)=\rho_{\tree(\sequence{V}')}(x)=:\rho(x)$, and 
$\tree_{\rho(x)}(\sequence{V})=\tree_{\rho(x)}(\sequence{V}')$. 

Let $F(r)$ be the event that the first instance of $i_\randomtree(r)$ is an immediate successor of the first instance of $i_\randomtree(r-1)$, 
and let $W \in_u \cS_{\dseq^*}$. 
By the conclusions of the preceding paragraph, it follows that 
\begin{align}
& \probC{\pnt_{\randomtree}^b(i(r)) \not\in \randomtree_{\rho(x)}}
{(i_\randomtree(1),\ldots,i_\randomtree(j))=(i(1),\ldots,i(j)),F(r)} \nonumber\\
& 
= 
\probC{\pnt_{\tree(\sequence{V}')}^{b-1}(i(r-1)) \not\in \tree(\sequence{V}')_{\rho(x)}}
{(i_{\tree(\sequence{V}')}(1),\ldots,i_{\tree(\sequence{V}')}(j))=(i(1),\ldots,i(j)),F(r)} \nonumber\\
& 
= 
\mathbf{P}\left(\pnt^{b-1}_{\tree(\sequence{W})}(i^*(r-1)) \not\in \tree(\sequence{W})_{\rho(x)} \big|
(i_{\tree(\sequence{W})}(1),\ldots,i_{\tree(\sequence{W})}(j-1)=(i^*(1),\ldots,i^*(j-1))\right)\nonumber\\
& \le \pran{1-\frac{x}{n-2}}^{b-1}\, ,\label{eq:nbr_induction}
\end{align}
the second equality holding since $i^*(r-1)=i(r-1)$ and $\rho_{\tree(\sequence{W})}(x)=\rho(x)$, and the final bound again holding by induction. 

For $V \in_u \cS_{\dseq}$, on the event that $(i(1,\sequence{V}),\ldots,i(j,\sequence{V}))=(i(1),\ldots,i(j))$, if we have $\pnt_{\tree(\sequence{V})}(i(r,\sequence{V})) \not\in \tree_{\rho(x)}(\sequence{V})$ 
then either $F(r)$ occurs or else $E(r,\ell)$ occurs for some $\ell \in \{q+1,\ldots,r-1\}$, so \eqref{eq:jump_induction} and \eqref{eq:nbr_induction} together imply that 
\begin{align*}
& \probC{\pnt_{\randomtree}^b(v_{\rho_\randomtree(y)})
\not\in \randomtree_{\rho(x)}}
{(i_\randomtree(1),\ldots,i_\randomtree(j))=(i(1),\ldots,i(j)),
\pnt_{\randomtree}(i(r)) \not\in \randomtree_{\rho(x)}}\\
&=\probC{\pnt_{\randomtree}^b(i(r)) \not\in \randomtree_{\rho(x)}}
{(i_\randomtree(1),\ldots,i_\randomtree(j))=(i(1),\ldots,i(j)),
\pnt_{\randomtree}(i(r)) \not\in \randomtree_{\rho(x)}}\\
& \le 
\pran{1-\frac{x}{n-2}}^{b-1}\, .
\end{align*}
Summing over possible values for $(i_\randomtree(1),\ldots,i_\randomtree(j))$ yields that 
\[
\probC{\pnt_{\randomtree}^b(v_{\rho_\randomtree(y)})
\not\in \randomtree_{\rho(x)}}{\pnt_{\randomtree}(i(r)) \not\in \randomtree_{\rho(x)}}\le 
\pran{1-\frac{x}{n-2}}^{b-1}\, ,
\]
which combined with the bound for the case $b=1$ completes the proof. \end{proof}

\begin{proof}[Proof of Lemma~\ref{lem:degree_one_stretch}]
	The bound is obvious for $n_1=n-1$ so we assume $n_1 < n-1$. In that case, $\randomtree'$ has height at least $1$ deterministically, so for $0< h<1$ the first term in the upper bound equals $1$ and that is an upper bound for any probability. Therefore, we will from here on assume that $h\geq 1$. We begin with some fairly elementary combinatorics. Given a set $S=\{s_1,\ldots,s_m\}$, by a {\em labeled composition of $S$ with $k$ parts} we mean a weak composition $m_1,\ldots,m_k$ of $m$ together with a permutation $(s_{\sigma(1)},\ldots,s_{\sigma(m)})$ of $S$. 
	There are $(m+k-1)!/(k-1)!$ labeled compositions of $S$ with $k$ parts. 
	
	We now return to the setting of the lemma. By permuting the entries of $\dseq$, we may assume that the $n_1$ entries which equal $1$ appear at the end; this does not affect the law of the height of $\randomtree$.  In this case, we have $\dseq'=(d_1,\ldots,d_{n-n_1})$. 
	
	Next, given a tree $\tree\in \cT_{\dseq}$, let $\tree^- \in \cT_{\dseq'}$ be obtained from $\tree$ by suppressing all degree-one vertices. More precisely, $\tree^-$ is formed from $\tree$ by replacing any path $\gamma$ in $\tree$ all of whose internal vertices have exactly one child by a single edge. If this results in the root having degree $1$, we remove the root and its adjacent edge and let its only child be the new root; see Figure~\ref{fig:degone}. 
	
	\begin{figure}[htb]
		\begin{centering}
			\includegraphics[width=0.5\textwidth]{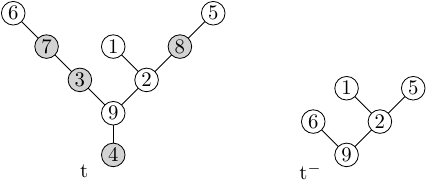}
		\end{centering}
		\caption{Left: a tree $\tree$. Right: the tree $\tree^-$ obtained from $\tree$ by suppressing degree-one vertices. Considering the vertices in the order $(1,2,5,6,9)$, the corresponding labeled composition of the set $\{3,4,7,8\}$ of degree-one vertices is $(),(),(8),(7,3),(4)$.}
		\label{fig:degone}
	\end{figure}
	
	For each tree $\tree' \in \cT_{\dseq'}$, there are 
	\[
	\frac{(n-1)!}{(n-n_1-1)!} 
	\]
	trees $\tree \in \cT_{\dseq}$ with $\tree^-=\tree'$.   This is the case since each such tree $\tree$ is uniquely determined by $\tree'$ together with a labeled composition of the $n_1$ degree-one vertices $\{n-n_1+1,\ldots,n\}$ into $n-n_1$ parts, as follows (see Figure~\ref{fig:degone}). Fix an ordering of the vertices of $\tree'$ as $v_1,\ldots,v_{n-n_1}$. Then for each $i \in [n-n_1]$, assign the vertices of the $i$'th part of the composition to $v_i$'s ancestral edge, in the same order they appear in the composition. (When $v_i$ is the root, these vertices are simply attached as ancestors of $v_i$.)
	
	It follows that if $\randomtree \in_u \cT_{\dseq}$ then $\randomtree^- \in_u \cT_{\dseq'}$. Moreover, 
	given that $\randomtree^-=\tree'$, the conditional distribution of $\randomtree$ may be described as follows. Choose any fixed order of the vertices of $\tree'$. Then choose a uniformly random labeled composition of $\{n-n_1+1,\ldots,n\}$ with $n-n_1$ parts, and then assign degree-one vertices to the ancestral edges of vertices of $\tree'$ as specified by the composition, as above. 
	
	For any {\em root-to-leaf path} $U=(u_1,\ldots,u_m)$ in $\tree'$, it follows from the above construction that conditionally given $\randomtree^-=\tree'$, 
	the total number of vertices lying along the corresponding root-to-leaf path in $\randomtree$ has the same distribution as the random variable $X_m$ described in the following experiment. 
	 Consider an urn with $n-1-n_1$ white balls and $n_1$ black balls. Repeatedly sample without replacement from the urn until it is empty. For $0\le i \le n-1$ write $W_i$ for the total number of white balls drawn after the $i$'th sample, and set $W_{n}=n-n_1$. Then let $X_m=\min(i: W_i=m)$. For $0 \le i \le n-1$ let $W_i^*$ be Binomial$(i,(n-1-n_1)/(n-1))$; so $W_i^*$ can be thought of as the number of white balls drawn after $i$ samples {\em with} replacement from the urn of the previous paragraph. By \cite[Proposition~20.6]{MR883646}, $W_i^*$ is a dilation of $W_i$, which is to say that there is a coupling $(W,W^*)$ of $W_i$ and $W_i^*$ so that $\e(W^*|W)=W$. This implies that $\e(\phi(W)) \le \e(\phi(W^*))$ for all convex functions of $\phi:\R \to \R$. In particular, concentration inequalities for $W_i^*$ which are proved by bounding exponential moments $\e(e^{\lambda W})$ of $W$ apply without change to $W_i$. By a Chernoff bound (see, e.g., \cite[Theorem 2.3c]{mcdiarmid98concentration}, applied with $\epsilon=1/2$) it follows that 
	that for $x \ge 2$, 
	\begin{align*}
		\p(X_m \ge  x mn/(n-n_1)) 
		& \le \p(W_{\lceil x mn/(n-n_1) \rceil} \le m) \\
		& \le \p(W_{\lceil x mn/(n-n_1) \rceil} \le \e W_{\lceil x mn/(n-n_1) \rceil}/2)\\
		& \le e^{-\tfrac{1}{8} xm}
	\end{align*}
This bound implies that for 
	$y \ge 2mn/(n-n_1)$, 
	\[
	\p(X_m \ge  y) \le e^{-y (n-n_1)/(8n)}\, .
	\]
Write $r'$ for the root of $\tree'$ and $r$ for the root of $\randomtree$. 
	Since the number of vertices on a fixed root-to-leaf path is one larger than the distance of that leaf from the root, it follows that for any leaf $v$ of $\tree'$, 
	for any $y \ge 2(\dist_{\tree'}(v,r')+1)n/(n-n_1)$,	
	\[
	\probC{\dist_{\randomtree}(v,r)\ge  y}{\randomtree^-=\tree'}
	\le e^{-(y+1)(n-n_1)/(8n)}.
	\]
If $\tree'$ has height at most $ h $ then since $h \ge 1$ we always have $2(\dist_{\tree'}(v,r')+1)n/(n-n_1) \le 4hn/(n-n_1)$, so if $\tree'$ has $n_0$ leaves and height at most $ h $ then, applying the above inequality to all root-to-leaf paths in $\tree'$, a union bound yields that for any $y \geq 4hn/(n-n_1)$, 
		\[
		\p(\height(\randomtree') \ge  y~|~\randomtree^-=\tree') \le n_0 e^{-(y+1)(n-n_1)/(8n)}< n_0 e^{-y(n-n_1)/(8n)}\, .
		\]
		Finally, since $\randomtree^- \eqdist \randomtree' \in_u\cT_{\dseq'}$, it follows that 
		\begin{align*}
			\p(\height(\randomtree)\ge y) 
			& \le  
			\p(\height(\randomtree^-) >  h) 
			+ \sum_{\{\tree' \in \cT_{\dseq'}: \height(\tree') \leq  h \}} \p(\height(\randomtree) \ge  y~|~
			\randomtree^-=\tree')\p(\randomtree^-=\tree')\\
			& \le \p(\height(\randomtree') >  h) + n_0 e^{-y(n-n_1)/(8n)}\, . \qquad\qquad\qquad\qquad\qquad\qquad\qquad\qedhere
		\end{align*}
\end{proof}

The proof of Proposition~\ref{cor.firststep} requires an auxiliary lemma, which itself has an auxiliary lemma.\footnote{``So, Nat'ralists observe, a Flea/Hath smaller Fleas that on him prey/And these have smaller yet to bite 'em/And so proceed {\em ad infinitum}.'' From {\em On Poetry: A Rapsody}, Johnathan Swift, 1733.} We first state and prove these two lemmas, then proceed to the proof of Proposition~\ref{cor.firststep}. 

The auxiliary auxiliary lemma, which is proved by a Chernoff bound-type argument, is similar to \cite[Lemma~15]{blanc}. 
\begin{lem}\label{lem:exp_sumbound}
	Fix positive constants $x_1,\ldots,x_m$ and 
	let $E_1,\ldots,E_m$ be independent with $E_m\sim \mathrm{Exp}(x_i)$. Fix $t > 0$ and write $S=x_1\I{E_1\le t}+\ldots+x_m\I{E_m \le t}$. Then 
	\[
	\p(S \le \e S/2) \le \exp(-t\e S/4). 
	\]
\end{lem}
\begin{proof} 
	For readability we take $t=1$, and explain how to handle general $t>0$ at the end of the proof. 
	Let $X_i=x_i \I{E_i \le 1}$, so $\e X_i = x_i(1-e^{-x_i})$ and 
	\[
	\overline{X}_i := X_i-\e X_i = 
	x_ie^{-x_i}-x_i\I{E_i > 1}\, ,
	\] 
	so that $\overline{S}:=S-\e S = \sum_{i=1}^m \overline{X}_i$. 
	An easy computation gives 
	\[
	\E{e^{-\overline{X}_i}}
	=
	e^{-x_ie^{-x_i}} \E{e^{x_i \I{E_i > 1}}}
	= 
	e^{-x_ie^{-x_i}}(2-e^{-x_i})
	\, ,
	\]
	and a delicate but elementary computation then shows that 
	this quantity is at most $\exp\left(x_i(1-e^{-x_i})/4\right)=\exp(\e X_i/4)$. 
	Markov's inequality then gives 
	\begin{align*}
		\p(S \le \e S/2)
		\le 
		\e e^{-\overline{S}} e^{-\e S/2} 
		= \prod_{i=1}^m
		\e e^{-\overline{X}_i}e^{-\e X_i/2}
		\le \prod_{i=1}^m e^{-\e X_i/4} = e^{-\e S/4}.
	\end{align*}
	This proves the lemma in the case $t=1$. For general $t$, the proof works identically by instead showing that 
	$\e e^{-t \overline{X}_i} \le e^{ t\e X_i/4}$.
\end{proof}
We use Lemma~\ref{lem:exp_sumbound} to prove a lower tail bound for sums of the form $\sum_{1 \le j \le k} (d_{i_{\randomtree}(j)}-1)$. 
\begin{lem}\label{lem.firststephighvariance}
Fix a degree sequence $\dseq=(d_1,\ldots,d_n)$ and let $\randomtree \in_u \cT_{\dseq}$. 
Then 
taking $c=(1-e^{-2})/8$, 
for all $t \in [0,1]$, with $\sigma=\sigma_\dseq=(n^{-1}\sum_{j \in [n]} d_j(d_j-1))^{1/2}$, we have 
\[
\p\Big(
\sum_{1 \le j \le nt} (d_{i_{\randomtree}(j)}-1)
\le \frac{c\sigma^2nt}{1+\sigma n^{1/2}t}
\Big) \le 
\exp(-\tfrac{3}{16}tn)
+
\exp\pran{-
\frac c4\frac{n \sigma^2t^2}{1+\sigma n^{1/2}t}}. 
\]
\end{lem}
\begin{proof}
We assume $\dseq$ is compressed, so that $m:=|\{i \in [n]: d_i \ne 0\}|=\max(i\in [n]: d_i \ne 0)$ is the number of non-leaf vertices of $\randomtree$. 
Let $E_1,\ldots,E_m$ be independent with $E_i \sim \mathrm{Exp}(d_i)$, and let $(I_1,\ldots,I_m)$ be the permutation of $(1,\dots,m)$ for which $E_{I_1} < \ldots < E_{I_m}$. Then $d_{I_1},\ldots,d_{I_m}$ is a size-biased permutation of $(d_1,\ldots,d_m)$. In other words, $(I_1,\ldots,I_m)$ has the same distribution as $(i_\randomtree(1),\ldots,i_\randomtree(m))$ for $\randomtree \in_u \cT_{\dseq}$. For the remainder of the proof we may therefore assume that $\randomtree$ and $E_1,\ldots,E_m$ are coupled so that 
\[
(I_1,\ldots,I_m) = (i_\randomtree(1),\ldots,i_\randomtree(m)). 
\]

Now let $s=t/2$ and define 
\begin{align*}
N_s & := |\{i \in [m] :E_i \le s\}| = \sum_{i\in[m]} \I{E_i \le s}=\max(j: E_{I_j}\le s)\, , \\
X_s & := \sum_{i\in [m]} (d_i-1)\I{E_i \le s} = \sum_{j\in[N_s]} (d_{I_j} - 1) 
\end{align*}
and 
 \[
 M= \frac{1-e^{-2}}{4} \frac{n \sigma^2s}{1+\sigma  n^{1/2}s} > 
 \frac{1-e^{-2}}{8} \frac{n \sigma^2t}{1+\sigma  n^{1/2}t}
 =\frac{cn \sigma^2t}{1+\sigma  n^{1/2}t}
 \]
Under the above coupling, 
\[
\sum_{1 \le j \le nt} (d_{i_\randomtree(j)}-1)
=\sum_{1 \le j \le nt} (d_{I_j}-1). 
\]
It follows that if $\sum_{1 \le j \le nt}(d_{i_\randomtree(j)}-1) \le M$ then either $nt \le N_s$ or $X_s \le M$, so 
\begin{align}\label{eq:coupling_bound}
\p\Big(
\sum_{1 \le j \le nt} (d_{i_{\randomtree}(j)}-1)
\le \frac{c\sigma^2nt}{1+\sigma n^{1/2}t}
\Big) & 
\le \p\Big(\sum_{1 \le j \le nt} (d_{i_\randomtree(j)}-1) \le M\Big) \nonumber\\
& \le \p(N_s \ge nt) + \p(X_s \le M),
\end{align}
where in the first line we have used our lower bound on $M$. 

To make use of this inequality, we first bound 
$\e X_s = \sum_{i\in [m]} (d_i-1) (1-e^{-d_i s})$ 
from below. We consider two cases, depending on whether $\sigma^2=n^{-1}\sum_{i \in [m]} d_i (d_i-1)$ is dominated by the contributions from small or large summands. 

Fix $r \in (0,1)$. First suppose that $\sum_{i \in [m]: d_is \le 2} d_i(d_i-1) \ge rn\sigma^2$. 
For $d_is \le 2$ we have $(1-e^{-d_is}) \ge (1-e^{-2}) d_is/2$, so 
\[
\sum_{i \in [m]: d_is \le 2} (d_i-1) (1-e^{-d_is})
\ge (1-e^{-2})\frac s2 
\sum_{i \in [m]: d_is \le 2} d_i(d_i-1) \ge 
(1-e^{-2})\frac s2 rn \sigma^2\, .
\]

Next suppose that 
$\sum_{i \in [m]: d_is \le 2} d_i(d_i-1) < rn\sigma^2$. 
For $i$ such that $d_is > 2$ we have $1-e^{-d_i s} > 1-e^{-2}$ and $(d_i-1)/d_i > (2-s)/2$, so using that $\sum_i x_i > (\sum_i x_i^2)^{1/2}$ for non-negative $x_i$, we obtain 
\begin{align*}
\sum_{i \in [m]: d_is > 2} (d_i-1) (1-e^{-d_is})
& \ge 
(1-e^{-2})
\Big(\sum_{i \in [m]: d_is > 2} (d_i-1)^2\Big)^{1/2}\\
& \ge 
(1-e^{-2})
\Big(\frac{2-s}{2}\sum_{i \in [m]: d_is > 2} d_i(d_i-1)\Big)^{1/2} \\
& > (1-e^{-2})
\frac{2-s}{2}(1-r) n^{1/2} \sigma \, .
\end{align*}
Taking $r=(2-s)/(2-s+\sigma  n^{1/2} s)$ makes the lower bounds on $\sum_{i \in [m]: d_is \le 2} d_i(d_i-1)$ and on 
$\sum_{i \in [m]: d_is > 2} d_i(d_i-1)$ in the two cases equal, and yields the bound
\[
\e X_s=\sum_{i \in [m]} (d_i-1)(1-e^{-d_is})\ge \frac{1-e^{-2}}2 \frac{(2-s)n \sigma^2s}{2-s+\sigma  n^{1/2}s} 
\ge \frac{1-e^{-2}}2\frac{n \sigma^2s}{1+\sigma  n^{1/2}s}  =2M\, ,
\]
the last inequality holding since $s =t/2\le 1/2$. 
By Lemma~\ref{lem:exp_sumbound} it follows that 
\begin{equation}\label{eq:xs_lower}
\p(X_s \le M) 
\le 
\p(X_s \le \e X_s/2)
\le \exp(- s\e X_s/4) \le 
\exp(-tM/4)\, .
\end{equation}
Next, note that
\[
\e N_s = \sum_{i \in [m]} \p(E_i \le s)
=\sum_{i \in [m]} (1-e^{-d_i s})
\le \sum_{i \in [m]} d_i s = s(n-1)<sn\, .
\]
Since $N_s$ is a sum of $[0,1]$-valued random variables, Bernstein's inequality \cite{Bernstein1927,Craig1933} then gives
\begin{equation}\label{eq:ns_upper}
\p(N_s \ge tn) = \p(N_s \ge 2sn) \le \exp(-\tfrac{3}{8}sn)
=\exp(-\tfrac{3}{16}tn)
\, .
\end{equation}

Using the bounds (\ref{eq:xs_lower}) and (\ref{eq:ns_upper}) in 
 (\ref{eq:coupling_bound}) 
 now yields
\begin{align*}
\p\Big(
\sum_{1 \le j \le nt} (d_{i_{\randomtree}(j)}-1)
\le \frac{c\sigma^2nt}{1+\sigma n^{1/2}t}
\Big)
& \le \exp(-\tfrac{3}{16}tn)
+
\exp(-tM/4).
\end{align*}
In view of the lower bound on $M$, this proves the bound claimed in the lemma.
\end{proof}

\begin{proof}[Proof of Proposition~\ref{cor.firststep}]
First recall that the non-leaf vertices of $\randomtree_{\rho(\alpha\sigma n^{1/2})}$ are precisely $i(1),\dots, i(k)$, with $k$ minimal such that $\sum_{j=1}^k(d_{i(j)}-1)\geq \alpha\sigma n^{1/2}$. Now, if $\height\left(\randomtree_{\rho(\alpha\sigma n^{1/2})} \right)>\tfrac{bn^{1/2}}{2\sigma}$, then 
	$\randomtree_{\rho(\alpha\sigma n^{1/2})}$ must contain more than $\tfrac{1}{2}bn^{1/2}/\sigma $ non-leaf vertices. This implies that $k>\tfrac{1}{2}bn^{1/2}/\sigma$, which by the definition of $k$ yields that 
	\[\sum_{1\leq j \leq \tfrac{1}{2}bn^{1/2}/\sigma} \left(d_{i_{\randomtree}(j)}-1\right) < \alpha\sigma n^{1/2}\, .\] 
	It follows from this inclusion of events that 
	\[
\p\left( 
\height\left(\randomtree_{\rho(\alpha\sigma n^{1/2})} \right)>\tfrac{bn^{1/2}}{2\sigma}
\right)
\le
\p\left(
	\sum_{1\leq j \leq \tfrac{1}{2}bn^{1/2}/\sigma} \left(d_{i_{\randomtree}(j)}-1\right) < \alpha\sigma n^{1/2}.
\right)\, .
	\]
	
	Now fix $b \ge 1$ and set $t=b/(2\sigma n^{1/2})$ so that 
	$tn=bn^{1/2}/(2\sigma)$. Then with $c=(1-e^{-2})/8$ as in the statement of Lemma \ref{lem.firststephighvariance}, we have $c=3\alpha$, so since $b\geq 1$, 
	\[
	\alpha\sigma n^{1/2}=\frac{c\sigma n^{1/2}}{3} \leq c\sigma n^{1/2}\frac{b/2}{1+b/2}=\frac{c \sigma^2 nt}{1+\sigma n^{1/2}t}\, .
	\]
	It follows that 
	\[
	\p\left(\sum_{1\leq j \leq \tfrac{1}{2}bn^{1/2}/\sigma} \left(d_{i_{\randomtree}(j)}-1\right) < \alpha\sigma n^{1/2} \right)\leq \p\left(
	\sum_{1 \le j \le nt} (d_{i_{\randomtree}(j)}-1)
	\le \frac{c \sigma^2nt}{1+\sigma n^{1/2}t}
	\right)\, ,
	\]
and Lemma \ref{lem.firststephighvariance} now yields the result.
\end{proof}

	\section{\bf Bienaym\'e trees, simply generated trees, and random block-stable graphs}\label{sec.bienandsimpygen}

In this section we use Theorems~\ref{thm.gaussiantails} and~\ref{thm.infinitevariance} to prove Theorems~\ref{thm.bieninfinitevariance},~\ref{thm.biensubcriticalnoexpdecay} and~\ref{thm.biengaussiantails}, and additionally to prove 
a conjecture from \cite{McDiarmidScott} on random block-stable graphs, as well as
two conjectures from \cite{janson_simpygen} on \emph{simply generated trees}. The latter class of trees is defined as follows. Fix non-negative real weights $\weight=(w_k,k\geq 0)$ with $w_0>0$. Given a finite plane tree $\tree$, we define the weight of $\tree$ to be 
\[\weight(\tree)=\prod_{v\in v(\tree)} w_{d_{\tree}(v)}.
\]
Next, for positive integers $n$, let
\[Z_n=Z_n(\weight)=\sum_{\{\text{plane trees }\tree:|v(\tree)|=n\}} \weight(\tree).
\]
If $Z_n >0$, then we define a random tree $\randomtree_{\weight,n}$ by setting
\[\p(\randomtree_{\weight,n}=\tree)=\frac{\weight(\tree)}{Z_n} \]
for plane trees $\tree$ with $|v(\tree)|=n$. The random tree $\randomtree_{\weight,n}$ is called a simply generated tree of size $n$ with weight sequence $\weight$. 
Note that if $\sum_{k\geq 0} w_k=1$, then $\weight$ is an offspring distribution, and $\randomtree_{\weight,n}$ is indeed distributed as a Bienaym\'e tree with offspring distribution $\weight$, conditioned to have size $n$ -- so this notation agrees with (but generalizes) that from earlier in the paper. From this we also see that conditioned Bienaym\'e trees are a subclass of simply generated trees; we will use this fact below.

Let $\Phi(z)=\Phi_{\weight}(z)=\sum_{k\geq 0} w_k z^k$ be the generating function of the weight sequence $\weight$,  and let $\rho=\rho_\weight=\sup(s\ge 0: \Phi(s)<\infty)$ be its radius of convergence. For $s>0$ such that $\Phi(s)<\infty$, we let
\[
\Psi(s)=\Psi_\weight(s)=\frac{s\Phi'(s)}{\Phi(s)}=\frac{\sum_{k\geq 0}kw_ks^k}{\sum_{k\geq 0}w_ks^k}.
\]
If $\Phi(\rho)=\infty$ then also define 
\[
\Psi(\rho)=\Psi_\weight(\rho)=\lim_{s\uparrow\rho}\Psi(s); 
\]
by Lemma 3.1(i) in \cite{janson_simpygen}, $\Psi$ is strictly increasing on $[0,\rho)$ so this limit exists. 
Let $\nu=\nu_{\weight}=\Psi_\weight(\rho)$, and define
\begin{equation}\label{eq:tau_def}
\tau=\tau_{\weight} := 
\begin{cases}
\rho & \mbox{ if } \nu_{\weight} \le 1 \\
\Psi^{-1}(1) & \mbox{ if } \nu_{\weight} > 1\, .
\end{cases}
\end{equation}
Note that if $\nu_{\weight} > 1$ then $\tau \in [0,\rho)$.
Define $\hat{\sigma}^2=\hat{\sigma}_{\weight}^2=\tau\Psi'(\tau)$. 
For later use, we note that if $\sequence{w}=(w_k,k \ge 0)$ is a probability distribution with finite mean, so $\sum_{k \ge 0} w_k=1$ and $|\sequence{w}|_1=\sum_{k \ge 0}k w_k <\infty$, then we always have $\rho \ge 1$, so 
\begin{align}\label{eq:psi_formulas}
\begin{split}
\nu & = \Psi_{\sequence{w}}(\rho) \ge 
\Psi_{\sequence{w}}(1) = \sum_{k\ge 0} kw_k = |\sequence{w}|_1, \mbox{ and }\\
\Psi'_{\sequence{w}}(1) & = \sum_{k \ge 0} k^2w_k-\left( \sum_{k\ge 0} kw_k \right)^2  = |\sequence{w}|_2-|\sequence{w}|^2_1\, .
\end{split}
\end{align}

The following theorem resolves Conjecture 21.5 and Problem 21.7 from \cite{janson_simpygen}. 
\begin{thm}\label{thm.simplygen}
	Let $\weight=(w_k,k\geq 0)$ be a weight sequence with $w_0>0$ and with $w_k>0$ for some $k\geq 2$. 
	If either $\hat{\sigma}^2=\infty$ or $\nu < 1$ then $n^{-1/2}\height(\randomtree_{\weight,n})\to0$,
	where the convergence is in both probability and expectation, as $n\to \infty$ along integers $n$ such that $Z_n(\weight)>0$.
\end{thm}
We also prove the following, more quantitative theorem, which resolves Problem 21.9 from \cite{janson_simpygen}.
\begin{thm}\label{thm.simplygengaussian}
	Fix $\epsilon \in [0,1)$. Then, there exist constants $c_1=c_1(\epsilon)$ and $c_2=c_2(\epsilon)$ such that the following holds. Let $\weight=(w_k,k\geq 0)$ be a weight sequence with $w_0>0$ and $w_k>0$ for some $k\geq 2$ and $\tfrac{w_1\tau}{\Phi(\tau)}<1-\epsilon$. For any $t>1$ and any $n$ large enough with $Z_n(\weight)>0$ and
	\[\p\left( \height(\randomtree_{\weight,n})\ge  tn^{1/2}\right)< c_1\exp(-c_2t^2).\]
\end{thm}
In fact, our result is stronger than what is conjectured in Problem 21.9 from \cite{janson_simpygen}, because our parameters $c_1$ and $c_2$ only depend on $\weight$ through $\tfrac{w_1\tau}{\Phi(\tau)}$. (Note that Problem 21.9 does not exclude weight sequences $\weight$ with $w_k=0$ for all $k\geq 2$, but for such sequences either $Z_n(\weight)=0$ for all $n\geq 2$ or $ \height(\randomtree_{\weight,n})=n-1$ deterministically.)

This theorem has the following immediate corollary. 

\begin{cor}\label{cor.expectedheightsimplygen}
	For any weight sequence $\weight$ on $\N$ with $w_0>0$ and $w_k>0$ for some $k\geq 2$,  $\e[\height(\randomtree_{\weight,n})]=O(n^{1/2})$, and, more generally, for any fixed $r<\infty$, $\e[\height(\randomtree_{\weight,n})^r]=O(n^{r/2})$ as $n\to\infty$ over all $n$ such that $\p(|\randomtree_\mu|=n)>0$. 
	\qed
\end{cor}

In this section we will make use of the following notation. For $\weight$ a weight sequence, let $\sequence{D}=\sequence{D}(\weight,n)=(D_1,\dots,D_n)$ be a random degree sequence with the following law. Let $\randomtree$ be a simply generated tree of size $n$ with weight sequence $\weight$. Conditionally given $\randomtree$, let $\hat{\randomtree}$ be the random tree obtained as follows: label the vertices of $\randomtree$ by a uniformly random permutation of $[n]$, then forget about the plane structure. Let $D_i$ be equal to the degree of $i$ in $\hat{\randomtree}$.  Then, for $k\in \N$, let $N_k=N_k(\weight,n)=|\{i \in [n]: D_i=k\}|$ be the number of entries of $\sequence{D}$ which equal $k$, so $N_k=n_k(\sequence{D}(\sequence{w},n))$ and $\sum_{k=0}^\infty N_k=n$. 

Our proofs will exploit two distributional identities. 
The first is the following. Let 
$\randomtree_{\sequence{D}(\weight,n)} \in_u \cT_{\sequence{D}(\weight,n)}$, by which we mean that, conditionally given that 
$\sequence{D}(\weight,n)=\dseq$, we have  
$\randomtree_{\sequence{D}(\weight,n)} \in_u\cT_{\dseq}$. Then 
$\hat{\randomtree}\overset{d}{=} \randomtree_{\sequence{D}(\weight,n)}$ by Proposition \ref{prop:planetrees}, so $\height(\randomtree_{\weight,n})\overset{d}{=}\height(\randomtree_{\sequence{D}(\weight,n)})$.
The second is the fact that for any weight sequence $\sequence{w}=(w_k,k \ge 0)$ and any constants $a,b>0$, the weight sequence $\hat{\sequence{w}}$ with $\hat{w}_k = ab^k w_k$ is {\em equivalent} to $\sequence{w}$, i.e., it satisfies that $\randomtree_{\sequence{w},n} \eqdist\randomtree_{\hat{\sequence{w}},n}$ for all $n$ for which either (and thus both) of the random trees are defined. This is an immediate consequence of the formula \eqref{eq:twn_ident} for the distribution of $\randomtree_{\sequence{w},n}$.

These two equalities in distribution imply that if we obtain good control over the asymptotic behaviour of $(N_k(\hat{w},n),k \ge 0)$ for some weight sequence $\hat{w}$ which is equivalent to $\sequence{w}$, then we can use Theorems \ref{thm.gaussiantails} and \ref{thm.infinitevariance}, on the heights of trees with given degree sequences, to prove tail bounds for $\height(\randomtree_{\weight,n})$. 
To obtain such control, we rely on the following result.

\begin{thm}[\cite{janson_simpygen}, Theorem 11.4; \cite{addarioberry2021universal}, Theorem 5.2]\label{thm.janson}
	Let $\weight=(w_k,k\geq 0)$ be a weight sequence with $w_0>0$ and $w_k>0$ for some $k\geq 2$. Write $\tau=\tau_{\sequence{w}}$ and define  \[ \hat{\mu}_k=\hat{\mu}_k(\weight)=\frac{w_k\tau^k}{\Phi(\tau)}\] for $k\in \N$. Then $\hat{\mu}=(\hat{\mu}_k,k\geq 0)$ is a probability distribution with mean $m=1\wedge \nu$ and variance $\hat{\sigma}^2$. For $k \ge 0$ define $N_k=n_k(\sequence{D}(\sequence{w},n))$, as above. Then for every $\eps > 0$ there exists $c=c(\eps)>0$ that does not depend on $\weight$ such that for all $n$ sufficiently large, for every 
	 integer $k\geq 0$, 
	\[\p\left(\left|\frac{N_k(\weight,n)}{n}-\hat{\mu}_k\right|>\epsilon \right)<e^{-cn}.
	\]
\end{thm}
Theorem 11.4 of \cite{janson_simpygen} handles the case $\nu > 0$ in the above statement. 
The exponentially small error bounds stated above are not made explicit in the statement of \cite[Theorem 11.4]{janson_simpygen}, but are recorded in the course of its proof (see \cite[pages 163-164]{janson_simpygen}). The case $\nu=0$ was handled in \cite{addarioberry2021universal}. 

Before we prove Theorems \ref{thm.simplygen}, \ref{thm.simplygengaussian}, \ref{thm.bieninfinitevariance} and \ref{thm.biengaussiantails}, we  first illustrate that the requirements in Theorem \ref{thm.biengaussiantails}, that $1-\mu_0-\mu_1$ and $\mu_0/(\mu_0+\mu_1)$ are bounded from below, are necessary. To accomplish this we will consider probability distributions $\mu=\mu^{p,q}$ of the form 
\[
\mu_0 = q(1-p)\qquad
\mu_1 = (1-q)(1-p)\qquad
\mu_2 = p\, ,
\]
for $p,q\in (0,1)$. 
\begin{claim}\label{claim:smalldegreesyieldtalltrees}
	For any $x > 0$ and any $q>0$, $\lim_{p\downarrow0}\liminf_{n \to \infty} \p(\height(\randomtree_{\mu^{p,q},n})> xn^{1/2})=1$; and for any $x > 0$ and $p>0$, $ \lim_{q\downarrow0}\liminf_{n \to \infty} \p(\height(\randomtree_{\mu^{p,q},n})> xn^{1/2})=1$
\end{claim}
\begin{proof}
We apply Theorem~\ref{thm.janson} with $\sequence{w}=\sequence{w}^{p,q}=\mu^{p,q}$.  Elementary computation shows that the probability distribution $\hat{\mu}=\hat{\mu}^{p,q}$ from 
	Theorem~\ref{thm.janson} is given by 
	\[
	\hat{\mu}_1=\hat{\mu}_1^{p,q}=\frac{(1-q)\sqrt{1-p}}{(1-q)\sqrt{1-p}+2\sqrt{p}\sqrt{q}}
	\]
	and $\hat{\mu}_0=\hat{\mu}_2=(1-\hat{\mu}_1)/2$. 
	Since $\hat{\mu}$ is equivalent to $\mu$, it follows that that $T_{\mu^{p,q},n} \eqdist T_{\hat{\mu}^{p,q},n}$ for all $n$. Write $\sigma^{p,q}=(|\hat{\mu}^{p,q}|^2_2-1)^{1/2}=(1-\hat{\mu}^{p,q}_1)^{1/2}$ for the standard deviation of $\hat{\mu}^{p,q}$. Using this equivalence in distribution together with the convergence of the search-depth process for large critical random Bienaym\'e trees \cite[Theorem~23]{MR1207226}, it follows that 
	\[
	\frac{2}{n^{1/2}} \height(T_{\mu^{p,q}},n)\eqdist \frac{2}{n^{1/2}} \height(T_{\hat{\mu}^{p,q}},n) \convdist \frac{1}{\sigma^{p,q}}\sup_{0 \le t \le 1} B(t)\, ,
	\]
	where $B$ is a standard Brownian excursion. Finally, since $\hat{\mu}^{p,q}_1 \to 1$ (and thus $\sigma^{p,q} \to 0$) as either $p \to 0$ or $q \to 0$, and also $\p(\sup_{0 \le t \le 1} B(t)>0)=1$, it follows that 
$(\sigma^{p,q})^{-1} \sup_{0 \le t \le 1} B(t)\to \infty$ in probability as either $p \to 0$ or $q \to 0$. The result follows. 
\end{proof}

In our proof of Theorem \ref{thm.simplygen}, we make use of the following consequence of Theorem \ref{thm.janson}.

\begin{cor}\label{cor.simplygen}
	If $\nu<1$, or if $\hat{\sigma}^2=\infty$, then for each $C>0$ there exists $c=c(\weight,C)>0$ such that for all $n$ sufficiently large, 
	\[\p\left(\sigma_{\sequence{D}(\weight,n)}\leq C\right) \leq e^{-cn}\, .\]
\end{cor}
\begin{proof}
	First suppose that $\nu<1$; in this case $\sum_{i=1}^\infty k \hat{\mu}_k=\nu$. Let $n\geq \tfrac{4}{1-\nu}$. Suppose without loss of generality that $C\geq1$ and set $K=\lceil \tfrac{4C^2}{1-\nu}\rceil$. Then, $\sum_{k=1}^{K}k\hat{\mu}_k\leq \nu,$ so by Theorem  \ref{thm.janson}, there is $c=c(\weight,C)$ such that with $N_k=n_k(\sequence{D}(\sequence{w},n))$, 
	\[ 
	\p\left( \sum_{k=0}^{K} k N_k > \frac{1+\nu}{2} n \right)\leq e^{-cn}
	\]
	for all $n$ sufficiently large.
	But on the event that $\sum_{k=0}^{K} k N_k\leq  \tfrac{1+\nu}{2} n$, using that $\sum_{k=1}^\infty kN_k=n-1$, we have $\sum_{k=K+1}^\infty kN_k\geq \tfrac{1-\nu}{2}n-1\geq \tfrac{1-\nu}{4}n$ by our assumption on $n$. This implies that
	\[\sum_{i=1}^n D_i(D_i-1)\geq \tfrac{1-\nu}{4}Kn\geq C^2 n,\]
 so $\sigma_{\sequence{D}(\weight,n)} > C$. 
	This proves the claim in the case that $\nu < 1$. 
	
	Next suppose that $\hat{\sigma}^2=\infty$. In this case there exists $K\in \N$ such that $\sum_{k=0}^K\hat{\mu}_k k(k-1)>2C$. 
	Since $\sigma_{\sequence{D}(\weight,n)} = n^{-1}\sum_{i=1}^n D_i(D_i-1)\geq n^{-1}\sum_{k=0}^{K} N_k k(k-1)$, Theorem \ref{thm.janson} then implies that there exists $c=c(\weight,C)$ such that 
	\[
	\p\left(\sigma_{\sequence{D}(\weight,n)} \le C\right)\le 
	\p\left( \sum_{k=0}^{K} N_k k(k-1) \leq Cn \right) \leq e^{-cn}\]
	for all $n$ sufficiently large. This proves the claim in the case that $\hat{\sigma}^2 = \infty$.	
\end{proof}
In the next proof we write $x \vee y :=\max(x,y)$. 
\begin{proof}[Proof of Theorem \ref{thm.simplygen}]
	Fix $0<\epsilon\le2^{-{15}}$. Let $\weight$ be a weight sequence with $\nu=1$ and $\hat{\sigma}^2=\infty$, or with $\nu<1$, and fix $n$ such that $Z_n(\weight)>0$. Let $\hat{\mu}_1=\hat{\mu}_1(\weight)$ and let $K$ be large enough such that $2\log(k+1)/k<\epsilon^2$ and such that $\log(k+1) \ge (\tfrac12\log\tfrac{2}{(1-\hat{\mu}_1)})\vee 2^{14}$ 
	for all $k\geq K$. 
Then let 
	\[
	\cD_n = \Big\{\mbox{degree sequences }\sequence{d}=(d_1,\ldots,d_n): 
	n_1(\sequence{d}) \le \frac{n(1+\hat{\mu}_1)}{2},\sigma_{\sequence{d}}\geq K\Big\}.
	\]
	
	By Theorem \ref{thm.janson} and Corollary \ref{cor.simplygen}, there exists $c=c(\weight)>0$ such that for all $n$ sufficiently large, 
	\[ \p(\sequence{D}\not\in \cD_n)\leq e^{-cn}\, .\]
	Moreover, for any $\sequence{d} \in \cD_n$, with $(\sigma')^2 = \sigma_{\sequence{d}}^2n/(n-n_1(\sequence{d}))$ as in Theorem \ref{thm.infinitevariance}, we have $\log(\sigma'+1) \le \log(\sigma_{\sequence{d}}+1)+\tfrac12\log\tfrac{2}{(1-\hat{\mu}_1)}< 2\log(\sigma_{\sequence{d}}+1)$, so $\log(\sigma'+1) /\sigma_{\sequence{d}} \le \eps^2$. Also, $\sigma'\ge \sigma_{\sequence{d}}$, so $\log(\sigma'+1)\geq 2^{14}$.

	Fix any $t\geq \epsilon$. We apply Theorem \ref{thm.infinitevariance} with $x=\epsilon^{-2}t\ge2^{15}$ to obtain that for any $n$ and any $\sequence{d}\in \cD_n$,
	\begin{align*}
	\p\left(\height(\left.\randomtree_{\weight,n})>t n^{1/2}\right|\sequence{D}=\sequence{d}\right)
	& = 
	\p\left(\height(\left.\randomtree_{\weight,n})>x\eps^2 n^{1/2}\right|\sequence{D}=\sequence{d}\right)
	\\
	& \le 
	\p\left(\height(\left.\randomtree_{\weight,n})>x n^{1/2}\frac{\log(\sigma'+1)}{\sigma_{\sequence{d}}}\right|\sequence{D}=\sequence{d}\right)\\
	& = \p\left(\height(\randomtree_{\sequence{d}})> x n^{1/2}\frac{\log(\sigma'+1)}{\sigma_{\sequence{d}}}
	\right) \\
	& < 4\exp(- x\log(\sigma'+1) /2^{14})\le 4\exp(- \epsilon^{-2}t ).
	\end{align*}
		Since
	\[
\p\left(\left.\height(\randomtree_{\sequence{D}})>t n^{1/2}\right|\sequence{D}\in \cD_n\right)
=
\sum_{\dseq \in \cD_n} 
\p\left(\left.\height(\randomtree_n)>t n^{1/2}\right|\sequence{D}=\sequence{d}\right) \p(\sequence{D}=\sequence{d}|\sequence{D} \in \cD_n)\, ,
\]
	it follows that for all $n$ sufficiently large,
	\begin{align*}
	\e\left[\height(\randomtree_{\weight,n})\one_{\{\height(\randomtree_{\weight,n})>\epsilon n^{1/2}\}}\right]&\leq\e\left[\left.\height(\randomtree_{\sequence{D}})\one_{\{\height(\randomtree_{\sequence{D}})>\epsilon n^{1/2}\}}\right|\sequence{D}\in \cD_n  \right]+ n\p(\sequence{D}\not\in \cD_n)\\
		&\leq n^{1/2}\left[\epsilon+\int_\epsilon^\infty \p\left(\left.\height(\randomtree_{\sequence{D}})>tn^{1/2}\right| \sequence{D}\in \cD_n\right)dt\right] +o(1)\\
		&\leq n^{1/2}\left[\epsilon+4\int_\epsilon^\infty \exp(- \epsilon^{-2}t)dt\right] +o(1)\\
		&= n^{1/2}\left[\epsilon+ 4 \epsilon^2\exp(- \epsilon^{-1}) +o(1)\right]\, .
	\end{align*}
	We can pick $\epsilon$ arbitrary small, so the statement follows.
\end{proof}

	\begin{proof}[Proof of Theorem \ref{thm.simplygengaussian}]
		Let $\sequence{D}=\sequence{D}(\weight,n)$ and, as above, write $N_k=n_k(\sequence{D})$. Then, by Theorem \ref{thm.janson}, there is $c=c(\epsilon)>0$ such that for all $n$ sufficiently large,
		\[\p \left( 1-\frac{N_1}{n}< \epsilon/2\right)<e^{-c n}.
		\]
		Let $\cD_n$ be the set of degree sequences $\sequence{d}$ with $\p(\mathrm{D}=\mathrm{d})>0$  such that $1-n_1(\sequence{d})/n\geq \epsilon/2$; so $\p(\sequence{D} \not\in\cD_n) < e^{-cn}$. 
		By Theorem \ref{thm.gaussiantails},
		there are $c_1=c_1(\epsilon)$ and $c_2=c_2(\epsilon)$ such that for any $\sequence{d}\in \cD_n$, for all $t>1$, 
		\[\p\left(\left.\height(\randomtree_n)\ge t n^{1/2}\right|\sequence{D}=\sequence{d}\right)<c_1\exp\left(-c_2 t^{2}\right).\]
		The theorem now follows from the observation that 
		\[\p\left( \height(\randomtree_{\weight,n})\ge tn^{1/2}\right)\leq \p(\sequence{D}\not\in\cD_n)+\p\left(\left.\height(\randomtree_{\sequence{D}})\ge t n^{1/2}\right|\sequence{D}\in \cD_n\right) 
		\]
and the fact that 
\[
\p\left(\left.\height(\randomtree_{\sequence{D}})\ge t n^{1/2}\right|\sequence{D}\in \cD_n\right)
=
\sum_{\dseq \in \cD_n} 
\p\left(\left.\height(\randomtree_n)\ge t n^{1/2}\right|\sequence{D}=\sequence{d}\right) \p(\sequence{D}=\sequence{d}|\sequence{D} \in \cD_n)\, . \qedhere 
\]
	\end{proof}

	\begin{proof}[Proof of Theorems~\ref{thm.bieninfinitevariance} and~\ref{thm.biensubcriticalnoexpdecay}]
	We use the fact that conditioned Bienaym\'e trees are a special case of simply generated trees.
	
	First, if $|\mu|_1 \le 1$ and $|\mu|_2=\infty$ then $\mu_0>0$ and $\mu_k>0$ for some $k \ge 2$. Next, since $|\mu|_2=\infty$, for all $t > 0$ we have $\sum_{k\geq 0} e^{tk}\mu_k=\infty$. 
	This implies that the generating function $\Phi=\Phi_\mu$ has radius of convergence $\rho=\rho_\mu=1$, so $\nu_{\mu} = \Psi(\rho)=\Psi(1)=|\mu|_1 \le 1$. This implies that $\tau=\tau_{\mu}$ defined by \eqref{eq:tau_def} satisfies $\tau=\rho=1$ and so by \eqref{eq:psi_formulas} we have 
	\[
	\hat{\sigma}^2=\hat{\sigma}^2_{\mu}:= \tau \Psi'(\tau) = \Psi'(1) 	= |\mu|_2^2-|\mu|^2_1  = \infty. 
	\]
	Theorem~\ref{thm.simplygen} now implies that $n^{1/2}\height(\randomtree_{\mu,n}) \to 0$ in probability and expectation along integers $n$ such that $\p(|\randomtree_\mu|=n)>0$. This proves Theorem~\ref{thm.bieninfinitevariance}.
	
	Similarly, if $|\mu|_1 < 1$ and $\sum_{k\geq 0} e^{tk}\mu_k=\infty$ for all $t > 0$ then $\mu_0>0$ and $\mu_k>0$ for some $k \ge 2$. Moreover, $\Phi_\mu$ again has radius of convergence $\rho_\mu=1$, and so $\nu =\Psi(1)= |\mu|_1<1$. In this case Theorem~\ref{thm.simplygen} also implies that $n^{-1/2}\height(\randomtree_{\mu,n}) \to 0$ in probability and expectation along integers $n$ such that $\p(|\randomtree_\mu|=n)>0$. This proves Theorem~\ref{thm.biensubcriticalnoexpdecay}. 
	\end{proof}	
	\begin{proof}[Proof of Theorem \ref{thm.biengaussiantails}]	We again use the fact that Bienaym\'e trees are special cases of simply generated trees. We aim to apply Theorem \ref{thm.simplygengaussian}, so proceed to verify the assumptions of that result. The assumptions on $\mu_0$ and $\mu_1$ in particular imply that $\mu_0>0$ and that $\mu_k>0$ for some $k \ge 2$, so that requirement of the theorem is satisfied.

Recall from \eqref{eq:tau_def} that $\tau = \rho$ if $\Psi(\rho) \le 1$, and $\tau=\Psi^{-1}(1)$ if $\Psi(\rho)>1$. With $\hat{\mu}_k = \mu_k\tau^k/\Phi(\tau)$ as in Theorem~\ref{thm.janson}, we then have
\[
\hat{\mu}_1 = \frac{\mu_1\tau}{\Phi(\tau)}
= \frac{\mu_1\tau}{\sum_{k \ge 0} \mu_k\tau^k}.
\]
If $\tau \ge 1$ then the denominator is at least $\mu_1\tau + (1-\mu_1-\mu_0)\tau > \mu_1\tau+\eps\tau$, since by assumption $(1-\mu_1-\mu_0) > \eps$. This yields that 
\[
\hat{\mu}_1 < \frac{\mu_1\tau}{\mu_1\tau + \eps \tau} < \frac{(1-\eps)\tau}{(1-\eps)\tau+\eps \tau} = 1-\eps\, ,
\]
the second inequality holding since $\mu_1< 1-\eps$ and $x/(x+\eps \tau)$ is increasing in $x$. 
On the other hand, if $\tau < 1$ then 
\[
\hat{\mu}_1 = \frac{\mu_1\tau}{\Phi(\tau)} < 
\frac{\mu_1\tau}{\mu_0+\mu_1\tau} < 
\frac{\mu_1}{\mu_0+\mu_1} < 1-\eps\, ,
\]
the last inequality holding by the assumptions of the theorem. In either case we have $\hat{\mu}_1 < 1-\eps$, so the result follows by Theorem \ref{thm.simplygengaussian}.
	\end{proof}
	
\subsection{Block-graphs}\label{sec:blockgraphs}
Theorem \ref{thm.simplygengaussian} implies tail-bounds on the block-diameter of a random graph from a block-stable family and, in particular, we resolve a conjecture from \cite{McDiarmidScott}. We will now introduce these families of random graphs. 

A {\em bridge} in a graph is an edge whose deletion increases the number of connected components. A \emph{block} in a graph is either a maximal $2$-connected subgraph or the subgraph formed by a bridge or an isolated vertex. Call a class $\mathcal{A}$ of graphs \emph{block-stable} when a graph $G$ is in the class if and only if each block of $G$ is in the class. The \emph{block-diameter} of a graph $G$, denoted by $\operatorname{Diam}(G)$, is the greatest number of blocks that any path in $G$ passes through. Let $\cA_n$ denote the subset of $\cA$ that contains all graphs with vertex set  $[n]$ and let $\cA^c_n$ denote the subset of connected graphs in $\cA_n$. Theorem 1.2 in \cite{McDiarmidScott} implies that for any block-stable class $\cA$, for $\mathrm{G}^c_n\in_u \cA^c_n$, we have that $\operatorname{Diam}(\mathrm{G}^c_n)=O(\sqrt{n\log n})$ with high probability (where $n$ runs over all integers for which $\cA_n$ is non-empty). Theorem \ref{thm.simplygengaussian} yields the following improvement to this result.

\begin{thm}\label{thm:blockgraphs}
For any block-stable class $\cA$ there exist constants $c_1=c_1(\cA)$ and $c_2=c_2(\cA)$ such that for any $t>1$ and any $n$ with $\cA_n\neq \emptyset$, for $\mathrm{G}^c_n\in_u \cA^c_n$,
\[
\p\left(\operatorname{Diam}(\mathrm{G}^c_n)\ge tn^{1/2}\right)< c_1\exp(-c_2 t^2).
\]
\end{thm}
Theorem \ref{thm:blockgraphs} follows from Theorem \ref{thm.simplygengaussian} by an observation by Stufler that he states as Corollary 6.45 in \cite{MR4132643}. The proof goes as follows. The coupling between $\mathrm{G}^c_n$ and a particular simply generated tree (that is given in Lemma 6.1 in the survey paper \cite{MR4132643}) preserves the diameter up to a potential additive error of $1$ (where the diameter in a tree is simply the number of edges on the longest path). Therefore, because the diameter of a tree is at most twice its height, tail bounds on $\operatorname{Diam}(\mathrm{G}_n^c)$ follow from tail bounds on the height of the simply generated tree corresponding to $\cA$, so Theorem \ref{thm.simplygengaussian} implies Theorem \ref{thm:blockgraphs} (after we increase $c_1$ to ensure that the result holds for all $n$).

We can use Theorem \ref{thm:blockgraphs} to study the block-diameter of the random (potentially disconnected) graph $\mathrm{G}_n\in_u\cA_n$ by conditioning on its component sizes and applying the result to the individual components, observing that a connected component of $\mathrm{G}_n$ of size $m$ has the same law as $\mathrm{G}^c_m$, up to relabelling of its vertices, and independently of the other components conditioned on their sizes. This yields the following theorem.
\begin{thm}\label{thm:discon_blockgraphs}
For any block-stable class  $\cA$, there exist constants $c_1,c_2>0$ such that for any $t>1$ and any $n$ with $\cA_n\neq \emptyset$, for $\mathrm{G}_n\in_u \cA_n$,
\[
\p\left(\operatorname{Diam}(\mathrm{G}_n)\ge tn^{1/2}\right)< c_1\exp(-c_2 t^2).
\]
\end{thm}
Theorem~\ref{thm:discon_blockgraphs} in particular implies that for any sequence $(a_n,n \ge 1)$ with $a_n\to \infty$, it holds that $\operatorname{Diam}(\mathrm{G}_n)=o(a_n\sqrt{n})$ with high probability; this resolves the conjecture stated just before Theorem 1.2 in \cite{McDiarmidScott}.
\begin{proof}[Proof of Theorem~\ref{thm:discon_blockgraphs}]
We let $c_1=c_1(\cA)$ and $c_2=c_2(\cA)$ be the constants from Theorem~\ref{thm:blockgraphs}. We may assume $c_1 > e^2$ since increasing $c_1$ only weakens the claim.
Therefore, we only need to prove the theorem for $\exp(-c_2 t^2)<e^{-2}$, because otherwise the claimed upper bound exceeds $1$, which is an upper bound for any probability. 

Let $N_1,\dots,N_M$ be the sizes of the connected components of $\mathrm{G}_n$, and observe that
\begin{align*}
\p\left(\operatorname{Diam}(\mathrm{G}_n)\ge tn^{1/2}\right)&=\e\left[\p\left(\operatorname{Diam}(\mathrm{G}_n)\ge tn^{1/2}\mid N_1,\dots,N_M)\right] \right).
\end{align*}
Moreover, for any $m\geq 1$, for any $n_1,\dots,n_m \in \N$ for which $n_1+\dots+ n_m=n$ and $\cA^c_{n_1}, \dots, \cA^c_{n_m} \neq \emptyset$, we see that, for $\mathrm{G}^c_{n_1}, \dots, \mathrm{G}^c_{n_m}$ independent uniformly random samples from $\cA^c_{n_1},\dots, \cA^c_{n_m}$ respectively,
\begin{align*}
& \p\left(\operatorname{Diam}(\mathrm{G}_n)\ge tn^{1/2}\mid M=m, N_1=n_1,\dots,N_m=n_m\right)\\
&=\p\left(\max_{i\in[m]} \operatorname{Diam}(\mathrm{G}^c_{n_i})\ge tn^{1/2}\right)\\
&\leq \sum_{i=1}^m \p\left( \operatorname{Diam}(\mathrm{G}^c_{n_i})\ge (tn^{1/2}n_i^{-1/2}){n_i}^{1/2}\right)\\
&\leq \sum_{i=1}^m \exp\left(-c_2 n n_i^{-1} t^2\right),
\end{align*}
where the last inequality follows from Theorem \ref{thm:blockgraphs}. 
The statement now follows if we show that for $\exp(-c_2 t^2)<e^{-2}$, for all $m\geq 2$ and all $p_1,\dots, p_m\in (0,1)$ for which $p_1+\dots+p_m=1$, it holds that 
\[\sum_{i=1}^m \exp\left(-c_2 p_i^{-1} t^2\right)\leq \exp\left(-c_2 t^2\right).\]
We prove this last statement by induction on $m$. For any $a<e^{-2}$, for $p\in (0,1)$, define $f_a(p)=a^{1/p}+a^{1/(1-p)}$ and for $p\in \{0,1\}$, set $f_a(p)=a$, so that $f_a$ is continuous. It is straightforward to check that $f_a$ is convex on $[0,1]$ for $a<e^{-2}$, so $f_a(p)$ is maximized at $p\in \{0,1\}$ and $f_a(p)\leq a$ for all $p\in (0,1)$. The statement for $m=2$ then follows from this upper bound on $f_a(p)$ with $a=\exp\left(-c_2 t^2\right)$. Now, fix $k\geq 2$ and suppose the statement holds for $m=k$. Then, fix $p_1,\dots, p_{k+1}\in (0,1)$ for which $p_1+\dots+p_{k+1}=1$. Now, we get that 
\[ \exp\left(-c_2 p_{k}^{-1} t^2\right)+\exp\left(-c_2 p_{k+1}^{-1} t^2\right)\leq \exp\left(-c_2 (p_{k}+p_{k+1})^{-1} t^2\right) \] 
by the upper bound on $f_a(p)$ with $a=\exp\left(-c_2 (p_{k}+p_{k+1})^{-1} t^2\right)$ and $p=p_k(p_{k}+p_{k+1})^{-1}$, using that $\exp\left(-c_2 (p_{k}+p_{k+1})^{-1} t^2\right)<\exp\left(-c_2 t^2\right)<e^{-2}$. The result now follows from the induction hypothesis applied to  $p_1,\dots, p_{k-1}, p_k+p_{k+1}$. 
\end{proof}

\section{\bf Stochastic domination results}\label{sec:stochastic}

This section presents the proof of Theorem~\ref{thm.stochasticorder2}. 
The following decomposition is a key input to the proof. 
Given a tree $\tree$, let $f(\tree)$ be the unordered set of rooted trees obtained from $\tree$ by removing all edges from vertices $1$ and $2$ to their children. Also, write $\tree^{12}$ for the tree obtained from $\tree$ by swapping the labels of vertices $1$ and $2$. Then we say that $\tree \sim \tree'$ if either $\tree$ and $\tree'$ have the same root and $f(\tree)=f(\tree')$ or $\tree^{12}$ and $\tree'$ have the same root and $f(\tree^{12})=f(\tree')$. (See Figures \ref{fig.parallel} and \ref{fig.series}) for examples.)

\begin{figure}%
	\begin{subfigure}{0.8\textwidth}
		\centering
		\includegraphics[width=\textwidth]{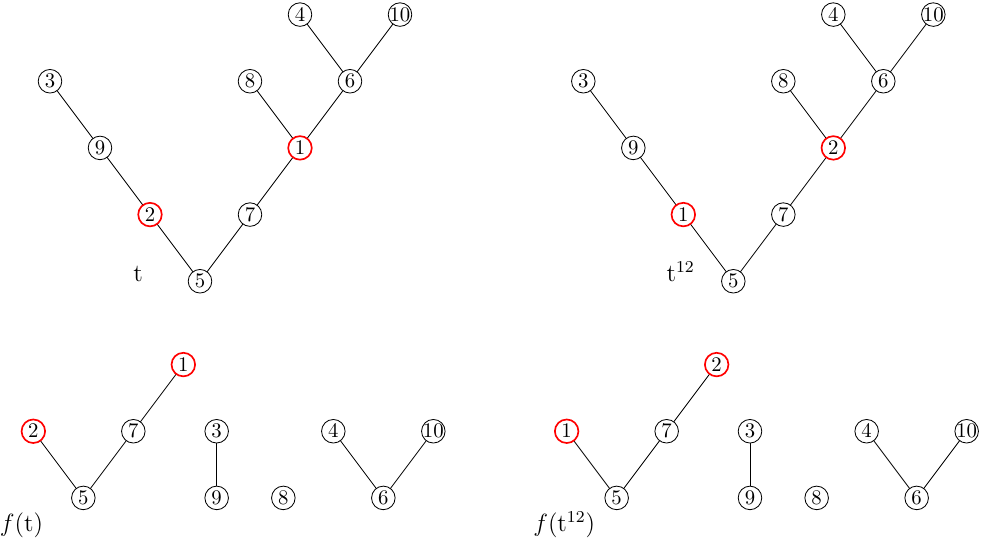}
		\caption{For a tree $\tree$, we obtain $f(\tree)$ by removing the edges to the children of vertex $1$ and $2$. We obtain $\tree^{12}$ by swapping the labels of $1$ and $2$ in $\tree$.}
		\label{fig:parallela}
	\end{subfigure}
	\begin{subfigure}{0.8\textwidth}
		\centering
		\includegraphics[width=0.95\textwidth]{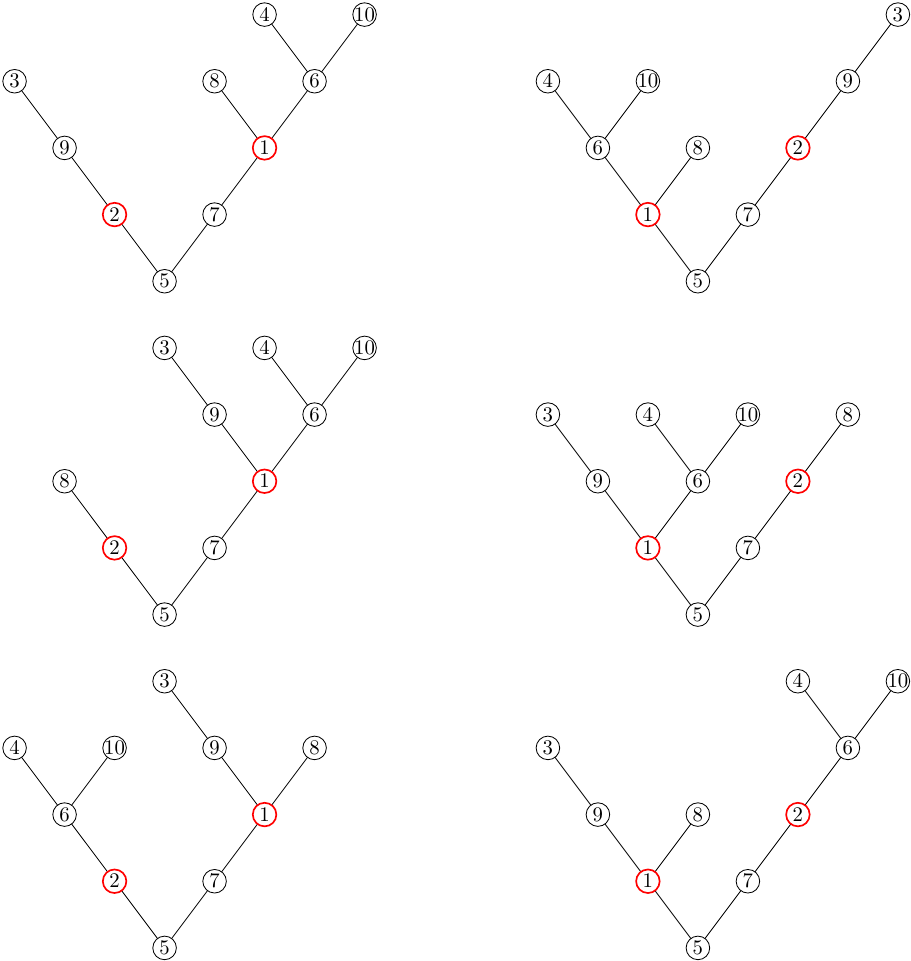}
		\caption{The trees $\tree'$ such that  $\tree'\sim\tree$ and $\deg_{\tree'}(1)=2$ and $\deg_{\tree'}(2)=1$. For each tree $\tree'$ on the left, $f(\tree')=f(\tree)$. For each tree $\tree'$ on the right, $f(\tree')=f(\tree^{12})$.}
		\label{fig:parallelb}
	\end{subfigure}
	\caption{We show a subset of the equivalence class of tree $\tree$. In total, $\tree$ is equivalent to $2^4$ trees: to specify a tree $\tree'$ such that $\tree'\sim \tree$ we must choose, for each element in $\{6,8,9\}$, whether its parent in $\tree'$ is $1$ or $2$, and we must choose whether or not to swap the labels of $1$ and $2$.}
	\label{fig.parallel}
\end{figure}

\begin{figure}%
	\begin{subfigure}{0.8\textwidth}
		\centering
		\includegraphics[width=\textwidth]{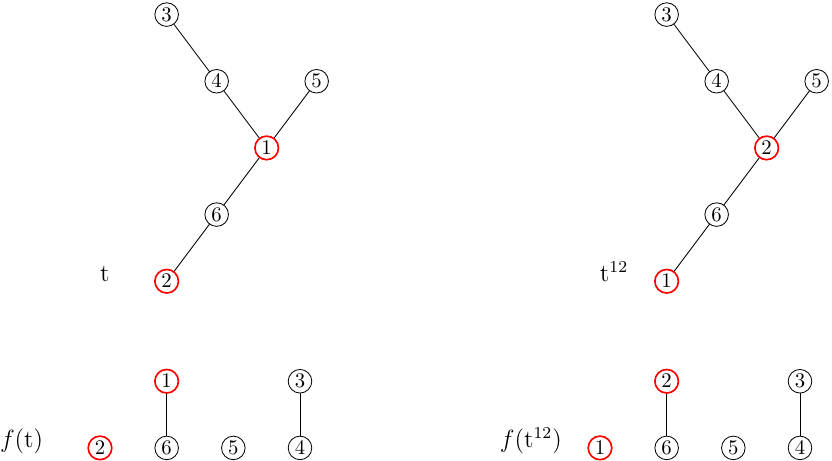}
		\caption{For a tree $\tree$, we obtain $f(\tree)$ by removing the edges to the children of vertex $1$ and $2$. We obtain $\tree^{12}$ by swapping the labels of $1$ and $2$ in $\tree$.}
		\label{fig:seriesa}
	\end{subfigure}
	\begin{subfigure}{0.8\textwidth}
		\centering
		\includegraphics[width=0.9\textwidth]{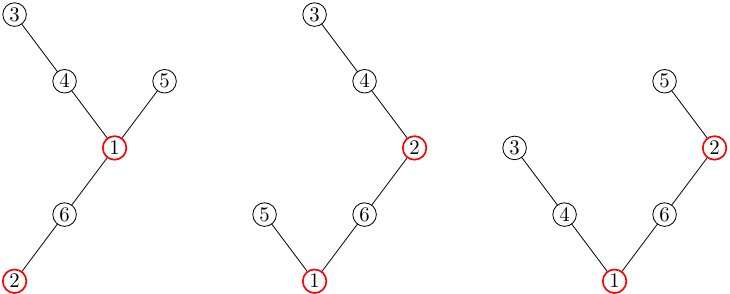}
		\caption{The trees in the equivalence class of $\tree$ in which $d_1=2$ and $d_2=1$. For the leftmost tree $\tree'$, we have that $f(\tree')=f(\tree)$ and $\tree'$ and $\tree$ have the same root. For each other tree  $\tree''$, we have that $f(\tree'')=f(\tree^{12})$ and  $\tree''$ and $\tree^{12}$ have the same root .}
		\label{fig:seriesb}
	\end{subfigure}
	\caption{
	We show a subset of the equivalence class of tree $\tree$. In total, $\tree$ is equivalent to $2^3$ trees: to specify a tree $\tree'$ such that $\tree'\sim \tree$ we must choose, for each element in $\{4,5\}$, whether its parent in $\tree'$ is $1$ or $2$, and we must choose whether or not to swap the labels of $1$ and $2$.}
	\label{fig.series}
\end{figure}

We prove Theorem~\ref{thm.stochasticorder2} using the following two propositions. 
\begin{prop}\label{prop:same_equiv_prob}
Let $\cC$ be an equivalence class for the equivalence relation $\sim$. Fix a degree sequence $\sequence{d}=(d_1,\ldots,d_n)$ with $d_2 \ge 1$ and let $\sequence{d'}=(d_1+1,d_2-1,\ldots,d_n)$. Then with $\randomtree \in_u \cT_\sequence{d}$ and $\randomtree' \in_u \cT_{\sequence{d}'}$, 
\[
\p(\randomtree \in \cC) = \p(\randomtree' \in \cC). 
\]
\end{prop}
\begin{prop}\label{prop:refined_stochastic_bd}
Fix a degree sequence $\sequence{d}=(d_1,\ldots,d_n)$ with $d_1\ge d_2 \ge 1$ and let $\sequence{d'}=(d_1+1,d_2-1,\ldots,d_n)$. 
Then for any $\sim$-equivalence class $\cC$ with $\cC\cap \cT_\sequence{d}\ne \emptyset$, letting $\randomtree_{\cC} \in_u \cT_\sequence{d}\cap \cC$ and $\randomtree_{\cC}' \in_u \cT_{\sequence{d}'}\cap \cC$, we have 
\[
\height(\randomtree_{\cC})' \preceq_{\mathrm{st}} 
\height(\randomtree_{\cC})\, .
\]
Moreover, if $\sequence{d}$ contains at least three non-zero entries then there exists at least one $\sim$-equivalence class $\cC$ for which the preceding stochastic domination is strict. 
\end{prop}
Before proving the propositions, we show how they straightforwardly imply Theorem~\ref{thm.stochasticorder2}. 
\begin{proof}[Proof of Theorem~\ref{thm.stochasticorder2}]
First, by relabelling the vertices it suffices to show that for any degree sequence $\sequence{d}=(d_1,\ldots,d_n)$ with $d_1 \ge d_2 \ge 1$, if $\sequence{d}'=(d_1+1,d_2-1,d_3,\ldots,d_n)$ then for $\randomtree \in_u \cT_{\sequence{d}}$ and $\randomtree' \in_u \cT_{\sequence{d}'}$ we have $\height(\randomtree') \preceq_{\mathrm{st}} \height(\randomtree)$, and that this stochastic domination is strict if $\sequence{d}$ contains at least three non-zero entries. 

Fix degree sequences $\sequence{d}$ and $\sequence{d'}$ related as in the previous paragraph, and let $\randomtree \in_u \cT_{\sequence{d}}$ and $\randomtree' \in_u \cT_{\sequence{d}'}$. For a $\sim$-equivalence class $\cC$ such that $\cT_{\sequence{d}}\cap\cC\neq \emptyset$ we will also use the notation $\randomtree_{\cC}$ and $\randomtree'_{\cC}$ to denote uniformly random elements of $\cT_{\sequence{d}}\cap \cC$ and $\cT_{\sequence{d}'}\cap \cC$, respectively. 

For any $x > 0$, writing $\sum_{\cC}$ to denote a summation over all $\sim$-equivalence classes $\cC$, we now have 
\begin{align*}
\p(\height(\randomtree) \le x)
&  = \sum_{\cC} \probC{\height(\randomtree) \le x}{\randomtree \in \cC}\p(\randomtree \in \cC) \\
& = \sum_{\cC} \p(\height(\randomtree_{\cC}) \le x)\p(\randomtree' \in \cC)\\
& \le \sum_{\cC} \p(\height(\randomtree_{\cC}') \le x)\p(\randomtree' \in \cC)\\
& = \sum_{\cC} \probC{\height(\randomtree') \le x}{\randomtree' \in \cC}\p(\randomtree' \in \cC)\\
& = \p(\height(\randomtree') \le x)
\end{align*}
The second equality holds by 
Proposition~\ref{prop:same_equiv_prob} and since the conditional distribution of $\randomtree$ given that $\randomtree \in \cC$ is precisely that of $\randomtree_{\cC}$. The inequality holds by Proposition~\ref{prop:refined_stochastic_bd}. The third equality again holds since the conditional distribution of $\randomtree'$ given that $\randomtree' \in \cC$ is precisely that of $\randomtree'_{\cC}$. This shows that $\height(\randomtree')\preceq_{\mathrm{st}}\height(\randomtree)$. 

Finally, if $\sequence{d}$ has at least three non-zero entries then by Proposition~\ref{prop:refined_stochastic_bd} there exists at least one equivalence class $\cC$ and some $x > 0$ for which $\p(\height(\randomtree_{\cC}) < x) < \p(\height(\randomtree'_{\cC}) < x)$. For such $x$ the above chain of inequalities then yields that $\p(\height(\randomtree) \le x) < \p(\height(\randomtree') \le x)$, so in this case in fact $\height(\randomtree')\prec_{\mathrm{st}}\height(\randomtree)$. 
\end{proof}

\begin{proof}[Proof of Proposition~\ref{prop:same_equiv_prob}]
Let $\tree$ be a tree in $\cC \cap \cT_{\sequence{d}}$.  
We first suppose that neither vertex $1$ nor vertex $2$ is an ancestor of the other in $\tree$.

The forest $f(\tree)$ contains $d_1+d_2+1$ trees; list their roots as $r_1,\ldots,r_{d_1+d_2+1}$ so that $r_{d_1+d_2+1}$ is the root of $\tree$. Then both $1$ and $2$ lie in the tree rooted at $r_{d_1+d_2+1}$. In this case, there are ${d_1+d_2 \choose d_1}$ trees $\hat{\tree} \in\cC \cap \cT_{\sequence{d}}$ with $f(\hat{\tree})=f(\tree)$: these are precisely the trees obtained from $f(\tree)$ as follows. 
\begin{itemize}
\item Select a set $S \subset [d_1+d_2]$ of size $d_1$.
\item Add edges from vertex $1$ to the vertices in the set $\{r_i,i \in S\}$, and add edges from vertex $2$ to the vertices in the set $\{r_i,i \in [d_1+d_2]\setminus S\}$. 
\end{itemize}
Likewise, there are ${d_1+d_2 \choose d_2}$ trees $\hat{\tree} \in \cC \cap \cT_{\sequence{d}}$ with $f(\hat{\tree})=f(\tree^{12})$; these are obtained from $f(\tree)$ as follows.
\begin{itemize}
\item Select a set $S \subset [d_1+d_2]$ of size $d_2$.
\item Add edges from vertex $1$ to the vertices in the set $\{r_i,i \in S\}$, add edges from vertex $2$ to the vertices in the set $\{r_i,i \in [d_1+d_2]\setminus S\}$, then swap the labels of vertices $1$ and $2$. 
\end{itemize}
It follows that $|\cC \cap \cT_d| = 2{d_1+d_2 \choose d_1}$. Likewise, we have $|\cC \cap \cT_{\sequence{d}'}| = 2{d_1+d_2 \choose d_1+1}$, since any element of $\cC \cap \cT_{\sequence{d}'}$ may be constructed by selecting a size-$(d_1+1)$ subset $S'$ of $\{r_i, i \in[d_1+d_2]\}$, then either (a) attaching the roots in $S'$ to $1$ and the remaining roots to $2$, or (b) attaching the roots in $S'$ to $2$ and the remaining roots to $1$ and switching the labels of vertices $1$ and $2$. 

Recalling the formula (\ref{eq:td_size}) for $|\cT_{\sequence{d}}|$, 
the preceding computations yield that 
\small
\[
\frac{|\cC\cap \cT_{\sequence{d}} |}{|\cT_{\sequence{d}} |} = 2{d_1+d_2 \choose d_1}\!{n-1 \choose d_1,\ldots,d_n}^{-1} 
\!=2{d_1+d_2 \choose d_1+1}\!{n-1 \choose d_1+1,d_2-1,d_3,\ldots,d_n}^{-1} 
\!=\frac{|\cC\cap\cT_{\sequence{d}'}|}{|\cT_{\sequence{d}'}|}\, ,
\]
\normalsize
so in this case $\p(\randomtree \in \cC)=\p(\randomtree' \in \cC)$, as required. 

Next suppose that vertex $1$ is an ancestor of vertex $2$ in $\tree$.  
The forest $f(\tree)$ contains $d_1+d_2+1$ trees; list their roots as $r_1,\ldots,r_{d_1+d_2+1}$ so that $r_{d_1+d_2}$ and $r_{d_1+d_2+1}$ are the roots of the trees containing vertices $2$ and $1$, respectively. (This also means that $r_{d_1+d_2+1}$ is the root of $\tree$, and that $r_{d_1+d_2}$ is a child of $1$ in $\tree$.) 
In this case there are ${d_1+d_2-1 \choose d_2}$ trees $\hat{\tree} \in \cC \cap \cT_{\sequence{d}}$ with $f(\hat{\tree})=f(\tree)$: these are precisely the trees obtained from $f(\tree)$ as follows. 
\begin{itemize}
\item Select a set $S \subset [d_1+d_2-1]$ of size $d_2$. 
\item Add edges from vertex $2$ to the vertices in the set $\{r_i,i \in S\}$, and add edges from vertex $1$ to the vertices in the set $\{r_i,i \in [d_1+d_2]\setminus S\}$. 
\end{itemize}
Similarly, there are ${d_1+d_2-1 \choose d_1}$ trees $\hat{\tree} \in \cC \cap \cT_{\sequence{d}}$ with $f(\hat{\tree})=f(\tree^{12})$; these are obtained from $f(\tree)$ as follows. 
\begin{itemize}
\item Select a set $S \subset [d_1+d_2-1]$ of size $d_1$.
\item Add edges from vertex $2$ to the vertices in the set $\{r_i,i \in S\}$, add edges from vertex $1$ to the vertices in the set $\{r_i, i \in[d_1+d_2]\setminus S\}$, then swap the labels of vertices $1$ and $2$. 
\end{itemize}
It follows that 
\[
|\cC \cap \cT_{\sequence{d}}| 
= 
{d_1+d_2-1 \choose d_2}
+{d_1+d_2-1 \choose d_1}
= {d_1+d_2 \choose d_1}\, ,
\]
and likewise $|\cC \cap \cT_{\sequence{d}'}| = {d_1+d_2 \choose d_1+1}$. (We omit the details for this last identity as they are so similar to the previous arguments.) It follows that in this case we also have $|\cC \cap \cT_{\sequence{d}}|/|\cT_{\sequence{d}}|=|\cC \cap \cT_{\sequence{d}'}|/|\cT_{\sequence{d}'}|$, so again $\p(\randomtree \in \cC)=\p(\randomtree' \in \cC)$. 

Finally, if vertex $2$ is an ancestor of vertex $1$ in $\tree$, then in $\tree^{12}$ vertex $1$ is an ancestor of vertex $2$, so since $\tree \sim \tree^{12}$, this situation is already handled by the previous case. 
\end{proof}
For the proof of Proposition~\ref{prop:refined_stochastic_bd} we require an additional lemma, which although fairly straightforward we find independently pleasing. Write ${[n] \choose k}=\{S \subset [n]: |S|=k\}$. Below we use the convention that $\max \emptyset = 0$. 
\begin{lem}[Eggs-in-one-basket lemma]\label{lem:ei}
Fix non-negative real numbers $0 < a_1 \le \ldots \le a_n$ and integers $k,\ell$ with  $n/2 \le k < \ell \le n$. 
\begin{enumerate}
\item Let $\rA \in_u {[n] \choose k} \cup {[n] \choose n-k}$ and $\rA' \in_u {[n] \choose \ell} \cup {[n] \choose n-\ell}$. Then $\max(a_i,i \in \rA') \preceq_{\mathrm{st}}\max(a_i,i \in \rA)$.
\item Let $\rB \in_u {[n-1] \choose k} \cup {[n-1] \choose n-k}$ and $\rB' \in_u {[n-1] \choose \ell} \cup {[n-1] \choose n-\ell}$. Then $\max(a_i,i \in \rB') \preceq_{\mathrm{st}} \max(a_i,i \in \rB)$.
\end{enumerate}
\end{lem}
The proverb ``don't put all your eggs in one basket'' means ``don't put all your resources into a single endeavour'' or, more pithily, ``diversify your portfolio''. (Its origins are obscure but an Italian equivalent, ``non mettere tutte le uova in un solo cesto'', has been traced to at least 1666 \cite{uova}.) To understand our use of this phrase, note that if the ``portfolio'' is the random set $\rA$ or the set $\rB$ from the lemma, and the payoff of a portfolio is the value of its largest element, then the lemma implies that larger-entropy portfolios have stochastically higher payoffs. To our knowledge, this lemma is the first rigorous proof of the wisdom of the proverb. 
\begin{proof}[Proof of Lemma~\ref{lem:ei}]
If $a_1,\ldots,a_\ell$ are all equal then the result is obvious so we hereafter assume that this is not the case. 
It suffices to prove the lemma with $\ell=k+1$; the general case follows by induction. It is useful to set $a_0=0$. Then for any $0 \le i < n$ and $a_i < x \le a_{i+1}$, 
		\begin{align*}\p\left(\max\{a_j:j\in \mathrm{A}\}<x  \right) &=\p\left(\mathrm{A}\cap\{i+1,\dots,n\}=\emptyset \right)\\
			&=\frac{1}{2\binom{n}{k}}\left[ \binom{i}{k}+\binom{i}{n-k}\right].
\end{align*}
To prove the first claim of the lemma, that $\max(a_i,i \in \rA') \preceq_{\mathrm{st}}\max(a_i,i \in \rA)$, it thus suffices to show that 
		\[ 
		\frac{1}{2\binom{n}{k}}\left[ \binom{i}{k}+\binom{i}{n-k}\right]\le  \frac{1}{2\binom{n}{k+1}}\left[ \binom{i}{k+1}+\binom{i}{n-k-1}\right] .
		\]
	It is possible that some of the binomial coefficients above are zero; regardless, 
multiplying through by $2\binom{n}{n-i}$ and rearranging terms yields that this is equivalent to showing that
 \[ \binom{n-k}{n-i}-\binom{n-k-1}{n-i}\le \binom{k+1}{n-i}-\binom{k}{n-i},\]
 which by the addition rule of binomials reduces to
\[ \binom{n-k-1}{n-i-1}\le \binom{k}{n-i-1}.\]
This holds, because $k>n-k-1$. 

For the second claim of the lemma, 
note that $\mathrm{B}$ has the law of $\mathrm{A}$ conditional on the event $\{n \not\in  \mathrm{A}\}$, which, by symmetry, has probability $1/2$. This implies that for any $x\leq a_n$, 
\begin{align*}
\p\left( \max\{a_i:i\in \mathrm{A}\}<x\right)&=\p\left(n\not\in \mathrm{A}\right)\p\left(\left.  \max\{a_i:i\in \mathrm{A}\}<x\right| n \not\in  \mathrm{A}\right)\\
&= \frac{1}{2}\p\left( \max\{a_i:i\in \mathrm{B}\}<x\right),
\end{align*}
 and similarly, 
 \[\p\left( \max\{a_i:i\in \mathrm{A}'\}<x\right)=\frac{1}{2}\p\left(\max\{a_i:i\in \mathrm{B}'\}<x\right).\]
 Therefore,
 \begin{align*}
&  \p\left( \max\{a_i:i\in \mathrm{A}\}<x\right)-\p\left( \max\{a_i:i\in \mathrm{A}'\}<x\right)\\
&  =\frac12\left(\p\left( \max\{a_i:i\in \mathrm{B}\}<x\right)-\p\left( \max\{a_i:i\in \mathrm{B}'\}<x\right)\right)\, ,
 \end{align*}
so the second claim of the lemma follows from the first. 
\end{proof}
\begin{proof}[Proof of Proposition~\ref{prop:refined_stochastic_bd}]
Fix a $\sim$-equivalence class $\cC$ and a tree $\tree \in \cC \cap \cT_\sequence{d}$. We first suppose that neither vertex $1$ nor vertex $2$ is an ancestor of the other in $\tree$. The forest $f(\tree)$ contains $d_1+d_2+1$ trees; list their roots as $r_1,\ldots,r_{d_1+d_2+1}$ so that $r_{d_1+d_2+1}$ is the root of $\tree$, and let $\tree_i$ be the tree with root $r_i$. 
Then both $1$ and $2$ lie in the tree $\tree_{d_1+d_2+1}$ rooted at $r_{d_1+d_2+1}$. Write $h_1$ and $h_2$ for the distance from $r_{d_1+d_2+1}$ to $1$ and $2$, respectively. 

By the definition of $\cC$, 
starting from $f(\tree)$ we may sample $\randomtree_{\cC} \in_u \cC \cap \cT_\sequence{d}$ as follows. 
\begin{itemize}
\item Let $(\rC,\rA) \in_u \big\{(1,S): S \in {[d_1+d_2] \choose d_1}\big\} \cup \big\{(2,S): S \in {[d_1+d_2] \choose d_2}\big\}$. 
\item Add edges from vertex $1$ to the roots $\{r_i,i \in \rA\}$ and from vertex $2$ to the roots $\{r_i,i \in [d_1+d_2]\setminus A\}$. 
\item If $\rC=2$ then swap the labels of vertices $1$ and $2$. 
\end{itemize}
Note that $\rA \in_u {[d_1+d_2]\choose d_1} \cup {[d_1+d_2] \choose d_2}$. 
For $1 \le i \le d_1+d_2$ letting $a_i=1+\height(\tree_i)$, with the above construction of $\randomtree_{\cC}$ we then have 
\[
\height(\randomtree_{\cC})
=\max(\height(\tree_{d_1+d_2+1}),h_1+\max(a_i,i \in \rA),h_2+\max(a_i,i \in [d_1+d_2]\setminus \rA))\, .
\]

Next, again starting from $f(\tree)$, apply the same procedure (with $d_1$ and $d_2$ replaced by $d_1+1$ and $d_2-1$, respectively) to sample $\randomtree_{\cC}' \in_u \cC \cap \cT_\sequence{d}$. We obtain 
\[
\height(\randomtree_{\cC}')
=\max(\height(\tree_{d_1+d_2+1}),h_1+\max(a_i,i \in \rA'),h_2+\max(a_i,i \in [d_1+d_2]\setminus \rA'))\, .
\]
where $\rA' \in_u {[d_1+d_2]\choose d_1+1} \cup {[d_1+d_2] \choose d_2-1}$.

Since $\rA$ and $[d_1+d_2]\setminus \rA$ have the same distribution, as do $\rA'$ and $[d_1+d_2]\setminus \rA'$, we may assume without loss of generality that $h_1 \ge h_2$. It then follows that both the above maxima are at least $M^-:= \max(\height(\tree_{d_1+d_2+1}),h_2+\max(a_i,i \in [d_1+d_2]))$ and at most $M^+:= \max(\height(\tree_{d_1+d_2+1}),h_1+\max(a_i,i \in [d_1+d_2]))$
, so for $x \le M^-$ we have 
\[
\p(\height(\randomtree_{\cC}) < x)=0=\p(\height(\randomtree_{\cC}') < x)
\]
while for $x > M^+$ we have 
\[
\p(\height(\randomtree_{\cC}) < x)
=1=\p(\height(\randomtree_{\cC}') < x).
\]
For $M^- < x \le M^+$, we have 
\[
\p(\height(\randomtree_{\cC})< x)
=
\p(h_1+\max(a_i,i \in \rA) < x)
\]
and 
\[
\p(\height(\randomtree_{\cC}')< x)
=
\p(h_1+\max(a_i,i \in \rA') < x)\, ,
\]
so the first part of the eggs-in-one-basket lemma yields that 
\[
\p(\height(\randomtree_{\cC})< x)
\le 
\p(\height(\randomtree_{\cC}')< x). 
\]
This establishes that $\height(\randomtree_{\cC}') \preceq_{\mathrm{st}} \height(\randomtree_{\cC})$ when neither $1$ nor $2$ is an ancestor of the other for trees in $\cC$. 

We next suppose that either $1$ is an ancestor of $2$ in $\tree$ or vice-versa. Note that $\cC \cap \cT_\sequence{d}$ contains a tree in which $1$ is an ancestor of $2$ if and only if it contains a tree in which $2$ is an ancestor of $1$. It follows that, by replacing $\tree$ by another element of $\cC \cap \cT_\sequence{d}$ if necessary, we may assume that in fact $1$ is an ancestor of $2$ in $\tree$. 
 
List the roots of the trees in $f(\tree)$ as $r_1,\ldots,r_{d_1+d_2+1}$ so that $r_{d_1+d_2}$ and $r_{d_1+d_2+1}$ are the roots of the trees containing vertices $2$ and $1$, respectively, and write $\tree_i$ for the tree of $f(\tree)$ with root $r_i$. Necessarily $r_{d_1+d_2+1}$ is also the root of $\tree$, and $r_{d_1+d_2}$ is a child of $1$ in $\tree$. Write $h_1$ and $h_2$ for the distance from $r_{d_1+d_2+1}$ to $1$ and from $r_{d_1+d_2}$ to $2$, respectively. 

By the definition of the equivalence class $\cC$, 
starting from the forest $\tree_1,\ldots,\tree_{d_1+d_2+1}$, we may sample $\randomtree_{\cC} \in_u \cC \cap \cT_\sequence{d}$ as follows.
\begin{itemize}
\item Let 
\[
(\rC,\rB) \in_u \Big\{(1,S):S \in {[d_1+d_2-1] \choose d_2}\Big\} \cup \Big\{(2,S): S \in {[d_1+d_2-1] \choose d_1}: \Big\}\,. 
\]
\item Add edges from vertex $2$ to the roots $\{r_i,i \in \rB\}$ and from vertex $1$ to the roots $\{r_i,i \in [d_1+d_2]\setminus \rB\}$. 
\item If $\rC=2$ then swap the labels of vertices $1$ and $2$. 
\end{itemize}

Note that $\rB \in_u \{S \subset [d_1+d_2-1]: |S| \in \{d_1,d_2\}\}$. 
Moreover, letting $a_i=1+\height(\tree_i)$ for $ i \in [d_1+d_2-1]$, and letting 
$H=\max(\height(\tree_{d_1+d_2+1}),h_1+1+\height(\tree_{d_1+d_2}))$, 
with the above construction of $\randomtree_{\cC}$ we then have 
\[
\height(\randomtree_{\cC})
=\max(H,
h_1+\max(a_i,i \in [d_1+d_2-1]\setminus\rB),h_1+1+h_2+\max(a_i,i \in \rB))\, .
\]
The term $H$ accounts for the possibility that the height of $\randomtree_\cC$ is achieved by a vertex of either $\tree_{d_1+d_2}$ or $\tree_{d_1+d_2+1}$. 

Next, again starting from $f(\tree)$, apply the same procedure (with $d_1$ and $d_2$ replaced by $d_1+1$ and $d_2-1$, respectively) to sample $\randomtree_{\cC}' \in_u \cC \cap \cT_\sequence{d}$. We obtain 
\[
\height(\randomtree_{\cC})
=\max(H,h_1+\max(a_i,i \in [d_1+d_2-1]\setminus \rB'),h_1+1+h_2+\max(a_i,i \in \rB'))\, .
\]
where $\rB' \in_u \{S \subset [d_1+d_2-1]: |S| \in \{d_1+1,d_2-1\}\}$.

The heights of $\randomtree_\cC$ and $\randomtree'_\cC$ both lie between 
\[
M^- := \max(H, h_1+\max(a_i,i \in [d_1+d_2-1]))
\]
and 
\[
M^+ := \max(H, h_1+1+h_2+\max(a_i,i \in [d_1+d_2-1]))
\]
so for $x \not\in (M^-,M^+]$ we have $\p(\height(\randomtree_C) < x) = \p(\height(\randomtree'_C) < x)$. 
For $M^- < x \le M^+$, 
we have $\height(\randomtree_C) < x$ if and only if 
$h_1+1+h_2+\max(a_i,i \in \rB) < x$, and likewise $\height(\randomtree_C') < x$ if and only if 
$h_1+1+h_2+\max(a_i,i \in \rB')<x$. It thus follows by the second part of the eggs-in-one-basket lemma that 
\[
\p(\height(\randomtree_\cC) < x) \le 
\p(\height(\randomtree_\cC') < x). 
\]
This establishes that $\height(\randomtree_{\cC}') \preceq_{\mathrm{st}} \height(\randomtree_{\cC})$ when $1$ is an ancestor of $2$ in $\tree$. (As already noted, this also handles the case where $2$ is an ancestor of $1$.) 

It remains to establish strict stochastic inequality when there are at least three non-leaf vertices. We accomplish this by showing that in this case there exists $\tree \in \cT_\sequence{d}$ such that for $\cC$ the $\sim$-equivalence class of $\tree$, if $\randomtree_{\cC} \in_u \cC \cap \cT_\sequence{d}$ and $\randomtree_{\cC}' \in_u \cT_\sequence{d'}$ then $\height(\randomtree_\cC') \prec_{\mathrm{st}} \height(\randomtree_\cC)$. 

By relabelling we may assume that vertex $3$ is not a leaf. (We also still assume that $d_1 \ge d_2 \ge 1$.) Consider a tree $\tree \in \cT_\sequence{d}$ with root $1$, such that $3$ is a child of $2$ and $2$ is a child of $1$, and such that all other children of vertices $1$ and $2$ are leaves (see Figure~\ref{fig:stochdom}). Then the $\sim$-equivalence class $\cC$ of $\tree$ contains ${d_1+d_2-1 \choose d_1}$ trees with root $2$ and ${d_1+d_2-1 \choose d_2}$ trees with root $1$, so ${d_1+d_2 \choose d_1}$ trees in total. 
\begin{figure}[hbt]
\begin{centering}
\includegraphics[width=0.2\textwidth]{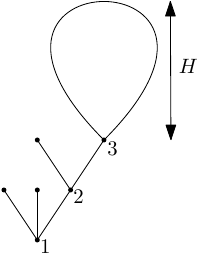}
\end{centering}
\caption{A schematic depiction of a tree in an equivalence class $\cC$ for which, if $\randomtree_{\cC} \in_u \cC \cap \cT_\sequence{d}$ and $\randomtree_{\cC}' \in_u \cT_\sequence{d'}$ then $\height(\randomtree_\cC') \prec_{\mathrm{st}} \height(\randomtree_\cC)$.}
\label{fig:stochdom}
\end{figure}

Let $H$ be the height of the subtree of $\tree$ rooted at vertex $3$ (this is at least $1$ by the assumption that $3$ is not a leaf). Then the height of $\randomtree_\cC \in_u \cC \cap \cT_\sequence{d}$ is either $H+1$ or $H+2$, and is $H+2$ precisely if either $1$ is the root and $3$ is a child of $2$, or if $2$ is the root and $3$ is a child of $1$. The total number of trees in $\cC \cap \cT_\sequence{d}$ with height $H+2$ is thus $2{d_1+d_2-2 \choose d_1-1}$, so 
\[
\p(\height(\randomtree_\cC)=H+2) = 
2{d_1+d_2-2 \choose d_1-1}{d_1+d_2\choose d_1}^{-1} = 
\frac{2d_1d_2}{(d_1+d_2-1)(d_1+d_2)}. 
\]
This probability decreases if $d_1$ and $d_2$ are replaced by $d_1+1$ and $d_2-1$, respectively, which establishes the required strict stochastic domination. 
\end{proof}

\section{\bf Future directions}\label{sec:conclusion}
Although part of our point in this work is to show that tail bounds for the height do not rely on being in a setting where there is convergence to a limiting tree (or limiting space), our results and techniques can be useful for {\em proving} such convergence. For example, note that for $k\in \N$ and $\sequence{d}$ a degree sequence corresponding to a tree with at least $k$ leaves, for  $\sequence{v}\in \mathcal{S}_\sequence{d}$, the subtree of $\tree(\sequence{v})$ spanned by the $k$ smallest labeled leaves is encoded by the first $j(k)$ elements of $\sequence{v}$, where $j(k)$ is the index of the $k$'th repeat in $\sequence{v}$. Since the law of $\randomtree_{\sequence{d}}\in_u \cT_{\sequence{d}}$ is invariant under  relabelling of the leaves by an independent uniform permutation, our sampling technique thus gives direct access to the law of the subtree spanned by a uniformly random set of $k$ leaves, also called the \emph{$k$-dimensional distribution} of the tree. Furthermore, our approach to bounding the height  in Theorem \ref{thm.gaussiantails} in fact yields a stronger result: we show that, when the tree is grown sequentially according to the bijection (as illustrated in Figure \ref{fig:sequence_ex}), \emph{all} vertices in the tree are close (on the scale of $\Theta(\sqrt{n})$) to the tree built in the first $\Theta(\sqrt{n})$ steps. The counterpart of this fact and convergence of the $k$-dimensional distributions are the two ingredients of Aldous' proof that a uniformly random tree $\randomtree_n \in_u \cT(n)$ 
converges in distribution to the Brownian continuum random tree after rescaling \cite[Theorem 8]{aldousContinuumRandomTree1991}. Therefore, our techniques give immediate access to comparable results for suitable sequences of trees with a given degree sequence, simply generated trees and conditioned Bienaym\'e trees. This idea is exploited for various laws on trees with a fixed degree sequence by Arthur Blanc-Renaudie in \cite{blanc}, to prove rescaled convergence of such trees toward inhomogeneous continuum random trees. 

Our approach is also useful for the study of other random tree models. Indeed, via the bijection of Section~\ref{sec:bij}, any distribution on labeled rooted trees gives rise to a distribution on sequences. If one can understand the law of the first $k$ repeated entries in such random sequences, one thereby obtains the law of the subtree spanned by the $k$ smallest labeled leaves. 
Our arguments for controlling the height can then in principle be combined with this spanning subtree information in order to prove both global height bounds and rescaled convergence in distribution (when the trees are viewed as measured metric spaces). On the other hand, by considering any natural distribution on sequences, the bijection yields a random tree model. Here is one natural example: consider a multiset $\cM$ of elements of $[n]$ with size (counted with multiplicity) at least $n-1$. Then construct a sequence of length $n-1$ by uniform sampling without replacement from $\cM$, and consider the tree $\randomtree$ on $[n]$ encoded by this sequence via the bijection.  (If $\cM$ itself has $n-1$ elements then $\randomtree$ has the law of a uniform tree with a given degree sequence, where for $i\in [n]$, the multiplicity of $i$ in $\cM$ is the degree of $i$ in $\randomtree$. More generally, the number of copies of $i$ in the subsample is the degree of  $i$ in $\randomtree$, and $\randomtree$ is a uniform tree with this degree sequence.) It turns out that $\randomtree$ has the law of a uniform spanning tree of a random \emph{tree-rooted graph} with degree sequence defined by $\cM$. Tree-rooted graphs are the non-planar analogues of tree-rooted maps \cite{MR314687,MR205882,MR3366469,MR2285813}, whose random instantiations are active objects of study in the planar probability and statistical physics community \cite{li2017,gwynne2017,MR3778352,MR3861296,MR4010949,MR4126936}.
In \cite{tree-weightedgraphs}, a different sampling method is used to show that in the finite variance regime, the spanning trees of large random tree-rooted graphs converge after rescaling to the Brownian continuum random tree.  In future work, we plan to build on the current work to study distances in and convergence of random tree-rooted graphs for other degree regimes.

Our results can also be applied to many models of random graphs that contain cycles, because the tree models that we study are important building blocks for many such sparse random graphs.  One family of examples is provided by the {\em configuration model}, which is used to sample graphs with a given degree sequence. In a number of recent works \cite{bhamidiUniversalityCriticalHeavytailed2020,conchon--kerjanStableGraphMetric2021,bhamidiGlobalLowerMassbound2020}, it is shown that, under conditions which ensure that the resulting random graphs are in some sense ``critical'', the large components in the configuration model converge in distribution to random compact measured metric spaces, after rescaling.  In all aforementioned papers, the techniques rely on studying spanning trees of the large components, and connecting these to the models for random trees that we study. 
However, none of those works achieve control over inter-vertex distances in smaller components. The reason for this is precisely that these components do not necessarily have scaling limits (in particular because some of them will have ``atypical'' degree distributions), and so previous results bounding the diameters of random trees do not apply.
This lack of control prohibits the authors from proving convergence of the ordered sequence of components in any topology stronger than the product topology. As a consequence, important statistics such as the diameter and the length of the longest path cannot be shown to converge under rescaling. 
We believe that, with the results and techniques of the current work in hand, it is now possible to bound distances in all components, thereby proving convergence in a stronger topology (and deducing convergence in distribution for the rescaled diameter), in all the above works. 

Finally, we believe that our results can be used to obtain non-asymptotic tail bounds on distances in uniform connected graphs with a given degree sequence and fixed surplus, which occur as the components of the configuration model. Such graphs can be related to trees with a fixed degree sequence by uniform cycle breaking, and twice the height of the resulting tree is an upper bound for the diameter in the graph. The law on trees with a given degree sequence induced by this cycle-breaking procedure has a non-trivial bias relative to the uniform distribution , whose form we believe is predicted by analogous results for components in the Erd\H{o}s--R\'enyi random graph \cite{criticalrandomgraphs}. By studying the bias it should be possible to translate our results to this set-up. 

\begin{acks}[Acknowledgments]
The first author acknowledges the financial support of NSERC during the preparation of this work. 
The second author would like to acknowledge the financial support of Centre de recherches math\'ematiques and Institut des sciences math\'ematiques; and the CogniGron research center
and the Ubbo Emmius Funds (Univ.\ of Groningen).

The authors both thank Arthur Blanc-Renaudie for useful discussions, in particular regarding the stochastic domination results, and Benedikt Stufler, for pointing out the application of our results to block graphs. 
We would also like to thank an associate editor and two anonymous referees for their very careful reading, which substantially improved the paper.
\end{acks}

\bibliographystyle{imsart-number} 
\bibliography{height.bib}       

\end{document}